\documentclass{amsart}

\usepackage{amsfonts, amsmath,amscd, amssymb, latexsym, mathrsfs, stmaryrd, verbatim, wasysym }

\usepackage{amssymb}

\usepackage{hyperref}
\usepackage[dvips]{graphicx}
\usepackage{epsfig}
\usepackage{psfrag}

\usepackage[colorinlistoftodos]{todonotes}
\usepackage[english]{babel}
\usepackage[utf8x]{inputenc}
\usepackage[all]{xy}
\usepackage{amscd}

\usepackage{appendix}
\usepackage{amsthm}

\usepackage{appendix}
\usepackage{mathtools}

\usepackage[dvips]{graphicx}
\usepackage{epsfig}
\usepackage{psfrag}
\usepackage{chngcntr}
\usepackage{subfigure}
\usepackage{captcont}
\usepackage{dsfont}

\usepackage{color}
\definecolor{darkgreen}{cmyk}{1,0,1,.2}
\definecolor{m}{rgb}{1,0.1,1}
\definecolor{green}{cmyk}{1,0,1,0}
\definecolor{darkred}{rgb}{0.55, 0.0, 0.0}
\definecolor{test}{rgb}{1,0,0}
\definecolor{cmyk}{cmyk}{0,1,1,0}

\newcounter{diagram}
\numberwithin{diagram}{section}
\numberwithin{equation}{section}

\newtheorem{Equation}{}[section]

\newtheorem{theorem}[Equation]{Theorem}
\newtheorem{proposition}[Equation]{Proposition}
\newtheorem{lemma}[Equation]{Lemma}
\newtheorem{corollary}[Equation]{Corollary}

\newtheorem{definition}[Equation]{Definition}

\newtheorem{remark}[Equation]{Remark}

\def\Coker{\operatorname{Coker}}

\def\Dom{\operatorname{Dom}}

\def\ind{\operatorname{ind}}
\def\End{\operatorname{End}}

\def\red{\operatorname{red}}
\def\reg{\operatorname{reg}}
\def\av{\operatorname{av}}

\def\Im{\operatorname{Im}}
\def\Ind{\operatorname{Ind}}

\def\Ker{\operatorname{Ker}}

\def\Proj{\operatorname{Proj}}
\def\Sgn{\operatorname{Sgn}}
\def\Sign{\operatorname{Sign}}

\def\Tr{\operatorname{Tr}}
\def\tr{\operatorname{tr}}

\def\Range{\operatorname{Range}}

\def\Ind{\operatorname{Ind}}
\def\ind{\operatorname{ind}}
\def\End{\operatorname{End}}
\def\maA{\mathcal{A}}
\def\maB{\mathcal{B}}
\def\maC{\mathcal{C}}
\def\maD{\mathcal{D}}
\def\maE{\mathcal{E}}
\def\maH{\mathcal{H}}
\def\maK{\mathcal{K}}

\def\maJ{\mathcal{J}}
\def\maL{\mathcal{L}}
\def\maM{\mathcal{M}}
\def\maN{\mathcal{N}}

\def\maP{\mathcal{P}}

\def\maF{\mathcal{F}}
\def\del{\partial}
\def\bJ{\mathrm{J}}
\def\bB{\mathrm{B}}

\def\C{\mathbb C}

\def\R{\mathbb R}
\def\S{\mathbb S}
\def\Z{\mathbb Z}
\def\N{\mathbb N}

\def\maA{{\mathcal A}}
\def\maB{{\mathcal B}}
\def\maC{{\mathcal C}}
\def\maE{{\mathcal E}}
\def\maF{{\mathcal F}}
\def\maM{{\mathcal M}}
\def\maQ{{\mathcal Q}}

\def\maH{{\mathcal H}}

\def\maJ{{\mathcal J}}
\def\maN{{\mathcal N}}
\def\maU{{\mathcal U}}

\def\maS{{\mathcal S}}

\def\tb{{\widetilde b}}

\def\tv{{\widetilde v}}
\def\tV{{\widetilde V}}
\def\tN{{\widetilde N}}
\def\tM{{\widetilde M}}
\def\tG{{\widetilde G}}

\def\tD{\widetilde{D}}

\def\tM{\widetilde{M}}

\def\tX{{\widetilde X}}
\def\tY{{\widetilde Y}}
\def\tS{{\widetilde S}}
\def\tU{\widetilde{U}}
\def\tQ{\tilde{Q}}
\def\tB{\tilde{B}}

\def\tP{\tilde{P}}
\def\tp{\tilde{p}}
\def\tJ{\tilde{J}}
\def\tK{\tilde{K}}
\def\tF{\tilde{F}}
\def\tW{\tilde{W}}
\def\tC{\tilde{C}}
\def\tE{\tilde{E}}

\def\hf{\hat{f}}
\def\hS{\hat{S}}
\def\hD{\hat{D}}
\def\hM{\hat{M}}

\def\hV{\hat{V}}
\def\hQ{\hat{Q}}
\def\hF{\hat{F}}
\def\hP{\hat{P}}

\def\tu{\widetilde{ u}}

\def\thS{\widetilde{\hat S}}
\def\thD{\widetilde{\hat D}}
\def\thM{\widetilde{\hat M}}
\def\thV{\widetilde{\hat V}}
\def\thQ{\widetilde{\hat Q}}
\def\thF{\widetilde{\hat F}}
\def\thP{\widetilde{\hat P}}

\def\pa{\partial}

\def\ep{\epsilon}

\begin{document}

\title{The Higson-Roe exact sequence and $\ell^2$  eta invariants}

\author[M-T. Benameur]{Moulay-Tahar Benameur}
\address{I3M, UMR 5149 du CNRS, Montpellier, France}
\email{moulay.benameur@univ-montp2.fr}

\author[I. Roy]{Indrava Roy}
\address{Sapienza Universita di Roma, Italy}
\email{indrava@gmail.com}

\begin{abstract}
The goal of this paper is to solve the problem of existence of an $\ell^2$ relative eta morphism on the Higson-Roe structure group. Using the Cheeger-Gromov $\ell^2$ eta invariant, we construct a group morphism from the Higson-Roe maximal structure group constructed in \cite{HigsonRoe2010} to the reals. When we apply this morphism to the structure class associated with the spin Dirac operator for a metric of positive scalar curvature, we get the spin $\ell^2$ rho invariant. When we apply this  morphism to the  structure class associated with an oriented homotopy equivalence, we get the difference of the $\ell^2$ rho invariants of the corresponding signature operators. We thus get new proofs for the classical $\ell^2$ rigidity theorems of Keswani obtained in \cite{KeswaniVN}.
\end{abstract}

\date{\today}

\maketitle
\tableofcontents

\section{Introduction}

The eta invariant of elliptic operators first appeared in \cite{APS1} as a boundary correction term appearing in the calculation of the index of a Fredholm operator associated with a global boundary value problem on even dimensional manifolds with boundary. The eta invariant $\eta (D)$ is a measure of asymmetry of the spectrum of the operator $D$ and turns out to be well-defined for any elliptic self-adjoint differential operators $D$ on a closed odd dimensional manifold $M$.  This is a  sensitive invariant, but there is a relative version which is more stable and often has interesting topological properties.  More precisely, given two group morphisms  $\sigma_1, \sigma_2: \pi_1(M)\to U(N)$ and the associated flat bundles $E_{\sigma_i}$, we may form the twisted elliptic differential operators $D\otimes E_{\sigma_i}$ and the relative eta invariant  is by definition \cite{APS2, APS3}
$$
\rho_{\sigma_1, \sigma_2} (D) := \eta (D\otimes E_{\sigma_1}) - \eta (D\otimes E_{\sigma_2}).
$$
If $D$ is for instance the signature operator on $M$ then it was proved by Atiyah, Patodi and Singer that $\rho_{\sigma_1, \sigma_2} (D)$ is a differential invariant of $M$. This property had important consequences, as when $\pi_1(M)$  has torsion  this invariant is not a homotopy invariant, see for instance \cite{Weinberger, ChangWeinberger1, ChangWeinberger2, Mathai, PiazzaSchick}. Notice that the relative index is zero thanks to the Atiyah-Singer index formula and the relative eta invariant can  thus be seen as a refined {\em secondary} invariant, in fact some transgression of the index \cite{CheegerSimons, Lott}. In general, when reduced modulo $\Z$,  this invariant becomes more computable and inherits topological properties,  there is indeed a topological index formula  in $\R/\Z$ which expresses it in terms of characteristic classes \cite{APS3}.

In \cite{CheegerGromov}, Cheeger and Gromov extended the APS eta invariant and introduced an $\ell^2$ version of the eta invariant exactly as Atiyah introduced an $\ell^2$ version of the index. More precisely, given a Galois $\Gamma$-covering $\tM\to M$ and a $\Gamma$-invariant generalized Dirac operator  $\tD$ over $\tM$, the Cheeger-Gromov eta invariant is defined by the absolutely convergent integral \cite{CheegerGromov}:
$$
\eta_{(2)} (\tD) := \int_0^\infty \tau
 (\tD e^{-t{\tD}^2})\,dt/{\sqrt {\pi t}}.
$$
The operator $\tD$ induces a generalized Dirac  operator $D$ over $M$ and the Cheeger-Gromov rho invariant (also called $\ell^2$ relative eta invariant) is given by:
$$
\rho_{(2)} (\tD) := \eta_{(2)} (\tD) - \eta (D).
$$
Again, while the $\ell^2$ index coincides with the usual index by Atiyah's theorem \cite{AtiyahCovering}, the Cheeger-Gromov $\ell^2$ relative eta invariant is in general non trivial, and provides interesting geometric invariants. For the signature operator, Cheeger and Gromov proved that it is a differential invariant and this was again used to distinguish homotopy invariant non diffeomorphic manifolds \cite{ChangWeinberger1, ChangWeinberger2}. There is another important scope of applications of these invariants to the study of the moduli spaces of metrics of positive scalar curvatures on spin manifolds, we refer for instance to \cite{PiazzaSchickI,PiazzaSchick} for more details and explanations.\\

The goal of this paper is to explore the relation between the rigidity theorems of the $\ell^2$ relative eta invariant of Cheeger-Gromov and the recently obtained Higson-Roe exact sequence. This will hopefully  improve the understanding of the $\ell^2$ eta invariants and its close relation with the different assembly maps in $K$-theory as well as in $L$-theory \cite{KeswaniAPS, Weinberger}.

Given a discrete countable group $\Gamma$, the Kasparov assembly map for the group $\Gamma$ was constructed by G. Kasparov, in his $K$-theory approach to the Novikov conjecture,  using the higher index of elliptic operators    \cite{KasparovRussian}. This is a group morphism $\mu_{\red, \Gamma}$ from the analytic $K$-homology of the classifying space $B\Gamma$ of $\Gamma$, to the $K$-theory of the reduced $C^*$-algebra $C_{\red}^*\Gamma$ of $\Gamma$, which factors through a similar assembly map $\mu_\Gamma$ with values in the $K$-theory of the full $C^*$-algebra $C^*\Gamma$:
$$
\mu_{\red, \Gamma}\; : \;K_* (B\Gamma) \stackrel{\mu_\Gamma}{\longrightarrow} K_* (C^*\Gamma) \longrightarrow K_*(C^*_{\red}\Gamma).
$$
From the early works, it was expected that for a large class of torsion free groups $\Gamma$, the map $\mu_{\red, \Gamma}$ should be an isomorphism. Notice that for $K$-amenable groups (which includes amenable and even a-T-menable groups \cite{HigsonKasparov}) the maps $\mu_{\red, \Gamma}$ and $\mu_{ \Gamma}$ coincide, but in general they are different as can be easily seen for property (T) groups \cite{Julg, Valette}.  On the other hand, since the APS relative eta invariant \cite{APS1, APS3} is built out of the representation theory of $\Gamma$, it can  naturally  be related with  maximal assembly maps.

When the group $\Gamma$ has torsion, the assembly maps $\mu_{\red,\Gamma}$ and $\mu_{\Gamma}$ are  not isomorphisms in general. Roughly speaking, one needs to add the higher indices for proper (non-free) actions. There is indeed a more elaborate assembly map, which was constructed by P. Baum and A. Connes \cite{BaumConnes} (see also the important subsequent paper \cite{BaumConnesHigson}) and which  replaces the analytic $K$-homology of the classifying space (the LHS) by a more refined $K$-homology group associated with a classifying space for all proper actions of $\Gamma$.  This Baum-Connes assembly map is conjectured to always be an isomorphism, see for instance the monographs \cite{Julg, Valette} or the more recent overview paper \cite{Schick}.

In his PhD thesis, see \cite{KeswaniAPS, KeswaniControlledPaths}, N. Keswani first used the  isomorphism assumption of the maximal Baum-Connes assembly map to obtain some rigidity properties of the relative APS eta invariant. Because the group was assumed to be torsion free, Keswani actually used the Kasparov maximal assembly map and this was conceptualized  by Higson and Roe in \cite{HigsonRoe2010}, who obtained a clearer relation between the relative APS eta invariant and  the Kasparov maximal assembly map, precisely through their maximal structure exact sequence. Recently similar results were obtained by Deeley-Goffeng \cite{DeeleyGoffeng1, DeeleyGoffeng2} using geometric models of the analytic structure groups of Higson and Roe. While the $C^*$-algebra $K$-theory approach of Keswani allowed him to also deduce the similar rigidity results for the $\ell^2$ relative eta invariant of Cheeger-Gromov \cite{CheegerGromov}, only the APS relative eta invariant was treated in \cite{HigsonRoe2010}. It is one of the goals of the present paper to show that the Higson-Roe exact sequence \cite{HigsonRoe} can be used some steps further  to encompass the Cheeger-Gromov invariant. In the process, we obtained  independent semi-finite results which will be used in a forthcoming paper to deduce similar properties for foliated rho invariants \cite{BenameurPiazza}.  \\

We proceed now to explain more precisely the results of the present paper.
Recall  that the analytic $K$-homology group of $B\Gamma$ is  isomorphic through the so-called Paschke-duality \cite{Paschke, Higson} to an inductive limit of $K$-theory groups of  finite propagation  Calkin algebras on some geometric Hilbert modules, see section \ref{HRanalytic}. It is also well known that  the $C^*$-algebra $C^*\Gamma$ is Morita equivalent to the $C^*$-algebra of compact operators on the same Hilbert modules.  Using these Hilbert modules and the above mentioned identifications, the Higson-Roe exact sequence was obtained in \cite{HigsonRoe} and can be stated as an exact triangle in which the Kasparov assembly map appears  as a boundary map, see the first of the following diagrams:

\begin{displaymath}\label{Figure1}
\xymatrixcolsep{1pc}\xymatrix{
K_{[*]}(B\Gamma) \ar[rr]^{\mu^{\red}_\Gamma}  &  & K_{[*]}(C_{\red}^*\Gamma)  \ar[dl]^{}\\ & \maS_{\red, [*+1]} (\Gamma) \ar[ul]_{}  }
\quad
\xymatrixcolsep{1pc}\xymatrix{
K_{[*]}(B\Gamma) \ar[rr]^{\mu_\Gamma}  &  & K_{[*]}(C^*\Gamma)  \ar[dl]^{}\\ & \maS_{[*+1]} (\Gamma) \ar[ul]_{}
}\\
\end{displaymath}
\begin{center}
Higson-Roe exact sequences
\end{center}

This gives a hint of the obstruction structure groups $\maS_{\red, *}(\Gamma)$ for the assembly  maps to be  isomorphisms, and they are thus expected, according to the Baum-Connes conjecture, to be trivial  when $\Gamma$ is torsion free.
 In \cite{HigsonRoe}, Higson and Roe actually constructed a commutative diagram from the classical surgery exact sequence  \cite{Wall} to the above analytic sequence, viewed as a long exact sequence, encompassing homotopy invariance properties of  signature operators.
There is a similar exact sequence associated with the maximal $C^*$-algebras which  was obtained in \cite{HigsonRoe2010} and which can be stated as the periodic $6$-term  exact triangle which is the second triangle above.

As explained above, this second diagram is  adapted  to the representation theory of the group $\Gamma$. The structure group $\maS_1(\Gamma)$ was indeed  intensively used in \cite{HigsonRoe2010} as a receptacle for  higher structure invariants, and they  deduced the rigidity  theorems of Keswani  about the APS relative eta invariant \cite{KeswaniAPS}. In \cite{BenameurMathaiPSC, BenameurMathaiSignature, BenameurMathaiSS}, the authors used again this exact sequence and extended the above results so as to obtain explicit connections with the APS spectral flow. \\

We show here that the Higson-Roe exact sequence, can  be used to also deduce the deep rigidity results about semi-finite spectral invariants. Depending on the geometric situation, one needs to introduce the appropriate exact sequence modifying, for foliations for instance \cite{BenameurPiazza},  the Higson-Roe sequence. We postpone this discussion and we only concentrate here on the Cheeger-Gromov relative eta invariant for Galois coverings \cite{CheegerGromov}. We  show more precisely  that  there is an $\ell^2$  structure group $\maS_*^{(2)}(\Gamma)$ which is a natural receptacle for a higher Cheeger-Gromov relative $\ell^2$ eta invariant. The group $\maS_1^{(2)}(\Gamma)$ is introduced as an inductive limit of $K$-theory groups of appropriate $C^*$-algebras but now the $C^*$-algebras  are associated with  semi-finite von Neumann algebras of $\Gamma$-invariant  operators on Galois $\Gamma$-coverings \cite{AtiyahCovering}.  These semi-finite structure groups fit into the following commutative diagram which is explained later:

\vspace{0,2cm}
\begin{equation*}\label{beta}
\hspace{-0,47cm}
\begin{CD}
K_0(B\Gamma) @>{\mu_{\Gamma}}>>  K_0(C^*\Gamma) @>>>  \maS_{1}(\Gamma) @>>> K_1(B\Gamma)  @>{\mu_{\Gamma}}>> K_1(C^*\Gamma) \\
@V{=}VV  @V{(\tau_*,\Tr_*)\circ\alpha_*}VV                          @V{\alpha_*}VV       @V{=}VV                           \\
 K_0(B\Gamma) @>{\partial}>>  \R\oplus \Z @>>>  \maS_{1}^{(2)}(\Gamma) @>>> K_1(B\Gamma)  @>>> 0  \\
\end{CD}
\end{equation*}
\vspace{0,2cm}

The group morphism $
\alpha_*: \maS_*(\Gamma) \longrightarrow \maS_*^{(2)}(\Gamma)$ is defined by combining
 the regular and trivial representations of  $\Gamma$ (see Section 3.3 for the precise definition of the morphism $\alpha_*$).
Even if the definition of the structure groups $\maS_*^{(2)} (\Gamma)$ is valid for $*=0,1$, only the group $\maS_1^{(2)}(\Gamma)$ will be used here. We show that it enters in a short exact sequence
$$
0\to \R \longrightarrow \maS_1^{(2)}(\Gamma)\longrightarrow K_1(B\Gamma) \to 0,
$$
and that it is a home, under appropriate assumptions, for the Cheeger-Gromov $\ell^2$ invariants. It is worth pointing out that the kernel of the morphism  $\ind: K_0(B\Gamma)\to \Z$ induced by the index map can be shown a fortiori, using Atiyah's $\ell^2$ index theorem \cite{AtiyahCovering}, to coincide with the group $\maS_0^{(2)}(\Gamma)$.

The second part of our study is devoted to the geometric picture of our structure group. More precisely, we introduce a geometric version of the $\ell^2$ structure group that we call $\maS_{1}^{(2), geo}  (\Gamma)$ using cycles \`a la Baum-Douglas together with some choices and moding out by moves similar to the Higson-Roe ones except that we need to apply the $\ell^2$ index theorem for coverings with boundary, as proved by Ramachandran in \cite{Ramachandran}. We show  that there is   a well defined group morphism
$$
\xi: \maS_{1}^{(2), geo}  (\Gamma)\longrightarrow \R
$$
which allows to recover the Cheeger-Gromov relative $\ell^2$ invariant in the interesting geometric situations. Moreover, we prove that the geometric and analytic $\ell^2$ structure groups are in fact isomorphic. Associated with any $\ell^2$ geometric cycle, there is an obvious analytic class in $\maS_1^{(2)}(\Gamma)$. The following is one of the main results in this paper:\\

{\bf{Theorem }}\ref{analyticGeometric}
{\em {The analytic class of  a geometric cycle only depends on its class in the $\ell^2$ geometric group $\maS_1^{(2), geo}(\Gamma)$ and hence induces a well defined group morphism
$$
\maS_1^{(2), geo}(\Gamma) \longrightarrow \maS_1^{(2)}(\Gamma).
$$
}}
The proof  occupies an important part of the paper, due to the continuity of the involved spectra of operators and to the use of some deep  results on Boundary Value Problems \cite{BW} that we had to extend to Galois coverings. As a corollary, we thus eventually succeeded to construct the allowed $\ell^2$ group morphism
$$
\xi_{(2)} : \maS_1 (\Gamma) \longrightarrow \R.
$$
In order to explain the new issues created by the semi-finite situation, we point out that if $\Lambda$ is the subgroup of $\R$ which is the image of the  $K$-theory group of $C^*_{\red}\Gamma$ under the additive map induced by evaluation at the unit, then there is, a priori and in general, no well defined morphism from $K_1(B\Gamma)$ to $\R/\Lambda$ which would be compatible with our morphism $\xi_{(2)}$. Notice that $\Z\subset \Lambda$ and we have equality when the  Baum-Connes map is surjective.

The usual corollaries regarding obstructions to the existence of metrics with positive scalar curvature or regarding the homotopy invariance of the Cheeger-Gromov $\ell^2$ relative eta invariant are deduced using a construction similar to the one described in \cite{HigsonRoe2010}. We get more precisely a new proof of the following two theorems of Keswani:\\

{\bf{Theorem }}\ref{PSC} \cite{KeswaniVN}
{\em {Assume that $M$ is  a closed odd dimensional spin manifold which has a metric of positive scalar curvature and let $f:M\to B\Gamma$ be a classifying map for the $\Gamma$-cover $\tM\to M$. Assume that the assembly map $\mu_{\Gamma}$  is an isomorphism, then the Cheeger-Gromov rho invariant of the spin Dirac operator vanishes.}}\\

{\bf{Theorem }}\ref{Homotopy} \cite{KeswaniVN}
{\em{Let $M,M'$ be closed odd-dimensional manifolds equipped with maps $f:M\rightarrow B\Gamma$, $f':M'\rightarrow B\Gamma$, and let $h:M\rightarrow M'$ be an oriented  homotopy equivalence which is compatible with the maps $f$ and $f'$. Assume also that the assembly map $\mu_\Gamma$  is an isomorphism. Then  the Cheeger-Gromov rho-invariants  $\rho_{(2)}(f,D)$ and $\rho_{(2)}(f',D')$ associated with the odd signature operators $D$ and $D'$ on $M$ and $M'$ respectively do coincide. }}\\

Once we have constructed the morphism $\xi$, the proof of Theorem \ref{PSC} is a rephrasing of the Higson-Roe proof which immediately extends to our semi-finite situation, and uses  the previous theorem about the equivalence of analytic and geometric semi-finite structure groups.
On the other hand, we point out that in the proof of the second theorem  \ref{Homotopy} we use the APS projection $\chi_{\geq} (D)$. This is dictated by the BVP results used from \cite{BW} and by the  APS formulae from \cite{APS1}. The resulting minor difference from the conventions of \cite{HigsonRoe2010} is then easily adjusted by  modifying accordingly the class  associated with an oriented homotopy equivalence in our $\ell^2$ structure group (see Section \ref{SectionHomotopy})  and by introducing a sign change in the definition of the opposite of a cycle.\\

%On the other hand, we point out that in the proof of the second theorem  \ref{Homotopy} we use the APS projection $\chi_{\geq} (D)$  in opposition to the projection $\chi_> (D)$ used in \cite{HigsonRoe2010}[Definition 8.16]. This is imposed by our use of the BVP results from \cite{BW} and of the  APS formulae from \cite{APS1}. More precisely, we choose a different convention for the relative eta morphism  and hence there is a sign change in the definition of the opposite of a cycle. This is a minor difference from the conventions of \cite{HigsonRoe2010} since we just had to modify accordingly in Section \ref{SectionHomotopy} the class  associated with an oriented homotopy equivalence in our $\ell^2$ structure group. \\

Let us describe more precisely the contents of each section. We have devoted Section \ref{Background} to a brief review of some results on Hilbert modules associated with Galois coverings that will be used later on. In Section \ref{analytic}, we first review the maximal Higson-Roe exact sequence, then we introduce our $\ell^2$ structure group $\maS_*^{(2)} (\Gamma)$.  In the end of Section \ref{analytic}, we show our short exact sequences and their compatibility with the Higson-Roe exact sequence, by using the first appendix. Section \ref{geometric} is the heart of the paper and is divided into  subsections. In Subsection \ref{Structure}, we define  our geometric $\ell^2$ structure group $\maS_1^{geo, (2)} (\Gamma)$. In Subsection \ref{geometricanalytic}, the main theorem \ref{Isomorphism} is proved. We use here many results on BVP for Galois coverings which are stated in the second and third appendices. In Section \ref{RhoMorphism}, we show that the Cheeger-Gromov $\ell^2$ relative eta invariant allows to define a group morphism from $\maS_1^{geo, (2)} (\Gamma)$ to the reals. The last Section \ref{Applications} is devoted to the rigidity corollaries of Keswani that we deduce from our results. We end the paper with three appendices which have independent interest. Appendix \ref{Compact} reviews a folklore result on the $K$-theory of the $\tau
$ compact operators on Galois coverings. Appendix \ref{Resolvent} proves a semi-finite version of a classical result on compactness of resolvents. Finally the last Appendix \ref{BVP} explains how to extend some classical BVP to Galois coverings, in particular we deduce the $L^2$-invertibility of the double Dirac operator and review the properties of the Calderon projectors in this semi-finite setting. Here the recent results of \cite{XieYu} were  useful.\\

{\em{Acknowledgements.}} The authors benefited from discussions with many colleagues and would like to express their gratitude to all of them. In particular, they are grateful to P. Albin, P. Antonini, S. Azzali, T. Fack, N. Higson, V. Mathai,  P. Piazza,  G. Skandalis and G. Yu. The authors are also indebted to the referee for many helpful suggestions. Part of this work was done while the first author was visiting the mathematics department at La Sapienza-Roma and he would like to thank the members for the warm hospitality and the research network GDRE France-Italy of Noncommutative Geometry (GDRE GREFI-GENCO) for the financial support.  The second author is supported by the INdAM Cofund Marie Curie fellowship ''PCOFUND-GA-2009-245492'' and hosted at the Department of Mathematics at La Sapienza University, Rome, and he would also like to thank the I3M institute at Montpellier for the hospitality.

\section{Background on Hilbert modules for coverings}\label{Background}
We review in this section some classical constructions of Hilbert modules over Galois coverings which will be used in the sequel. Since all the construction are classical, we shall be brief and only give the main ideas. For the basic  theory of Hilbert modules, see for instance \cite{Lance}.

We assume in the whole paper that the group $\Gamma$ is a  {\underline{countable infinite}} discrete group. The full group $C^*$-algebra of  $\Gamma$, that is the maximal completion $C^*$-algebra, is denoted as usual  $C^*\Gamma$. There are two representations of $C^*\Gamma$ that will be mainly used in the present paper, the regular representation $\pi_{\reg}$ in the $\ell^2$ Hilbert space of $\Gamma$, and the trivial representation $\pi_{\av}$ in the complex numbers $\C$.

Let $X$ be a compact space (a finite CW complex) and  let $\tilde{X}\xrightarrow{\Gamma} X$ be a Galois covering over $X$. All the results of this section apply to locally compact $X$ replacing in all constructions the continuous sections by the compactly supported sections before passing to completions, but this will not be needed in the present paper. It will also be important in the next sections to sometimes assume that our spaces are manifolds, this will be emphasized explicitly.  We fix a Borel measure on $X$ and its lift to $\Gamma$-invariant Borel measure on $\tX$.   We denote by $\Xi_\Gamma$ the Mishchenko flat bundle of line $C^*\Gamma$-modules  whose total space is $\Xi_\Gamma= \tX\times_\Gamma C^*\Gamma$, the quotient of the Cartesian product $\tX\times C^*\Gamma$ under the right action of $\Gamma$
$$
(\tilde{x}, T) \gamma := (\tilde{x} \gamma, \delta_{\gamma^{-1}} T).
$$
For $\alpha\in \Gamma$, $\delta_\alpha$ denotes the characteristic function of $\{\alpha\}$ viewed as an element of $C^*\Gamma$.
The space of continuous sections of the bundle $\Xi_\Gamma$ over $X$ can be identified with the space $C(\tilde{X},C^*\Gamma)^\Gamma$ of $\Gamma$-invariant elements of the space $C(\tX, C^*\Gamma)$ of continuous functions from $\tX$ to $C^*\Gamma$. The left action of $\Gamma$ on the space $C(\tilde{X},C^*\Gamma)$ is  given, for $\xi=\sum_g \xi_g\otimes \delta_g\in C(\tilde{X}, \C\Gamma)$ with compactly supported $\xi_g=g\xi_e$ ($e$ being the neutral element of $\Gamma$),  by:
$$
\alpha\cdot \xi = \sum_g \alpha \xi_g\otimes \delta_{\alpha g}, \quad\text{ where }(\alpha \xi_g)(\tilde{x})=\xi_g(\tilde{x}\alpha).
$$
It is then known that $C(\tilde{X},C^*\Gamma)^\Gamma$ yields a right Hilbert module (actually a finitely generated projective module)  $\maE_{X, \Gamma}$ over  the $C^*$-algebra $C(X)\otimes C^*\Gamma$. More precisely,  for $\xi,\eta$ in the dense subspace $ C(\tilde{X},\C\Gamma)^\Gamma$, the inner product is given by:
\begin{eqnarray*}
\label{innprod1}
<\xi,\eta>(x,g)=\sum_{[\tilde{x}]=x} <\xi(\tilde{x},e),\eta(\tilde{x},g)>
\end{eqnarray*}
where $\xi(\tilde{x},g):= \xi_g(\tilde{x})$ for an expansion $\xi=\sum_{g\in \Gamma} \xi_g \delta_g$, with $\xi_g\in C_c(\tilde{X})$ as above.

The module structure is given for $ f\otimes\delta_\alpha \in C(X)\otimes \C\Gamma$ by:
\begin{eqnarray*}
\label{mod1}
\xi\star(f\otimes \delta_\alpha)(\tilde{x})= f([\tilde{x}])\sum_{g\in\Gamma}\xi_g(\tilde{x})\delta_{g\alpha}.
\end{eqnarray*}
 It is easy to check that these rules define  the Hilbert module structure of $\maE_{X, \Gamma}$.

We denote by $S$ a fixed (topological) hermitian bundle over $X$ and by $\tS$ its  lift to a $\Gamma$-equivariant hermitian vector bundle over $\tX$. The multiplication action of $C(X)$ on the Hilbert space $L^2(X, S)$ extends to the left action of the $C^*$-algebra $C(X)\otimes C^*\Gamma$ on the Hilbert right $C^*\Gamma$-module $L^2(X,S)\otimes C^*\Gamma$.
We now similarly review the structures of the Hilbert module, denoted in the present paper by $\maE_{S, \Gamma}$,  of $L^2$-sections of the bundle $S\otimes \Xi_\Gamma$. For $\xi,\eta \in C(\tilde{X},\tilde{S}\otimes \C\Gamma)^\Gamma$ we define the inner-product by:
\begin{eqnarray*}
\label{innprod2}
<\xi,\eta>(g)=\int_{F} \sum_{\alpha\in \Gamma}<\xi(\tilde{x},\alpha),\eta(\tilde{x},\alpha g)>_{S_x} d\tilde{x}
\end{eqnarray*}
where $F$ is a fundamental domain in $\tX$ for the deck transformations.
The right module action of $\C\Gamma$ is given  for $\xi\in C(\tilde{X},\tilde{S}\otimes \C\Gamma)^\Gamma$ by:
\begin{eqnarray*}
\label{mod2}
(\xi\star \delta_\alpha)(\tilde{x})= \sum_{g\in\Gamma}\xi_g(\tilde{x})\otimes \delta_{g\alpha},
\end{eqnarray*}
where  $\xi=\sum_{g\in \Gamma} \xi_g \otimes \delta_g$, with as before $\xi_g\in C_c(\tilde{X}, \tilde{S})$.

Recall  the composition construction  which allows to define the right Hilbert $C^*\Gamma$-module  \cite{Lance}
$$
\maE_{X,\Gamma} \otimes_{C(X)\otimes C^*\Gamma}(L^2(X,S)\otimes C^*\Gamma).
$$
The following explicit isomorphism is needed later.

\begin{proposition}
\label{HMiso1}
We have an   isomorphism of right Hilbert $C^*\Gamma$-modules:
$$
\Psi:\maE_{X, \Gamma} \otimes_{C(X)\otimes C^*\Gamma}(L^2(X,S)\otimes C^*\Gamma) \; \longrightarrow \; \maE_{S,\Gamma},
$$
which is induced by the formula $
\Psi(\xi\otimes(h\otimes \phi))(\tilde{x})=h({x})\otimes (\xi\star\phi)(\tilde{x})$  for $\xi\in C(\tilde{X}, \C\Gamma)^\Gamma$, $h\in C(X,S)$  and $\phi \in \C\Gamma$.

\end{proposition}

\begin{proof}
That $\Psi$ is well-defined is clear. Indeed one has for $f\in C(X), \psi \in \C\Gamma$ and $\xi, h, \phi$ as above:
$$
\Psi(\xi.(f\otimes \psi)\otimes(h\otimes \phi))= \Psi(\xi\otimes(f.h\otimes \psi\star\phi)).
$$
since, for $\phi=\delta_\alpha$ and $\psi=\delta_\beta$ for fixed $\alpha, \beta\in \Gamma$, a direct computation shows that
\begin{multline*}
\Psi(\xi.(f\otimes \psi)\otimes(h\otimes \phi)) (\tilde x) = f(x) {h}({x})\otimes \sum_{g\in \Gamma} \xi_{g\beta^{-1}} (\tilde x) \delta_{g\alpha}, \text{ while }\\
\Psi(\xi\otimes(f.h\otimes \psi\star\phi)) (\tilde x) = f(x) {h}({x})\otimes \sum_{k\in \Gamma} \xi_k (\tilde x) \delta_{k\beta\alpha}.
\end{multline*}
A change of variables $k\beta = g$ ends the verification. In the same way, the direct computation shows that $\Psi$ is an isometry  on the dense pre-Hilbert submodule $C(\tilde{X},\C\Gamma)^\Gamma\otimes_{C(X)\otimes \C\Gamma}(L^2(X,S)\otimes \C\Gamma)$.

Let $h\in C(X,S)$ and let $\tilde{s}\in C(\tilde{X},\C\Gamma)^\Gamma$ inducing $s\in C(X,\tilde{X}\times_\Gamma\C\Gamma)$. Consider the element $\eta\in \maE_{S,\Gamma}$ given by $\eta=h\otimes s$. Then the element $\tilde{s}\otimes h\otimes \delta_e$ maps to $\eta \in L^2(X,S\otimes \Xi_\Gamma)$ under $\Psi$.
Since such elements $\eta$ generate a dense subspace of $\maE_{S,\Gamma}$, we conclude that $\Psi$ has dense image.

\end{proof}

The formula in the previous proposition simplifies when $\phi=\delta_\alpha$ as follows $
\xi\star \phi(\tilde{x}) =\sum_g \xi_g(\tilde{x})(\delta_{g\alpha}).$

There is an alternate description of the Hilbert module $\maE_{S, \Gamma}$  which was given by Connes and Skandalis for foliations  in \cite{ConnesSkandalis}. The Connes-Skandalis Hilbert $C^*\Gamma$ module will be the completion of the space $C_c(\tilde{X},\tilde{S})$ of compactly supported continuous sections of $\tS$ over $\tX$.  The inner-product and module action are given for $\xi,\eta\in C_c(\tilde{X},\tilde{S}),\phi \in \C\Gamma$ by:
$$
<\xi,\eta>(g)= \int_{\tilde{X}} <\xi(\tilde{x}g),\eta(\tilde{x})> d\tilde{x}\quad \text{ and } \quad \xi.\phi(\tilde{x})= \sum_{g\in \Gamma}\phi(g)\xi(\tilde{x}g^{-1}).
$$

So, in reference to Connes-Skandalis \cite{ConnesSkandalis}, we denote the completion obtained in this way by $\maE_{S,\Gamma}^{CS}$.

\begin{proposition}
\label{HMiso2}
The isomorphism of Hilbert $C^*\Gamma$-modules:
$$
\theta: \maE_{S, \Gamma}^{CS} \longrightarrow \maE_{{S}, \Gamma}.
$$
is induced by the map $C_c(\tilde{X},\tilde{S})\rightarrow C(\tilde{X},\tilde{S}\otimes \C\Gamma)^\Gamma$ given by
$
\theta(\xi)=\sum_{g\in \Gamma} g\xi\otimes \delta_g.$

\end{proposition}

\begin{proof}
Again the proof is straightforward and we shall be brief. We first notice that $\theta(\xi.\phi)(\tilde{x})$ and $[\theta(\xi)\star\phi](\tilde{x})$ both coincide with
\begin{eqnarray*}
 \sum_{g\in \Gamma} \sum_{g'\in \Gamma}\phi(g')\xi(\tilde{x}gg'^{-1}) \otimes \delta_g\
\end{eqnarray*}

That $\theta$ is an isometry is also clear since
\begin{eqnarray*}
<\theta(\xi),\theta(\xi)>(g)&=& \int_F \sum_{\alpha\in \Gamma} <\theta(\xi)(\tilde{x},\alpha),\theta(\xi)(\tilde{x},\alpha g)>_{S_x} d\tilde{x}\\
&=& \int_F \sum_{\alpha\in \Gamma} <\xi(\tilde{x}\alpha),\xi(\tilde{x}\alpha g)>_{S_x} d\tilde{x}\\
&=& \int_{\tilde{X}} <\xi(\tilde{x}),\xi(\tilde{x}g)>_{S_x} d\tilde{x}
\end{eqnarray*}
The equality of the two inner products follows as the last equation is the definition of $<\xi,\xi>(g)$.
Since any $s\in C(\tX,\tS\otimes \C\Gamma)^\Gamma$ can be expressed in the form $ s = \sum_g s_g\otimes \delta_g$, where $s_g\in C_c(\tX,\tS)$ and $s_g(\tilde{x})=s_e(\tilde{x}g)$. One has $\theta(s_e)=s$.
\end{proof}

Recall the representations of $C^*\Gamma$, $\pi_{reg}$ as bounded operators on $\ell^2\Gamma$, and $\pi_{av}$ on $\C$ \cite{BenameurPiazza}. For $\phi\in \C\Gamma, \psi\in \ell^2\Gamma, \gamma\in \Gamma$, they are defined as:
$$[\pi_{reg}(\phi)](\psi)(\gamma)= \sum_{\gamma'\in \Gamma} \phi(\gamma\gamma'^{-1})\psi(\gamma') \quad\text{ and }\quad \pi_{av}(\phi)=\sum_{\gamma'}\phi(\gamma')$$

There is an isometric  isomorphism described in \cite{BenameurPiazza}, Lemma 3.2:
\begin{eqnarray*}
\maE_{{S}, \Gamma}^{CS}\otimes_{\pi_{\reg}}\ell^2(\Gamma) \longrightarrow L^2(X,S)\otimes_{\C}\ell^2(\Gamma).
\end{eqnarray*}
hence, using  Proposition \ref{HMiso1} and Proposition \ref{HMiso2}, we deduce the following useful proposition.

\begin{proposition}
\label{HMisoreg}
There is an isomorphism of Hilbert spaces:
$$
\Psi_{\reg}: \maE_{S, \Gamma} \otimes_{\pi_{\reg}}\ell^2(\Gamma)\longrightarrow L^2(X,S)\otimes \ell^2(\Gamma),
$$
\end{proposition}

It is worth pointing out that $\Psi_{\reg}$ can be describe   on the dense subspace $\Im(\Psi)$ by the formula:
$$
\Psi_{\reg}(\Psi(\xi\otimes (h\otimes\phi)\otimes \psi))(x,\gamma)= h(x)\otimes \pi_{\reg}(\xi(\tilde{x})\star \phi)\psi(\gamma),
$$
for $\xi \in C(\tilde{X}, \C\Gamma)^\Gamma$, $h\in C(X,S)$, $\phi \in \C\Gamma$ and $\psi\in \ell^2(\Gamma)$. Here  $\tilde{x}$ is any element of the fibre of $\tX$ over $x$.

The similar assertion for the average representation $\pi_{\av}$ can be stated  using the isometric isomorphism \cite{BenameurPiazza}, Lemma 3.2:
$$
 \maE_{{S}, \Gamma}^{CS}\otimes_{\pi_{\av}} \C \longrightarrow L^2(X,S).
$$

\begin{proposition}
The following map is an isomorphism of Hilbert spaces:
$$
\Psi_{\av}: \maE_{S, \Gamma} \otimes_{\pi_{\av}}\C\longrightarrow L^2(X,S)
$$
given,  on the dense subspace $\Im(\Psi)$, by the formula:
$$
\Psi_{\av}(\Psi(\xi\otimes (h\otimes\phi)\otimes 1))(x)= [\sum_{g\in \Gamma} (\xi(\tilde{x})\star \phi)(g)]h(x)= \pi_{\av}(\xi(\tilde{x})\star \phi)h(x),
$$
for $\xi \in C(\tilde{X}, \C\Gamma)^\Gamma$, $h\in C(X,S)$ and $\phi \in \C\Gamma$.
\end{proposition}

\section{$\ell^2$ structure algebras}\label{analytic}

\subsection{Review of the Higson-Roe sequence}\label{HRanalytic}

In this subsection we recall the definitions of the Higson-Roe algebras $D^*_H(X)$ and $Q^*_H(X)$ associated with an ample representation of $C(X)$ on a separable Hilbert space $H$, as well as the {\em maximal} $C^*$-algebras $D^*_\Gamma(X)$ and $Q^*_\Gamma(X)$ defined  using the notion of ``lifts" of operators. These algebras fit into the following short exact sequences:
$$
0\to K(H)\longrightarrow D^*_H(X)\longrightarrow Q^*_H(X)\to 0
$$
and
$$
0\to \maK(\maE_{{S},\Gamma})\longrightarrow D^*_\Gamma(X)\longrightarrow Q^*_\Gamma(X)\to 0
$$
where $\maE_{{S},\Gamma}$ is the Hilbert $C^*\Gamma$-module $L^2(X, S\otimes \Xi_\Gamma)$ defined  in the previous section and $\maK(\maE_{{S},\Gamma})$ is the $C^*$-algebra of compact operators of $\maE_{{S},\Gamma}$ (\cite{Kasparov, Lance}). For any $f\in C(X)$ we shall as before also denote by $f$ the  operator in $H$ associated with $f$, and the Hilbert space $H$ will mostly be our favorite example $L^2(X, S)$.

\begin{definition}
We set the $C^*$-algebra
$$
D^*_H(X):=\{ T\in B(H)  \text{ such that } [T,f]=T f - fT \in K(H) \text{ for any }   f\in C(X)\}
$$
We denote by $Q^*_H(X)$ the quotient $C^*$-algebra of $D^*_H(X)$  with respect to the ideal $K(H)$ of compact operators.
\end{definition}

That $D^*_H(X)$ is a $C^*$-subalgebra of $B(H)$ is obvious and hence we get that $Q^*_H(X)$ is a $C^*$-subalgebra of the Calkin algebra.

Because we shall be working with the maximal $C^*$-algebras, we first recall the definition of a {\em {lift}} of an adjointable operator between Hilbert modules \cite{HigsonRoe2010}, also called {\em{a connection}} in \cite{ConnesSkandalis}, Appendix A. We shall implicitly use the isomorphism $\Psi$ from  Proposition \ref{HMiso1}.

If $A$ and $B$ are unital $C^*$-algebras and if $E$ and $E'$ are Hilbert $C^*$-modules over $A$ and $B$, respectively and  $\phi: A\rightarrow \maL(E')$ is a non-degenerate $*$-homomorphism, then  we denote for any $e\in E$ by $L_e: E' \to E\otimes_A E'$ the adjointable operator given by $L_e (e'):=e\otimes_A e'$. We denote by $L_e^*$ its adjoint operator.

\begin{definition}\cite{ConnesSkandalis, HigsonRoe2010}
\label{lift}
%\todo[inline]{Do this later.}
Let $T$ be an adjointable operator on the Hilbert $B$-module $E'$ which commutes with the action of $A$ modulo $B$-compact operators. A lift of $T$ is an adjointable operator ${\hat T}$ on the Hilbert $B$-module $E\otimes_A E'$ such that for any $e\in E$, the following diagrams commute up to $B$-compact operators:

\[
\begin{CD}
E'
@> T >>
  E' \\
  @V L_e VV
@VV  L_e V\\
  % \arrow{ur}{\psi} \arrow[swap]{d}{\chi} \arrow{d}{\Psi} \\
E\otimes_A E'
@>\ \  {\hat{T}}\ \
>>E\otimes_A E'
\end{CD}
\quad \text{ and \;\;}\quad
%\]
%
%and
%
%
%
%\[
\begin{CD}
 E\otimes_A E'
@> {\hat{T}} >>
  E\otimes_A E' \\
  @V L_e^* VV
@VV  L_e^* V\\
  % \arrow{ur}{\psi} \arrow[swap]{d}{\chi} \arrow{d}{\Psi} \\
E'
@>\ \  T\ \
>>E'
\end{CD}
\]

\end{definition}

\vspace{0,5cm}

Set $H_\Gamma:=H\otimes_{\C} C^*\Gamma$ for the free Hilbert $C^*\Gamma$-module. Given $\xi\in \maE_{X, \Gamma}$, we define the operator
$$
L_\xi: H_\Gamma \longrightarrow \maE_{X, \Gamma} \otimes_{C(X)\otimes C^*\Gamma} H_\Gamma \text{ as } L_\xi (u) := \xi \otimes_{C(X)\otimes C^*\Gamma} u.
$$
We also consider the operator $L'_\xi:= \Psi \circ L_\xi : H_\Gamma \to \maE_{S, \Gamma}$ where $\maE_{S, \Gamma}= L^2(X, S\otimes \Xi_\Gamma)$. Then, for a bounded operator $T$ on the Hilbert space $H$ and a lift $\hat T$ of $T\otimes I$,  the following  diagram commutes up to compact operators:

\[
\begin{CD}
H _\Gamma
@> T\otimes I >>
  H_\Gamma  \\
  @V L'_\xi VV
@VV  L'_\xi V\\
  % \arrow{ur}{\psi} \arrow[swap]{d}{\chi} \arrow{d}{\Psi} \\
\maE_{{S},\Gamma}
@>\ \  {\hat{T}}\ \
>>\maE_{{S},\Gamma}
\end{CD}
\]
\\

\begin{definition} We define $D^*_\Gamma (X)$ as the space of adjointable operators on $\maE_{S, \Gamma}$ which are lifts of operators of the form $T\otimes I$ with $T\in D^*_H(X)$, in the sense of Definition (\ref{lift}).
\end{definition}

\begin{lemma}\label{canon}
For any $T\in D^*_H(X)$, there exist lifts of the operator $T\otimes I$ on $H_\Gamma$ (which belong to $D^*_\Gamma (X)$).
\end{lemma}

\begin{proof}\ We give a standard construction of a lift for a given $T$, see for instance \cite{FrankLarson}.
Fix $T\in D^*_H(X)$ and let $\{U_i\}_{i\in I}$ be a good finite open cover of $X$. So we assume that
each intersection  $U_{ij}=U_i\cap U_j$ is  connected  and
that there exist continuous sections $\psi_i: U_i\rightarrow \pi^{-1}U_i$ over each $U_i$.

Denote by $g_{ij}: U_{ij}\rightarrow \Gamma$ the deck transformation  $\psi_i(U_{ij})\rightarrow \psi_j(U_{ij})$. If $U_{ijk}=U_i\cap U_j\cap U_k\neq \emptyset$, then the following relation is satisfied for any $x\in U_i\cap U_j\cap U_k$:
$$
g_{ij}(x) g_{jk}(x)=g_{ik}(x).
$$

We also have the relation $\psi_i(x) g_{ij}(x)=\psi_j(x)$ for any $x\in U_{ij}$. Let now $\{\phi^2_i\}_{i\in I}$ be a partition of unity subordinate to the cover $\{U_i\}_{i\in I}$. Then if $\tilde{x}=\psi_k(x)g$, we set
$$
\xi_j(\tilde{x})= \phi_j(x)\delta_{g^{-1}} \star \delta_{g_{kj}}.
$$
It is easy to check that if  $k'$ is another index such that $\tilde{x}=\psi_{k'}(x)g'$, then we have $
\xi_j(\psi_{k'}(x)g')= \xi_j(\psi_k(x)g)$
(notice  that $g_{kk'}g'=g$).  One checks that
\begin{equation}
\label{generators}
[\xi_j.<\xi_j,\xi>](\tilde{x},\gamma)=\phi_j^2(x)\xi(\tilde{x},\gamma) \text{ so that} \sum_{j\in I} \theta_{\xi_j,\xi_j}= Id_{\maE_{X,\Gamma}}.
\end{equation}
To end the proof we define $\widehat{T}_0 =\sum_{i\in I} L_{\xi_i}(T\otimes I)L^{*}_{\xi_i}$. It is then easy to see that  $\widehat{T}_0$  provides a lift of $T$ once conjugated by the isomorphism $\Psi$ of Proposition (\ref{HMiso1}).

\end{proof}

Since $D^*_H(X)$ is a $*$-algebra, so is $D^*_\Gamma (X)$ and it contains the space $\maK(\maE_{{S},\Gamma})$ of $C^*\Gamma$-compact operators on $\maE_{S, \Gamma}$. Moreover, $D^*_\Gamma (X)$ is uniformly closed in $\maL(\maE_{S, \Gamma})$ and is thus a $C^*$-algebra, see \cite{HigsonRoe2010}.

We also define $Q^*_\Gamma(X)$ as the quotient $C^*$-algebra of $D^*_\Gamma(X)$ with respect to the closed two-sided $*$-ideal $\maK(\maE_{{S},\Gamma})$. We have $Q^*_\Gamma(X)\cong Q^*_H(X)$, see  \cite{HigsonRoe2010}[Lemma 5.5], and we thus end up with the short exact sequence:
$$
0\to \maK(\maE_{{S},\Gamma})\longrightarrow  D^*_\Gamma(X)\longrightarrow Q^*_H(X)\to 0.
$$
Passing to the $K$-theory of $C^*$-algebras, we deduce  a long exact sequence, actually a six-term periodic exact sequence of abelian groups:
$$
\cdots \longrightarrow  K_0(X) \longrightarrow K_0(C^*\Gamma) \longrightarrow  K_0(D^*_\Gamma(X)) \longrightarrow K_1(X) \longrightarrow K_1(C^*\Gamma) \longrightarrow  \cdots
$$
Here we have used  \cite{Paschke, HigsonRoe2010}
\begin{enumerate}
\item  Morita equivalence between the $C^*$-algebras $\maK(\maE_{{S},\Gamma})$ and $C^*\Gamma$,
\item The Paschke duality isomorphisms $
K_* (Q^*_H(X) ) \simeq K_{*+1} (X).$
\end{enumerate}

\begin{definition}
Let $X$ be a finite CW-complex. The Higson-Roe structure group of the pair $X, H$ is defined as
$$
\maS_{1,\Gamma}(X, H):= K_0(D^*_\Gamma(X, H)),.
$$
where $D^*_\Gamma(X, H)$ is the $C^*$-algebra defined above using the ample representation in the Hilbert space $H$.
\end{definition}
We notice that the $K$-theory group $\maS_{1,\Gamma}(X, H)$  does not depend on the choice of the ample representation $H$, see \cite{HigsonRoe2010}.

Let now $X$ be a closed subset of a compact topological space $Y$, denote by $\iota$ the inclusion map $X \xhookrightarrow{\iota} Y$ and let $u:Y\rightarrow B\Gamma$ be a continuous map. Let as usual $\iota^*: C(Y)\rightarrow C(X)$ be the restriction map.  Let $\Xi_\Gamma^Y$ be the Mishchenko bundle on $Y$ defined as before for the $\Gamma$-covering over $Y$ associated with the map $u$. In the same way, using $u\circ \iota$ we define the Mishchenko bundle $\Xi_\Gamma^X$ over $X$. Let $\pi_X:C(X)\rightarrow B(H)$ be an ample representation, and set $\pi_Y:= \pi_X\circ \iota^*$. $\pi_X$ (resp. $\pi_Y$) induce representations of $C(X)\otimes C^*\Gamma$ (resp. $C(Y)\otimes C^*\Gamma$) on $H\otimes C^*\Gamma$ by tensoring with the identity on $C^*\Gamma$, which we continue to denote by $\pi_X$ (resp. $\pi_Y$).
Since $\iota^*$ is surjective, there are  isomorphisms
$D^*_H(X)\cong D^*_H(Y)$ and $Q^*_H(X)\cong Q^*_H(Y)$.

Notice that $\pi_Y$  may not be ample in general, so let $\pi'_Y$ be an ample representation of $C(Y)$ on a second separable Hilbert space $H'$ and consider the representation $\pi_Y\oplus \pi'_Y$ on the orthogonal direct sum $H\oplus H'$. This representation is ample, and we get $i^1_*:D^*_H(Y)\hookrightarrow D^*_{H\oplus H'}(Y)$ and $i_*^2: Q^*_H(Y)\hookrightarrow Q^*_{H\oplus H'}(Y)$ by extending an operator $T \in D^*_H(Y)$ by zero on $H'$. Then the following commutative diagram summarizes the situation:

\begin{equation}\label{Func1}
\xymatrix{
0\to  K(H)  \ar[r]^{}\ar[d]^{} &D^*_{H}(X) \ar[d]^{i^1_*}\ar[r]_{} &  Q^*_H(X) \ar[d]^{i^2_*} \to 0\\
0\to   K(H\oplus H') \ar[r]^{}& D^*_{H\oplus H'}(Y) \ar[r]_{} &  Q^*_{H\oplus H'}(Y) \to 0\\
}
\end{equation}

Choosing a completely positive section $\sigma: C(X)\rightarrow C(Y)$ of $\iota^*$ we deduce a restriction map $r: \maE_{Y,\Gamma}\rightarrow \maE_{X,\Gamma}$. It is then easy to check that  $r$
induces an isometric isomorphism of Hilbert $C^*\Gamma$-modules $\Psi_{X,Y}: \maE_{Y,\Gamma}\otimes_{C(Y)\otimes C^*\Gamma}(H\otimes C^*\Gamma)\rightarrow \maE_{X,\Gamma}\otimes_{C(X)\otimes C^*\Gamma}(H\otimes C^*\Gamma)$.

For $\hat{T}_X\in D^*_\Gamma(X, H)$, we define $\iota_*: D^*_\Gamma(X,H)\longrightarrow D^*_\Gamma(Y,H)$ by
$$
\iota_*(\hat{T}_X)= \Psi_{X,Y}^{-1}\hat{T}_X \Psi_{X,Y}.
$$

Using the relation $\Psi_{X,Y} L_{\xi}= L_{r(\xi)}$ one checks that $\iota_*$ is well-defined, i.e. for $\hat{T}_X\in D^*_\Gamma(X)$ a lift of $T\in D^*_H(X)$ and $\xi_Y\in \maE_{Y,\Gamma}$ one has
$$
 \iota_*(\hat{T}_X) L_{\xi_Y}-L_{\xi_Y} (T\otimes I) \in \maK(H_\Gamma,\maE_{Y,\Gamma}\otimes_{C(Y)\otimes C^*\Gamma} (H\otimes C^*\Gamma)).
   $$
It is again not difficult  to show that $\iota_*$ is an isomorphism.

We can now describe the  functoriality maps corresponding to inclusions.
Using the ample representation $\pi'_Y$ of $C(Y)$ on the separable Hilbert space $H'$, we  get  a map $i: \maE_{Y,\Gamma}\otimes_{C(Y)\otimes C^*\Gamma} (H\otimes C^*\Gamma)\rightarrow \maE_{Y,\Gamma}\otimes_{C(Y)\otimes C^*\Gamma} ((H\oplus H')\otimes C^*\Gamma)$ and the image of $i$ is an orthocomplemented submodule. Let $i^1_*: D^*_\Gamma(Y, H)\hookrightarrow D^*_\Gamma(Y,H\oplus H')$ be the inclusion. The composition $i^1_*\circ \iota_*$ is then our required functoriality map $D^*_\Gamma(X,H)\rightarrow D^*_\Gamma(Y, H\oplus H')$.

\begin{definition}

The Higson-Roe structure group of $\Gamma$ is defined as
$$
\maS_1(\Gamma) := \lim_{\to} \maS_{1,\Gamma}(X, H),
$$
where the inductive limit is taken over the compact subspaces of $B\Gamma$ with ample representations in separable Hilbert spaces, using the previous construction of $D^*_\Gamma(X,H)\rightarrow D^*_\Gamma(Y, H\oplus H')$.
\end{definition}

\subsection{$\ell^2$ Analytic structure group}\label{ell2structure}

We next denote by $\maM_X$  the von Neumann algebra $B(L^2(\tilde{X},\tilde{S}))^\Gamma$ of bounded $\Gamma$-invariant operators, and  which was first studied in this geometric setting by Atiyah \cite{AtiyahCovering}. This von Neumann algebra is naturally endowed with a semifinite {\em{normal faithful}} positive trace $\tau$. The trace $\tau$ can be defined using the characteristic function $\chi=\chi_F$ of a fundamental domain $F$ in $\tX$ as follows. If $T\in \maM_X$ is nonnegative then
$$
\tau (T) := \Tr ( M_\chi  T M_\chi) \text{ with } M_\chi \text{ the multiplication operator by $\chi$ in } L^2 (\tX, \tS).
$$
In $\maM_X$, there is the bilateral closed $*$-ideal   $\maK(\maM_X,\tau)$ of  $\tau$-compact operators. See for instance \cite{Benameur03} for the background definitions and properties. The von Neumann algebra $\maM_X$ is isomorphic to $B(L^2(X, S))\otimes \maN\Gamma$, where $\maN\Gamma$ is the group von Neumann algebra of $\Gamma$ \cite{Dixmier, AtiyahCovering}. The trace $\tau$ is then identified with $\Tr\otimes \tau_e$ where $\Tr$ is the usual trace on the Hilbert space $L^2(X, S)$ and $\tau_e$ is the finite trace of $\maN\Gamma$ induced by evaluation at the neutral element $e$. More generally, if $H$ is any ample separable Hilbert space representation of $C(X)$, then we define the semi-finite von Neumann algebra $\maM_{X, H} := B(H)\otimes \maN\Gamma$ in the same way with the induced representation of $C(X)$. The trace will then still be denoted $\tau$ for simplicity.

Our goal now is to define a  $C^*$-algebra $D^*_{(2)}(X, H)$ whose $K$-theory groups will not depend on $H$ and which fits into the  short exact sequence:
$$
0\to \maK(\maM_{X, H},\tau)\oplus K(H)\longrightarrow D^*_{(2)}(X, H)\longrightarrow Q^*_H(X)\to 0
$$

\begin{definition}
 We denote by $D^*(\maM_{X, H},\tau)$ the space of operators $T$ in the von Neumann algebra $\maM_{X, H}$ which satisfy the following:
$$
[T,f] = Tf-fT\quad \in \maK(\maM_{X, H},\tau), \quad \forall f\in C(X).
$$
\end{definition}

\begin{remark}
For an ample representation of $C(X)$ on a general separable Hilbert space $H$, we define isomorphisms $\Psi^H_{reg}: \maE_{X,\Gamma}\otimes_{C(X)\otimes C^*\Gamma}(H\otimes C^*\Gamma)\otimes_{\pi_{reg}}\ell^2\Gamma\rightarrow H\otimes \ell^2\Gamma$ as follows. Let $\{\phi_i\}_{i\in I}$ be a partition of unity on $X$ subordinate to a covering $\{U_i\}_{i\in I}$. For the element $\xi_i \in \maE_{X,\Gamma}$ defined in Lemma (\ref{canon}), $h\in H, \phi\in \C\Gamma, \psi\in \ell^2\Gamma$, define
$$\Psi_{reg}^H(\xi_i\otimes (h\otimes \phi)\otimes_{\pi_{reg}}\psi)=\phi_i.h\otimes [\pi_{reg}(\phi)](\psi)$$
Note that this completely characterizes $\Psi_{reg}^H$, since the space of elements $\xi$ of $\maE_{X,\Gamma}$ which can be written in the form
$$\xi= \sum_{i}\xi_i.(\alpha_i\otimes \beta_i)$$
where $\alpha_i\in C(X)$ and $\beta_i\in \C\Gamma, i\in I$, is dense in $\maE_{X,\Gamma}$.
Similarly, one defines an isomorphism $\Psi_{av}^H: \maE_{X,\Gamma}\otimes_{C(X)\otimes C^*\Gamma}(H\otimes C^*\Gamma)\otimes_{\pi_{av}}\C\rightarrow H$: for the element $\xi_i$ as above, $h\in H, \phi\in \C\Gamma$, define
$$\Psi_{av}^H(\xi_i\otimes (h\otimes \phi)\otimes_{\pi_{av}}1)=(\pi_{av}(\phi)).( \phi_i).h $$

\end{remark}

We use in the rest of this subsection the above isomorphisms $\Psi^H_{\reg}$ and $\Psi^H_{\av}$, dropping the superscript $H$ for notational convenience.

\begin{definition}
\label{D2}
We  associate with the representations $\pi_{\reg}:\maN\Gamma \to B(\ell^2\Gamma)$ and $\pi_{\av}: \maN\Gamma \to \C$ the $\ell^2$ structure algebra $D^*_{(2)}(X, H)$ composed of couples $(T_1,T_2) \in D^*(\maM_{X, H},\tau)\oplus D^*_H(X)$ such that for a lift $\hat{T}_2\in D^*_\Gamma(X, H)$ of $T_2$ we have
$$
\Psi_{\reg}(\hat{T}_2\otimes_{\pi_{\reg}} Id)\Psi_{\reg}^{-1}-T_1 \quad \in  \maK(\maM_{X, H},\tau)
$$
\end{definition}

The following lemma explains the above definition of $D^*_{(2)}(X, H)$. It shows the relation between a lift $\hat{T}$ of an operator $T \in D^*_H(X)$  and conjugation of $T$ by the isomorphism $\Psi_{\av}$.

\begin{lemma}
Let $\hat{T}\in D^*_\Gamma(X, H)$ be a lift of $T\in D^*_H(X)$. Then $\Psi_{\av} (\hat{T}\otimes_{\pi_{\av}} Id) \Psi_{\av}^{-1} -T $ is a compact operator in $H$.
\end{lemma}

\begin{proof}
Set $
{\hat T}':=\Psi^{-1} {\hat T}  \Psi : \maE_{X, \Gamma}\otimes_{C(X)\otimes C^*\Gamma} (H\otimes C^*\Gamma) \longrightarrow \maE_{X, \Gamma}\otimes_{C(X)\otimes C^*\Gamma} (H\otimes C^*\Gamma).
$
We prove more precisely that there exists $\xi \in C(\tilde{X}, C^*\Gamma)^\Gamma$ and operators $A$ and $B$ such that
$$
\Psi_{\av}(\hat{T}\otimes_{\pi_{\av}}Id)  \Psi^{-1}_{\av} -  T = A \left[ {\hat T}'  L_\xi  -  L_\xi   (T\otimes 1)\right]  B.
$$
This will allow us to conclude using the definition of a lift.

For any $\xi \in C(\tilde{X}, C^*\Gamma)^\Gamma$, we first  define $\pi_{\av}(\xi) \in C(X)\simeq C(\tX)^\Gamma$ as the $\Gamma$-invariant function
$$
\pi_{\av}(\xi)({\tilde x})=\pi_{\av} (\xi ({\tilde x})).
$$
So in particular, when $\xi\in C(\tilde{X}, \C\Gamma)^\Gamma$,  then $\pi_{\av}(\xi)(x) = \sum_{g\in \Gamma} \xi (\tilde{x}, g)$.

Denote in analogy with the case $H=L^2(X, S)$ the action of $C(X)$ on $H$ by $M_{\bullet}$ and set $\Psi'_{\av}:=\Psi_{\av}  (\Psi\otimes_{\pi_{\av}} 1)$.
We  introduce the following operators:
\begin{itemize}
\item $L_1:H\otimes C^*\Gamma\rightarrow H$ given  for $h\otimes \phi\in H\otimes C^*\Gamma$ by $
L_1(h\otimes \phi):= \pi_{\av}(\phi)h.$
\item $R_{\delta_e}: H\rightarrow H\otimes C^*\Gamma$ as $R_{\delta_e}(h)=h\otimes \delta_e$ for $h\in H$.
\item $R_{1}: \maE_{X, \Gamma}\otimes_{C(X)\otimes C^*\Gamma} (H\otimes C^*\Gamma)\rightarrow (\maE_{X, \Gamma}\otimes_{C(X)\otimes C^*\Gamma} (H\otimes C^*\Gamma))\otimes_{\pi_{\av}} \C$ given  for $\xi \in \maE_{X, \Gamma}, h\in H, \phi\in C^*\Gamma$ by
$$
R_1(\xi\otimes (h\otimes \phi))= \xi\otimes(h\otimes \phi)\otimes_{\pi_{\av}} 1
$$
\end{itemize}

Then straightforward verifications show that we have:
$$
L_1R_{\delta_e} = id_H, ({\hat T}' \otimes_{\pi_{\av}} 1)  R_1 = R_1 {\hat T}'
\text{ and } L_1 (T\otimes 1)  R_{\delta_e} = T,
$$
and also that $ \Psi'_{\av}  R_1 L_\xi R_{\delta_e} = M_{\pi_{\av}(\xi)}$,  for any $\xi \in C(\tilde{X}, \C\Gamma)^\Gamma$,

Hence, we compute
\begin{eqnarray*}
\Psi_{\av} (\hat{T}\otimes_{\pi_{\av}}Id)  \Psi^{-1}_{\av} M_{\pi_{\av}(\xi)} - M_{\pi_{\av}(\xi)}  T&=& \Psi_{\av} (\hat{T} \Psi \otimes Id) R_1 L_\xi  R_{\delta_e} - M_{\pi_{\av}(\xi)}  L_1  (T\otimes 1)  R_{\delta_e}\\
&=& \Psi'_{\av} R_1  \left[ {\hat T}'  L_\xi  -  L_\xi  R_{\delta_e}  L_1    (T\otimes 1)\right]  R_{\delta_e}\\
&=& \Psi'_{\av} R_1  \left[ {\hat T}'  L_\xi  -  L_\xi    (T\otimes 1)\right] R_{\delta_e}
%\left[ \Psi'_{\av}\circ R_1 \circ (\hat{T}' \circ L_\xi- M_{\pi_{\av}(\xi)} T \right] \circ R_{\delta_e}\\
%&=& \Psi_{\av}\circ R_1\circ[\hat{T}\circ  L_\xi-L_\xi\circ R_{\delta_e}\circ L_1(T\otimes Id)]\circ R_{\delta_e}\\
%&=& \Psi_{\av}\circ R_1\circ[\hat{T}\circ L_\xi-L_\xi\circ (T\otimes Id)]\circ R_{\delta_e}\\
\end{eqnarray*}
The last equality is a consequence of the relation $
\Psi'_{\av}  R_1 L_\xi R_{\delta_e} L_1 =  \Psi'_{\av} R_1 L_\xi,$
which in turn is a consequence of the identity $\pi_{\av} (\xi \star \phi) = (\pi_{\av}\phi) (\pi_{\av}\xi)$.

Since ${\hat T}'  L_\xi  -  L_\xi    (T\otimes 1)$ is a $C^*\Gamma$-compact operator between the Hilbert modules  $H_\Gamma$ and $\maE_{S, \Gamma}$, we leave it as an exercise to check that then the operator
$$
 \Psi'_{\av} R_1  \left[ {\hat T}'  L_\xi  -  L_\xi    (T\otimes 1)\right]  R_{\delta_e}
$$
belongs to $K(H)$. The proof is complete if we notice that there exists $\xi\in \maE_{X, \Gamma}$ such that $\pi_{\av}(\xi)=1$. Indeed due to the properness of the action of $\Gamma$ on $\tX$, there exists a smooth compactly supported function $\varphi$ on $\tX$ such that $\sum_{g\in \Gamma} g\varphi = 1$. Hence if we set
$$
\xi := \sum_{g\in \Gamma} g\varphi\otimes \delta_g, \quad i.e. \;\; \xi ({\tilde x}, g) := \xi ({\tilde x} g),
$$
then $\xi \in C(\tX, \C\Gamma)$ and it is $\Gamma$-equivariant. Moreover, we have   $\pi_{\av}(\xi)=1$.

\end{proof}

\begin{proposition}
\label{SES1}\
\begin{enumerate}
\item $D^*_{(2)}(X, H)$ is a $C^*$-algebra.
\item  There is a short exact sequence of $C^*$-algebras
$$
0\to \maK(\maM_{X, H},\tau)\oplus K(H)\stackrel{i}{\longrightarrow}  D^*_{(2)}(X, H)\stackrel{p}{\longrightarrow} Q^*_H(X, H)\to 0.
$$
\end{enumerate}
\end{proposition}

\begin{proof}\

\begin{enumerate}
\item   $D^*_{(2)}(X, H)$ is clearly a $*$-algebra. Consider a sequence of operators $(T_{1,n},T_{2,n})_{n\geq 0}\in D^*_{(2)}(X, H)$ which converges in the uniform topology to  $(T_1, T_2)$. Then we use the lift construction from Lemma (\ref{canon}) and set $ \hat{T}_{2,n}:= \sum_i L_{\xi_i} (T_{2,n}\otimes I) L^*_{\xi_i}$. Then $(\hat{T}_{2,n})_{n\geq 0}$ converges  in norm to $ \hat{T}_2:= \sum_i L_{\xi_i} (T_{2}\otimes I) L^*_{\xi_i}$. For $n\geq 0$, we have by easy verification:
$$
\Psi_{\reg}(\hat{T}_{2,n}\otimes Id) \Psi_{\reg}^{-1}- T_{1,n} \in \maK(\maM_X,\tau).
$$

This proves that $D^*_{(2)}(X, H)$ is a $C^*$-algebra.

\item It is  clear that $p\circ i=0$ and that the map $i$ is injective. Fix now  $(T_1,T_2)\in D^*_{(2)}(X, H)$ and recall that $p(T_1,T_2)=[T_2]\in Q^*_H(X).$
For any representative $T \in D^*_H(X)$ of a class $[T]\in Q^*_H(X)$ we consider some lift $\hat{T}$ of $T$ and deduce that the element $\tilde{T}=(\Psi_{\reg}(\hat{T}\otimes Id) \Psi_{\reg}^{-1}, T)$ belongs to $D^*_{(2)}(X, H)$ and satisfies $p(\tilde{T})=[T]$. Thus $p$ is surjective.
It is also clear that $p\circ i=0$.
Now, if $p(T_1,T_2)=[0] \in Q^*_H(X)$, then $T_2\in K(H)$. Since for any lift $\hat{T}_2$ of $T_2$, we have $\Psi_{\reg}(\hat{T}_2\otimes Id) \Psi_{\reg}^{-1} \in \maK(\maM_{X, H},\tau)$, this in turn implies that $T_1 \in \maK(\maM_{X, H},\tau)$.
\end{enumerate}
\end{proof}

Assume now that we have an embedding  $\iota: X\hookrightarrow Y$ of compact spaces,
then as in the discussion in subsection \ref{HRanalytic}, the map $\iota$ induces an isomorphism  $D^*_{(2)}(X, H)\cong D^*_{(2)}(Y, H)$ (where, as before we have included the dependence on the Hilbert space $H$ in the notation). The following diagram commutes as a consequence of Proposition \ref{HMisoreg}:

\[
\begin{CD}
\maE_{Y,\Gamma}\otimes_{C(Y)\otimes C^*\Gamma}(H\otimes C^*\Gamma)\otimes_{\pi_{\reg}}\ell^2(\Gamma)
@> \Psi_{X,Y}\otimes Id >>
  \maE_{X,\Gamma}\otimes_{C(X)\otimes C^*\Gamma}(H\otimes C^*\Gamma)\otimes_{\pi_{\reg}}\ell^2(\Gamma) \\
  @V \Psi^Y_{\reg} VV
@VV \Psi^X_{\reg}  V\\
  % \arrow{ur}{\psi} \arrow[swap]{d}{\chi} \arrow{d}{\Psi} \\
H\otimes \ell^2(\Gamma)
@>\ \  =\ \
>>H\otimes \ell^2(\Gamma)
\end{CD}
\]
Therefore we have the relation
 \begin{equation}
 \label{XY}
 \Psi^Y_{\reg}(\iota_*(\hat{T}_X)\otimes Id)(\Psi^Y_{\reg})^{-1} = \Psi^X_{\reg}(\hat{T}_X\otimes Id) (\Psi^X_{\reg})^{-1}
 \end{equation}

So we get  a  well-defined isomorphism of $C^*$-algebras $(id_{\maM_{X,H}}, id_{B(H)}): D^*_{(2)}(X,H)\xrightarrow{\cong} D^*_{(2)}(Y,H)$  according to Definition \ref{D2} and using Equation \ref{XY}.
Observe also that if $H'$ is a second representation of $C(Y)$ which is ample, then the inclusion $H\hookrightarrow H\oplus H'$ implies that the von Neumann algebra $\maM_{Y,H}:= B(H)\otimes \maN\Gamma$ embeds as a corner in the von Neumann algebra $\maM_{Y, H\oplus H'}:= B(H\oplus H')\otimes \maN\Gamma$, and thus induces a map $D^*_{(2)}(Y, H)\hookrightarrow D^*_{(2)}(Y, H\oplus H')$. Using this inclusion of $C^*$-algebras, the composition map $D^*_{(2)}(X,H)\xrightarrow{\cong} D^*_{(2)}(Y,H)\hookrightarrow D^*_{(2)}(Y,H\oplus H')$ is again our required functoriality map.
\begin{definition}
Define the $\ell^2$-analytic structure group of $X$ as $
\maS_1^{(2)}(X):= \maS_1^{(2)}(X, H) :=K_0(D^*_{(2)}(X, H)),$ for any ample representation $H$.
\end{definition}

\begin{remark}
We note that the $K$-theory groups $S^{(2)}_1(Y, \bullet)$ are again independent of the choice of ample representation up to isomorphism, so with the construction given above we end up with $i_*: S^{(2)}_1(X)\rightarrow S^{(2)}_1(Y)$.
\end{remark}

We are now in position to define the analytic $\ell^2$ structure group $\maS_1^{(2)}(\Gamma)$ associated with $\Gamma$.
\begin{definition}\
%\begin{itemize}
%\item Define the $L^2$-analytic structure group of $X$ as $
%\maS_1^{(2)}(X):= K_0(D^*_{(2)}(X, H)),$ for any ample representation $H$.
%\item
$\maS_1^{(2)}(\Gamma):= \underset{X\subset B\Gamma}{\varinjlim} \maS_1^{(2)}(X),$
where the direct limit is taken over finite CW subcomplexes $X$ of $B\Gamma$ with ample representations.
%\end{itemize}
\end{definition}

\subsection{Compatibility of the exact sequences}

Let us give some remarks on the relation between the $C^*$-algebras $D^*_{\Gamma}(X, H)$ and $D^*_{(2)}(X, H)$. For $\hat{T}\in D^*_\Gamma(X, H)$, we set
$$
\alpha(\hat{T}):= \left(\Psi_{\reg}(\hat{T}\otimes_{\reg}Id)\Psi^{-1}_{\reg},\Psi_{\av}(\hat{T}\otimes_{\av}Id)\Psi^{-1}_{\av}\right) \quad \in D^*_{(2)}(X, H).
$$
If $T$ is a compact operator on the Hilbert space $H$, any  lift $\hat{T}$ of $T$ is a compact operator on the Hilbert module $\maE_{H,\Gamma}$ and thus the element $\alpha(\hat{T})$ belongs to $\maK(\maM_{X, H} ,\tau)\oplus K(H)$. The Hilbert module $\maE_{H,\Gamma}$ is defined exactly as $\maE_{S,\Gamma}$ replacing everywhere $L^2(X, S)$ by the Hilbert space $H$. Therefore the following diagram commutes:

\hspace{2cm}
\begin{equation}\label{SES_HR_BR}
\xymatrix{
0\to  \maK(\maE_{H,\Gamma}) \ar[r]^{}\ar[d]^{\alpha} &D^*_\Gamma(X, H)\ar[d]^{\alpha}\ar[r]_{} &  Q^*_\Gamma(X, H) \ar[d]^{\cong} \to 0\\
0\to  \maK(\maM_{X, H},\tau)\oplus K(H) \ar[r]^{}& D^*_{(2)}(X, H) \ar[r]_{} &  Q^*_H(X) \to 0
}
\end{equation}
Since $\hat{T}$ is unique up to compact operators there is an isomorphism
$Q^*_\Gamma(X, H)\cong Q^*_H(X)$. Morever, $C^*\Gamma$ is Morita-equivalent to
$\maK(\maE_{H,\Gamma})$ and therefore their $K$-theory groups are isomorphic. Hence, as in the type I case,  the diagram \eqref{SES_HR_BR} induces a commutative diagram between the long exact sequences in $K$-theory. In particular, the following commutes:

%CHANGE THE FOLLOWING TO $B\Gamma$ INSTEAD OF $X$.

\vspace{0,2cm}
\begin{equation}\label{alpha}
\hspace{-0,45cm}
\begin{CD}
K_0(X) @>{\mu_\Gamma}>>  K_0(C^*\Gamma) @>>>  S_{1,\Gamma}(X) @>>> K_1(X)  @>>> K_1(C^*\Gamma) \\
@VV{=}V  @V{\alpha_*}VV                          @V{\alpha_*}VV       @V{=}VV                      \\
K_0(X) @>{\partial}>>  K_0(\maK(\maM_{X,H},\tau)\oplus K(H)) @>>>  S_{1}^{(2)}(X) @>>> K_1(X)  @>>> 0 \\
\end{CD}
\end{equation}

The map $\alpha_*$ is induced by the map $\alpha$. Moreover, the traces induce additive maps
$$
\tau_*: K_0(\maK (\maM_{X, H}, \tau)) \longrightarrow \R \text{ and } \Tr_*:K_0 (\maK (H))\longrightarrow \Z.
$$
The map $\Tr_*$ is a group isomorphism and so is the first additive map $\tau_*$. The latter statement is standard and we have given some details in  the appendix, see Lemma \ref{K}. On the other hand, the regular and average traces, $\tau_{\reg}$ and $\tau_{\av}$,  on the group $C^*$-algebra $C^*\Gamma$ also induce additive maps
$$
\tau_{reg, *} : K_0(C^*\Gamma) \longrightarrow \R \text{ and } \tau_{av, *} : K_0(C^*\Gamma) \longrightarrow \Z.
$$
We point out as an easy consequence of the Atiyah-Singer $\ell^2$ index theorem for Galois coverings that the range of $\tau_{reg, *}$ restricted to the Baum-Connes assembly map (and hence also to the Kasparov assembly map) coincides with the group $\Z$ of the integers. On the other hand, if we use the Morita equivalence, then  the following compatibility relation always holds:
$$
(\tau_*, \Tr_*)  \alpha_* = (\tau_{reg, *}, \tau_{av, *}).
$$

\begin{definition}
We denote by $\Lambda$ the image of the additive map $\tau_{reg, *}: K_0(C^*\Gamma) \longrightarrow \R$. We define the morphism
$\delta: \R\oplus \Z \longrightarrow \R$ by $\delta (x,n):=x-n$. So, there is an exact sequence of abelian groups
$$
0\to \Z \longrightarrow \Lambda\oplus \Z \stackrel{\delta}{\longrightarrow} \Lambda \to 0,
$$
where $\Z\hookrightarrow \Lambda\oplus \Z$ is the diagonal map.
\end{definition}

\begin{remark}
Notice that $\Z\subset \Lambda$ and if the $K_0$ Baum-Connes map is surjective, then $\Lambda = \Z$.
\end{remark}

As a consequence of the above definitions and the commutativity of the diagram (\ref{alpha}), by passing to the direct limits we obtain the following commutative diagram:

\vspace{0,2cm}
\begin{equation}\label{beta}
\hspace{-0,47cm}
\begin{CD}
K_0(B\Gamma) @>{\mu_\Gamma}>>  K_0(C^*\Gamma) @>>>  \maS_{1}(\Gamma) @>>> K_1(B\Gamma)  @>>> K_1(C^*\Gamma) \\
@V{=}VV  @V{(\tau_*,\Tr_*)\circ\alpha_*}VV                          @V{\alpha_*}VV       @V{=}VV                         \\
 K_0(B\Gamma) @>{\partial}>>  \R\oplus \Z @>>>  \maS_{1}^{(2)}(\Gamma) @>>> K_1(B\Gamma)  @>>> 0  \\
\end{CD}
\end{equation}
\vspace{1cm}

\begin{lemma}
Assume that $(e_1,e_2)$ and $(f_1, f_2)$ are idempotents from $\maK(\maM_{X, H}, \tau)\oplus \maK(H)$ such that
$$
\tau
(e_1) - \tau
(f_1) =  \Tr (e_2)- \Tr (f_2).
$$
Then the image of the class $[e_1,e_2] - [f_1, f_2]$ in $\maS_1^{(2)}(\Gamma)$ is trivial.
\end{lemma}

\begin{proof}
The class $([e_1]-[f_1], [e_2] - [f_2])\in K_0(\maK(\maM_{X, H}, \tau
)\oplus \maK(H))$ is thus identified, through the isomorphism $(\tau_{*}, \Tr_*)$,  with an element $(N,N)$ of the diagonal in $\R\oplus \Z$. Therefore, using the commutative diagram \eqref{alpha}, we deduce that there exists $u\in K_0(B\Gamma)$ such that
$$
([e_1]-[f_1], [e_2] - [f_2]) = (\alpha_*\circ \mu_\Gamma) (u).
$$
Therefore, exactness of the Higson-Roe sequence implies that the image of $([e_1]-[f_1], [e_2] - [f_2])$ in $\maS_1^{(2)}(\Gamma)$ is trivial.
\end{proof}

\begin{corollary}
For any $(x, n)\in \R\oplus \Z$, the class in $\maS_1^{(2)}(\Gamma)$ of any pair $(e_1, e_2)\in \maK(\maM_{X, H}, \tau
)\oplus \maK(H)$ satisfying $
\tau
(e_1)=x$ and $\Tr (e_2) = n$
only depends on  the difference $\delta (x,n):=x-n$.
\end{corollary}

\begin{definition}
For $x\in \R$, we define its class $[x] \in K_0(D^*_{(2)}(X, H))$ as the $K$-theory class  $[(e_1,e_2)] $ of any pair of projections $(e_1, e_2)\in  \maK(\maM_{X,H}, \tau
)\oplus \maK(H)$ such that
$$
\tau
(e_1)-\Tr(e_2)=x \hspace{2cm} (*)
$$
So, we have the equality $[x,n]= [x-n]$ in $\maS_1^{(2)} (X)$.
\end{definition}

Recall that the trivial representation of $\Gamma$ induces the index morphism $K_0(C^*\Gamma)\to \Z$. Composing with the assembly map $\mu_\Gamma$ we get the group morphism $\Ind_{B\Gamma}: K_0(B\Gamma)   \to \Z$ which corresponds roughly to the pairing with the trivial line bundle.

\begin{corollary}
\label{SES2}
The maps $(x, n) \mapsto [x,n]$ and $x\mapsto [x]$ fit into the commutative diagram
\vspace{0,2cm}
\hspace{0,5cm}
\begin{equation}\label{SS}
\begin{CD}
\cdots  @>>> K_0(B\Gamma)   @>{\mu_\Gamma}>> K_0(C^*\Gamma)@>>>   \maS_1 (\Gamma)  @>>> K_1(B\Gamma) @>{\mu_\Gamma}>>  K_1(C^*\Gamma)\cdots \\
  @.    @V{\Ind_{B\Gamma}}VV@V{(\tau_{*},\Tr_{*})}V{\circ \alpha_*}V  @V{\alpha_*}VV @V{=}VV\\
0 @>>> \mathbb{Z} @>>> \R\oplus \Z @>>> \maS_1^{(2)}(\Gamma) @>>>  K_1(B\Gamma)@>>>  0\hspace{1cm}\\
@. @VVV @V\delta VV @V{=}VV          @V{=}VV  @.\\
   @. 0 @>>> \R @>>>  \maS_1^{(2)}(\Gamma) @>>>  K_1(B\Gamma)@>>>  0\hspace{1cm}\\
\end{CD}
\end{equation}
\vspace{0,5cm}
 In particular, we have the following short exact sequence of abelian groups
 $$
0 \to \R \longrightarrow \maS_1^{(2)} (\Gamma) \longrightarrow K_1(B\Gamma) \to 0
$$
\end{corollary}

\section{Geometric structures and eta invariants}\label{geometric}

In this section we define an $\ell^2$ geometric  structure group, denoted $\maS_1^{(2), geo}(\Gamma)$. We then prove that this group is  isomorphic  to the $\ell^2$ analytic structure group $\maS_1^{(2)} (\Gamma)$ defined in the previous section. This allows us to show that  the relative Cheeger-Gromov eta invariant belongs to the range of  a group morphism from the Higson-Roe structure group $\maS_1 (\Gamma)$ to the reals.

\subsection{Review of Eta invariants and APS formulae}

Let us briefly recall the definition of the eta invariant and the $\ell^2$ eta invariant, as well as the Atiyah-Patodi-Singer theorem for even dimensional manifolds with boundary.
Given a generalized Dirac operator $D$  on a closed
oriented manifold $Y$ of dimension $2m-1$,
the {\em eta-function} of $D$ was defined in \cite{APS2} as
$$
\eta(s,D):=\Tr'(D|D|^{-s-1}),
$$
where $\Tr'$ stands for the trace restricted to the subspace orthogonal to
$\Ker(D)$. By \cite{APS1, APS2, APS3}, $\eta(s,A)$
is holomorphic when $\Re(s)>2m-1$ and can
be extended meromorphically to the entire complex plane with possible simple
poles only.
The eta function is then
known to be holomorphic  at $s=0$  \cite{APS3}.
The {\em eta-invariant}
of $A$ is then defined as $
\eta(A) = \eta(0, A).$
The eta invariant is related
to the heat kernel by a Mellin transform. More precisely, it is known that the integral  $
\frac{1}{\sqrt{\pi}}\int_0^\infty \Tr(De^{-t{D}^2})\,dt/{\sqrt t}$ is absolutely convergent to $\eta (D)$.

Given a compact manifold $\hM$ with boundary $M$ together with an elliptic generalized Dirac operator $\hD:C^\infty (E^0)\to C^\infty (E^1)$ and assuming that all structures are product-type near the boundary, Atiyah-Patodi-Singer proved a deep index theorem for global BVP. More precisely, assume that in a collar neighborhood of the boundary the operator $\hD$ has the form $\hD=\sigma (\partial_u + D)$, where $\partial_u$ is the inward normal vector field and $\sigma$ is a bundle isomorphism, then imposing the so-called global APS boundary condition, the resulting operator $\hD_{APS}$ has a well defined Fredholm index  and the APS formula computes this index as
$$
\Ind (\hD_{APS}) = \int_{\hM} \alpha_0(\hD)  - \frac{h+\eta (D)}{2},
$$
where $\alpha_0(\hD)$ is a local closed differential form and $h=\dim (\Ker D)$. We recall that the operator $\hD_{APS}$ is the operator $\hD$ acting from $\Dom (\hD_{APS})$ into  $L^2 (\hM, E^1)$, where
$$
\Dom (\hD_{APS})=\{u\in L^2 (\hM, E^0)\text{ such that } \hD (u)\in L^2 (\hM, E^1) \text{ and } \chi_{\geq} (D) (u\vert_M) = 0\}.
$$
Here and in the whole paper $\chi_{\geq}$ is the characteristic function of $[0, +\infty)$.

The Mellin expression of the eta invariant allowed Cheeger and Gromov to introduce the $\ell^2$ eta invariant for Galois coverings, see  \cite{CheegerGromov}. More precisely, given a Galois cover $\tM\to M$ over the closed odd dimensional manifold $M$, and a  $\Gamma$-invariant generalized Dirac operator $\tD: C^\infty (\tE^0) \to C^\infty (\tE^1)$ on $\tM$, the similar Mellin integral but using Atiyah's $\ell^2$ trace, is  absolutely convergent to a real number called the $\ell^2$-eta invariant of $\tD$, i.e. \cite{CheegerGromov}
$$
\eta_{(2)} (\tD) := \frac{1}{\sqrt{\pi}} \int_0^\infty \tau
 (\tD e^{-t{\tD}^2})\,dt/{\sqrt t} \quad\text{ is well defined.}
$$
The APS index theorem was then extended to Galois coverings in the PhD thesis of Ramachandran \cite{Ramachandran}. Assume that $\thM\to \hM$ is a Galois covering with boundary the Galois covering $\tM\to M$ and with obvious notations, one gets the similar formula under the similar assumptions (see \cite{Ramachandran}):
$$
\Ind_{(2)}(\thD_{APS}) = \int_{\hM} \alpha_0(\hD)  - \frac{{\tilde h}+\eta_{(2)} (\tD)}{2},
$$
with $\tilde h=\dim_{(2)}(\Ker \tD)$ is the $\ell^2$ dimension of the $\Gamma$-representation $\Ker \tD$.\\

For the reader's convenience we recall here the definition of a Clifford bundle  which enters in the definition of the version of geometric $K$-homology with oriented cycles (cf. \cite{KeswaniControlledPaths}, \cite{Guentner}, \cite{HigsonRoe2010}).

\begin{definition}
 Let $M$ be a closed orientable Riemannian manifold. A \emph{Clifford bundle} on $M$ is a smooth complex vector bundle $S$ endowed with a Hermitian metric such that
\begin{enumerate}
\item each fiber $S_x$ for $x\in M$ is a Clifford module of $T^*_xM$ (which is $\Z_2$-graded when $M$ is even dimensional), i.e. there is a linear map $c: T^*_xM\rightarrow \End(S_x)$ such that $c(\xi)^*=-c(\xi)$ and $c(\xi)^2=-|\xi|^2$, for all $\xi\in T^*_xM$.  We assume as it is customary that $c(\xi)$ is odd for the grading in the even dimensional case.
\item $S$ also carries a connection $\nabla^S:C^\infty(M, S)\rightarrow C^\infty(M,T^*M\otimes S)$ which is compatible with the grading on $S$ (in the even dimensional case) and with the Levi-Civita connection $\nabla$ on $T^*M$ in the following sense:
    $$
    \nabla_X^S(c(\omega) s)=c(\nabla_X\omega) s+ c(\omega) \nabla_X^Ss \quad \text{ for }s\in C^\infty(M,S), \omega\in C^\infty(M,T^*M) \text{ and } X\in C^\infty(M,TM).
    $$
\end{enumerate}

\end{definition}

As usual, the map $c$ will sometimes be referred to as a Clifford action.

\begin{remark}\label{Dirac}
 The above definition allows one to define a ``generalized" Dirac operator $D$ acting on smooth sections of $S$. In the even-dimensional case $D$ interchanges the even and odd-degree sections of $S$.\\
\end{remark}

\subsection{The geometric $\ell^2$ structure group}\label{Structure}

We first introduce the cycles for our $\ell^2$ geometric group. Recall from \cite{HigsonRoe2010},  that an odd geometric cycle for the classifying space $B\Gamma$ is a triple $(M,S,f)$, where $M$ is a smooth, oriented, closed Riemannian manifold  whose all connected components have odd dimension, $S$ is a Clifford bundle over $M$ and $f:M\to B\Gamma$ is a (homotopy class of a) continuous map. Notice that there is a modification of the geometric $K$-homology of Baum-Douglas, where the cycles are the Higson-Roe cycles \cite{KeswaniControlledPaths}, but this will not be used here. See also  \cite{BaumHigsonSchick} for more on the geometric Baum-Douglas $K$-homology groups.  By classical arguments, we may assume that the map $f$  is  classifying for a smooth Galois $\Gamma$-cover over $M$, whose total space  will be denoted in the sequel of this section by $\tM$, so without reference to (the homotopy class of) $f$. More generally, the lifts to the cover $\tM$ of objects on $M$ will be emphasized using a tilde.

As mentioned in Remark \ref{Dirac} above, associated with the Clifford bundle $S$ (called Dirac bundle in \cite{HigsonRoe2010}), there is a ``generalized'' Dirac operator $D$ on $M$ and its lift to $\tM$ that is denoted $\tD$, see for instance \cite{LawsonMichelson}. This is an elliptic first order differential operator $D$ whose commutator with the multiplication by a function $\varphi\in C^\infty (M)$ is the zero-th order differential operator which is Clifford multiplication in $S$ by the differential $1$-form $d\varphi$. Notice then that the same holds on the covering $\tM$ where $\varphi$ acts as a $\Gamma$-invariant operator as usual by multiplication. Notice also that such generalized Dirac operator satisfies the unique continuation property \cite{BW}.

More generally, we shall consider triples $({\hat M}, {\hat S}, {\hat f})$ for $B\Gamma$ where we now allow in addition ${\hat M}$ to be a compact manifold with boundary, the structures being compatible near the boundary \cite{APS1}. These latter triples can be roughly called geometric chains for $B\Gamma$. We denote  by $\Ind^{APS} (D_{\hat M} )$ the Atiyah-Patodi-Singer  index of the $L^2$ Fredholm operator induced by the Dirac operator $D_{\hat M}$ with the global boundary condition given by the $0$-th order pseudo differential operator which is the Szego projection of the boundary Dirac operator $\chi_{\geq} (D_M)$, see again \cite{APS1}. In the same way, on the Galois cover $\widetilde{\hat M}\to {\hat M}$, associated with ${\hat f}$, we consider the $\ell^2$ APS index $\Ind^{APS}_{(2)} ({\widetilde{D_{\hat M}}})$ using the similar global boundary condition, see \cite{Ramachandran} for the precise constructions.

\begin{definition}
A \textit{geometric cycle} is (an isomorphism class of) a 5-tuple $(M,S,f,D, x)$, where
\begin{enumerate}
\item $M$ is an odd dimensional smooth, oriented, closed Riemannian manifold with a continuous map $f:M\to B\Gamma$ and a Clifford bundle $S\to M$;
\item $D$ is a (generalized) Dirac operator associated with the Clifford bundle $S$;
\item $x$ is a real number.
\end{enumerate}
\end{definition}

We now introduce the bordism relation. We make the usual assumptions of product structures near the boundary and we identify the Dirac operator in a collar  neighborhood of the boundary with the usual matrix.

\begin{definition}\label{Boundary}
A geometric cycle is   a boundary if there exists an even geometric chain $({\hat M}, {\hat S}, {\hat f})$ for $B\Gamma$ whose boundary is $(M,S,f)$ and a Dirac operator $D_{\hat M}$  whose boundary is $D$,  such that the following relation holds:
$$
\Ind^{APS}_{(2)} ({\widetilde{D_{\hat M}}})-\Ind^{APS} (D_{\hat M} )= x.
$$
\end{definition}

We  introduce the disjoint union of two  geometric cycles $(M,S,f,D, x)$ and $(M',S',f',D', x')$ as the  geometric cycle
$$
(M, S, f, D, x) \amalg (M', S', f', D', x') := (M\amalg M', S\amalg S',f\amalg f',D\amalg D', x+x').
$$

Given a geometric cycle $(M,S,f,D, x)$, we introduce the {\em opposite} geometric cycle, denoted  $-(M,S,f,D, x)$, by considering the usual opposite $-(M, S, f):= (M, -S, f)$ of the cycle $(M, S, f)$ corresponding to the opposite Clifford multiplication \cite{KeswaniControlledPaths}, and by taking the operator $-D$ and the real number $ -x - {\tilde h} + h$, i.e.
$$
-(M,S,f,D, x) := (M, -S, f, -D, -x+ h - {\tilde h}).
$$
where $h= \dim (\Ker (D)) \text{ and } {\tilde h}= \dim_{(2)} (\Ker (\tD))$, the $\ell^2$-dimension of the  $L^2$ kernel of $\tD$.

An important relation is the so-called bundle modification and we proceed to define it following \cite{HigsonRoe2010}.
Let $(M, S, f, D, x)$ be a geometric cycle and $SO(2k) - P \stackrel{p}{\rightarrow} M$  a principal bundle, and let $\pi: {\hat M} = P\times_{SO(2k)} {\mathbb S}^{2k}\to M$ be the associated sphere bundle with the metric composed from the $SO(2k)$ invariant metric on the fibers and the  pulled-back metric from $M$ on the horizontal vectors.  Let $D_\theta$ be the $SO(2k)$-invariant Dirac operator on ${\mathbb S}^{2k}$ whose kernel is one dimensional and concentrated in degree $0$ (and whose cokernel is trivial). Recall that the Clifford bundle $S_\theta$ that we use on ${\mathbb S}^{2k}$ is  ``half''  the complex Clifford bundle of $T{\mathbb S}^{2k}$. We denote again by $S_\theta$ the induced fiberwise Clifford bundle associated over ${\hat M}$. The Clifford bundle ${\hat S}$ on ${\hat M}$ is then taken to be the tensor product $\pi^*S\otimes S_\theta$. The operator $D_\theta$ induces a fiberwise operator denoted $I\otimes D_\theta$. More precisely, $I\otimes D_\theta$ is defined using the identification of the $L^2$ sections with the $SO(2k)$-equivariant sections as usual:
\begin{eqnarray}
\label{decomp1}
L^2 ({\hat M}, {\hat S}) \simeq \left[ L^2 (P, p^*S) \otimes L^2 (\mathbb{S}^{2k}, S_\theta)\right]^{SO(2k)}.
\end{eqnarray}
The manifold $\mathbb {S}^{2k}$ being even dimensional, there is the usual grading operator $\ep$ on $S_\theta$. Now the new Dirac operator defined on ${\hat S}\to {\hat M}$ is  the operator ${\hat D}$ corresponding in this identification to the $SO(2k)$-invariant operator $I\otimes D_\theta + p^*D\otimes \ep$, i.e.
$$
{\hat D} \simeq I\otimes D_\theta + p^*D\otimes \ep,
$$
where $p^*D$ is a well defined $SO(2k)$-invariant differential operator which induces $D$ downstairs on $M$. Notice that $p^*D$ is only transversely elliptic to the action of $SO(2k)$. The map ${\hat f}=f\circ \pi: {\hat M}\to B\Gamma$ finishes the construction of the 5-tuple $({\hat M}, {\hat S}, {\hat f}, {\hat D}, x)$ which is the bundle modification of  $(M, S, f, D, x)$.

Notice that the homotopy class of ${\hat f}$ again defines the Galois cover $\pi^*\tM=\widetilde{\hat M}$ over ${\hat M}$. The above data then all lift to the Galois cover. In particular, we have the similar identification which is now $\Gamma$-equivariant:
\begin{eqnarray}
\label{decomp2}
L^2 (\widetilde{\hat M}, \widetilde{\hat S}) \simeq \left[ L^2 (\tP, \tp^*\tS) \otimes L^2 (\mathbb{S}^{2k}, S_\theta)\right]^{SO(2k)}.
\end{eqnarray}
We also have the identification of $\Gamma$-invariant operators
$$
{\widetilde{\hat D}} \simeq I\otimes D_\theta + \tp^*\tD\otimes \ep.
$$
\begin{definition}
We introduce the three  identifications beyond isomorphism classes of geometric cycles
\begin{enumerate}
\item {\underline{Bordism:}} We identify  the geometric cycles $(M,S,f,D, x)$ and $(M',S',f',D', x')$ if they are bordant, i.e. if the disjoint union
$$
(M, S, f, D, x) \amalg  - (M', S', f', D', x')
$$
is a boundary in the sense of Definition \ref{Boundary}.
\item {\underline{Disjoint union/Direct sum:}} We identify the geometric cycles
$$
(M, S_1, f, D_1, x) \amalg (M, S_2, f, D_2, x') \text{ and } (M, S_1\oplus S_2, f, D_1\oplus D_2, x+x').
$$
\item {\underline{Bundle modification:}} If $({\hat M}, {\hat S}, {\hat f}, {\hat D})$ is a bundle modification of $(M, S, f, D)$ as above, then for any $x\in \R$,  $({\hat M}, {\hat S}, {\hat f}, {\hat D}, x)$  is identified with $(M, S, f, D, x)$.
\end{enumerate}
\end{definition}

\begin{definition}
The quotient set with respect to  the equivalence relation generated by the above three moves is denoted $\maS_1^{geo, (2)} (\Gamma)$. This is clearly an abelian group with the sum induced by disjoint union. It will be called the $\ell^2$ geometric group of $\Gamma$.
\end{definition}

Notice that  the zero element is the class of the empty manifold with the real number $x=0$ and that the opposite cycle represents the opposite element in the group. More precisely,  the APS theorem and the $\ell^2$ APS theorem of Ramachandran, applied to the cylinder associated with $(M,S,f,D)$ implies that the class of
$$
(M, S, f, D, x) \amalg -(M, S, f, D, x),
$$
bounds in the sense of our definition of bordism above.

\subsection{Geometric versus analytic}\label{geometricanalytic}

Recall the class $[x]$ of the real number $x$ in the $\ell^2$ analytic structure group $\maS_1^{(2)} (\Gamma)$. Recall also that the function $\chi_\geq$ is the characteristic function of $[0, +\infty[$ and note that for any geometric cycle $(M, S, f, D, x)$, the couple $(\chi_\geq (\tD), \chi_\geq (D))$ defines a class in $\maS_1^{(2)} (M, L^2(M, S))$ and hence an element $f_* [\chi_\geq (\tD), \chi_\geq (D)]$ of the inductive group $\maS_1^{(2)} (\Gamma)$.

\begin{definition}
We associate with any geometric cycle $(M,S,f,D, x)$ its analytic class $[M,S,f,D, x]_{an}$ in $\maS_1^{(2)} (\Gamma)$ which is defined as
$$
[M,S,f,D, x]_{an} := f_*[\chi_\geq (\tD), \chi_\geq (D)] + [x] \quad \in \maS_1^{(2)} (\Gamma).
$$
\end{definition}

We shall sometimes drop $f_*$ and denote,  when no confusion can occur, as well by $[\chi_\geq (\tD), \chi_\geq (D)]$ the class in $\maS_1^{(2)} (M, L^2(M, S))$ and its image in $\maS_1^{(2)} (\Gamma)$. Our goal in the rest of this paragraph is to prove the following

\begin{theorem}\label{analyticGeometric}
The analytic class of  a geometric cycle only depends on its class in the $\ell^2$ geometric group $\maS_1^{(2), geo}(\Gamma)$ and hence induces a well defined group morphism
$$
\maS_1^{(2), geo}(\Gamma) \longrightarrow \maS_1^{(2)}(\Gamma).
$$
\end{theorem}

\begin{proof}
It is clear  that the analytic class associated with a disjoint union of two geometric cycles is equal to the sum of the corresponding analytic classes in $\maS_1^{(2)}(\Gamma)$. As an easy consequence we see that the analytic classes of two cycles which are related through the first relation, disjoint union/direct sum, are equal. Theorem \ref{analyticGeometric} is therefore a consequence of the next two propositions \ref{Modification} and \ref{Bordism}.
\end{proof}

\begin{proposition}\label{Modification}
The analytic class of the bundle modification $({\hat M}, {\hat S}, {\hat f}, {\hat D}, x)$ of the geometric cycle $ (M, S, f, D, x)$ is equal to the analytic class of $(M,S, f, D, x)$.\end{proposition}

\begin{proof}\
The proof in the type I case given in \cite{HigsonRoe2010} can be adapted  with some changes for the lifted data on the Galois cover as follows. Denote by $F_\theta$ the sign of the closed self-adjoint fiberwise operator $D_\theta$ and let $J$ and $\tJ$ be the operators corresponding in the two identifications (\ref{decomp1}) and (\ref{decomp2}) to (we use here \cite{HigsonRoe2010})
$$
J\simeq I\otimes {\sqrt{-1}} \ep F_\theta \text{ and } \tJ\simeq I\otimes {\sqrt{-1}} \ep F_\theta.
$$
Then $J$ and $\tJ$ are self-adjoint operators which {\underline{anticommute}} respectively with ${\hat D}$ and ${\widetilde{\hat D}}$. Moreover, $J^2 = I- Q$ where $Q$ is the orthogonal projection onto  $K= \left[ L^2 (P, p^*S) \otimes \Ker (D_\theta)\right]^{SO(2k)}$ and $\tJ^2 = I- \tQ$ where $\tQ$ is the orthogonal projection onto  $\tK=\left[ L^2 (\tP, \tp^*\tS) \otimes \Ker (D_\theta)\right]^{SO(2k)}$. With respect to the decompositions $
L^2 ({\hat M}, {\hat S}) \simeq K \oplus K^\perp$ and $L^2 (\widetilde{\hat M}, \widetilde{\hat S}) \simeq \tK \oplus \tK^\perp $,
we can write
$$
{\hat D}  = \left( \begin{array}{cc} A_K & 0 \\ 0 & A_K^\perp \end{array}\right) \text{ and } {\widetilde{\hat D}} = \left(\begin{array}{cc} A_{\tK} & 0 \\ 0 & A_{\tK}^\perp \end{array}\right)
$$
Notice that  $K\simeq L^2 (M, S)$ and $\tK\simeq L^2(\tM, \tS)$ and that the operators $A_K$ and $A_{\tK}$ correspond through these isomorphisms respectively to the initial  operators $D$ and $\tD$. All identifications upstairs are of course $\Gamma$-equivariant.
%similar to those used in the proof of Proposition \ref{RelativeEta} below. We shall thus anticipate and use the notations of the proof of Proposition \ref{RelativeEta}, so recall  the symmetries $J$ and $\tJ$ as well as the decomposition of the Hilbert space $L^2 ({\hat M}, {\hat S})$ with respect to the kernel of $I\otimes D_\theta$ and same for the Galois coverings. Recall that we then have decompositions $
%L^2 ({\hat M}, {\hat S}) \simeq K \oplus K^\perp$ and $L^2 (\widetilde{\hat M}, \widetilde{\hat S}) \simeq \tK \oplus \tK^\perp $, and that the Dirac operators split into
%$$
%{\hat D}  = \left( \begin{array}{cc} A_K & 0 \\ 0 & A_K^\perp \end{array}\right) \text{ and } {\widetilde{\hat D}} = \left(\begin{array}{cc} A_{\tK} & 0 \\ 0 & A_{\tK}^\perp \end{array}\right)
%$$
%From the choice of  $D_\theta$, we know that $K\simeq L^2 (M, S)$ and $\tK\simeq L^2(\tM, \tS)$ and that the operators $A_K$ and $A_{\tK}$ correspond through these isomorphisms respectively to the initial  operators $D$ and $\tD$.
It suffices to show that the couple $( \chi_\geq (A_{\tK}^\perp),  \chi_\geq (A_K^\perp))$ represents the zero class, so we are  reduced to work in the  $C^*$-algebra $D^*_{(2)} (M, K^\perp)$.

Denote by $F$ and $\tF$ the partial isometries appearing respectively in the polar decompositions of $A_K^\perp$ and $A_{\tK}^\perp$.
Since these latter operators are injective self-adjoint, they have dense images so that $F$ and $\tF$ are  invertible and self-adjoint. Since $JF+FJ=0$ and that $\tJ\tF+\tF\tJ=0$  the path
$$
\left(\cos (\theta) \tF + \sin (\theta) \tJ , \cos (\theta) F + \sin (\theta) J \right)\quad 0\leq \theta\leq \pi/2,
$$
is composed of symmetries  living in $D^*_{(2)} (M, K^\perp)$.
This shows that the class of the projection $\left(\frac{\tF+I}{2}, \frac{F+I}{2}\right)$ in $\maS_1^{(2)} (M, K^\perp)$ coincides with the class of $\left(\frac{\tJ+I}{2}, \frac{J+I}{2}\right)$ and this latter is zero because the operators commute with the action of $C(M)$ by Lemma \ref{ES}.
This finishes the proof, indeed  $2\chi_\geq (A_K^\perp) -1$ and $F$ are invertible symmetries which fit in the polar decomposition, hence they are equal. The same equality holds on the covers and we thus have:
$$
\left( \chi_\geq (A_{\tK}^\perp),  \chi_\geq (A_K^\perp)\right) = \left(\frac{\tF+I}{2}, \frac{F+I}{2}\right).
$$
\end{proof}

The following classical lemma immediately extends Proposition 8.2.8 in  \cite{HigsonRoeBook}.

\begin{lemma}\label{ES}
Let $(T_1,T_2) \in Proj(D^*_{(2)}(X,H))$ be such that $[T_1,f] =0 \in \maK(\maM_X,\tau)$ and $[T_2,f]=0 \in K(H)$ for all $f\in C(X)$. Then $[T_1,T_2]=0 \in K_0(D^*_{(2)}(X,H))$.
\end{lemma}

\begin{proof}
We follow \cite{HigsonRoeBook} and adapt it to our semi-finite setting. Consider the element $(\oplus_{\N} T_1,\oplus_{\N} T_2)$. This is a well-defined element in $D^*_{(2)}(X,\oplus_{\N}H)$ with the associated representation $\oplus_{\N}\pi$, since for any $f\in C(X)$ we have $[\oplus_{\N}T_1,f] =0 \in \maK(\maM_X\otimes B(\ell^2\N),\tau\otimes \tr)$ and $[\oplus_{\N}T_2,f]=0 \in K(\oplus_{\N}H)$. Using the isomorphism on the level of $K$-theory for ample representations on the Hilbert spaces $H$ and $\oplus_{\N}H$, we get
$$ [\oplus_{\N}T_1,\oplus_{\N}T_2]+[T_1,T_2]= [\oplus_{\N}T_1,\oplus_{\N}T_2]$$

Thus $[T_1,T_2]=0$.
\end{proof}

\begin{proposition}\label{Bordism}\
The analytic class of a geometric cycle which bounds, is trivial in $\maS_1^{(2)}(\Gamma)$.
\end{proposition}

The proof of Proposition \ref{Bordism} is long and will be split into different lemmas and will also use Appendices \ref{Resolvent} and \ref{BVP} which are of independent interest. We assume that the given odd geometric cycle  $(M, S, f, D, x)$ such that $(M,S,f,D)$ bounds the even chain $(\hM, \hS, \hf, \hD)$ in the sense of Definition \ref{Boundary}. We can and will assume that $\hM$ is connected. The Clifford bundle $\hS$ is $\Z_2$-graded; we denote the two graded components as $\hS^+$ and $\hS^-$. The Dirac operator $\hD$ on $\hS$ then decomposes as an odd operator (cf. \cite{HigsonRoe2010}, Definition 3.4):

$$
\hD=\left(\begin{array}{lll}0 & \hD^+\\ \hD^- & 0 \end{array}\right)
$$
Taking into account the grading we will sometimes also denote the ambient even chain $(\hM, \hS, \hf, \hD)$  as $(\hM, \hS^\pm, \hf, \hD^\pm)$.\\

 The space of distributional sections of $\hS$ over $\hM$ will always be for us the topological dual of the space $C_c^\infty (\hM\smallsetminus M, \hS)$ of smooth sections of $\hS$ over $\hM$ which are compactly supported in the interior $\hM\smallsetminus M$, see \cite{BW}.  As usual, we first consider the unbounded operator $\hD=\hD^\pm$ on the dense subspace  $C_c^\infty (\hM\smallsetminus M, \hS)$ and notice then that this is a closable operator.  We denote by $\hD_{\min}^\pm$ the minimal closures of these operators and by $\hD_{\max}^\pm$ the maximal closures. The domain of $\hD^+_{\min}$ is the space of $L^2$ sections $u_+$ of $\hS^+$ with $u_+\vert_M=0$ and such that $\hD^+u_+$, defined in the distributional sense, belongs to $L^2(\hM, \hS^-)$, see again the seminal reference \cite{BW}. It can also be viewed as the closure $W^1_0(\hM, \hS^+)$  of the space $C_c^\infty (\hM\smallsetminus M, \hS)$ in the Sobolev space $W^1(\hM, \hS^+)$.  In the same way, the domain of  $\hD^+_{\max}$ is the space of $L^2$ sections $u_+$ of $\hS^+$ such that $\hD^+u_+$, defined in the distributional sense, belongs to $L^2(\hM, \hS^-)$. The same observations apply to $\hD^-$.

By the von Neumann theorem, the operator $\left((\hD^+_{\min} )^*\hD^+_{\min} + I\right)^{-1/2}$ is thus a  bounded isomorphism between the Hilbert spaces $L^2(\hM, \hS^+)$ and $W^1_0(\hM, \hS^+)$. As a corollary of  the Rellich lemma, we deduce the classical result, see for instance  \cite{BaumDouglasTaylor}:
$$
\left((\hD^+_{\min} )^*\hD^+_{\min} + I\right)^{-1/2} : L^2(\hM, \hS^+)\longrightarrow L^2(\hM, \hS^+)\text{ is a compact operator.}
$$
Lemma 1.2 in \cite{BaumDouglasTaylor}  allows to show that the operator
$$
\left((\hD^+_{\max} )^*\hD^+_{\max} + I\right)^{-1/2} : L^2(\hM, \hS^+)\longrightarrow L^2(\hM, \hS^+)
$$
is a compact operator in restriction to the orthogonal of the kernel of $\hD^+_{\max}$. See again \cite{BaumDouglasTaylor} and in particular the proof of Proposition 3.1 there.

In  Lemma  \ref{Lemma9}, we extend \cite{BaumDouglasTaylor}[Lemma 1.2] to the semi-finite setting. It can thus be applied to our covering situation with
 the lifted data $(\tM, \tS,  \tD)$ on the Galois coverings $\tM\to M$ which bounds the lifted data $(\thM, \thS^\pm,  \thD^\pm)$ on the Galois covering with boundary $\thM\to \hM$. The operators $\thD^\pm$ are defined similarly and we end up with the corresponding minimal and maximal closures $\thD^\pm_{\min}$ and $\thD^\pm_{\max}$.

Again the domain of $\thD^+_{\min}$ can be identified with the closure $W^1_0(\thM, \thS^+)$ of the space $C_c^\infty (\thM\smallsetminus \tM, \thS^+)$ in the first Sobolev space $W^1(\thM, \thS^+)$. The $\Gamma$-equivariant Rellich lemma then shows that the $\Gamma$-invariant operator
$$
\left((\thD^+_{\min} )^*\thD^+_{\min} + I\right)^{-1/2} : L^2(\thM, \thS^+)\longrightarrow L^2(\thM, \thS^+)
$$
 is a $\tau$-compact operator in the von Neumann algebra $B(L^2(\thM, \thS^+))^\Gamma$. Hence by Lemma \ref{Lemma9} and the same argument as in \cite{BaumDouglasTaylor}, we also show that the operator
$$
\left((\thD^+_{\max} )^*\thD^+_{\max} + I\right)^{-1/2} :  \Ker (\thD^+_{\max})^\perp \rightarrow \Ker (\thD^+_{\max})^\perp
$$
is a $\tau$-compact operator in the von Neumann algebra of $\Gamma$-invariant operators in  the $\Gamma$ representation $\Ker (\thD^+_{\max})^\perp$.

Consider the self-adjoint closed operators
$$
\hQ= \left(\begin{array}{cc} 0 & (\hD_{\max}^+)^* \\ \hD_{\max}^+ & 0 \end{array} \right) \text{ and } \thQ= \left(\begin{array}{cc} 0 & (\thD_{\max}^+)^* \\ \thD_{\max}^+ & 0 \end{array} \right),
$$
acting on $L^2 (\hM, \hS)$ and $L^2(\thM, \thS)$ respectively. A corollary of the previous discussion is that the restriction of the operator
$$
(\hQ^2 +I)^{-1/2} \text{ ( respectively  } (\thQ^2 +I)^{-1/2}),
$$
to the orthogonal of its kernel, is a compact (respectively $\tau$-compact) operator.

\begin{lemma}
\label{denserange}
The closed operators
$$
\hD^+_{\max} : \Dom (\hD^+_{\max}) \longrightarrow L^2 (\hM, \hS^-) \text{ and } \thD^+_{\max} : \Dom (\thD^+_{\max}) \longrightarrow L^2 (\thM, \thS^-)
$$
have dense ranges.
\end{lemma}

\begin{proof}
The proof for $\hD^+$ is given in \cite{HigsonRoe2010} and it is a consequence of the unique continuation property of solutions of Dirac equations. This proof extends to the lifted data on the Galois covering as we proceed to explain  for the convenience of the reader.
We need  to prove that the operator $\thD^-_{\min}$ is injective. Given $\tu\in \Ker (\thD^-)$, there exists a sequence $(\tu_n)_{n\geq 0}$ of smooth  sections of $\thS$ over $\thM$ which are compactly supported  in $\thM\smallsetminus \tM$ such that we have in $L^2 (\thM, \thS^\pm)$:
$$
\tu_n \stackrel{n\to +\infty}{\longrightarrow} \tu \text{ and } \thD^- \tu_n \stackrel{n\to +\infty}{\longrightarrow} 0.
$$
Consider the double manifold $W=\hM\amalg_M (-\hM)$ with the induced map $g: W\to B\Gamma$. Then the map $g$ defines the Galois covering $\tW$ over $W$. The sections $(\tu_n)_{n\geq 0}$ as well as $\tu$ all extend by zero to the  manifold $\tW$. The operator $\thD^-$ also extends into a $\Gamma$-invariant Dirac operator on $\tW$ and we still have in the $L^2$ norms over the  manifold $\tW$ (a Galois $\Gamma$ covering without boundary over the closed manifold $W$):
$$
\tu_n \stackrel{n\to +\infty}{\longrightarrow} \tu \text{ and } \thD^- \tu_n \stackrel{n\to +\infty}{\longrightarrow} 0.
$$
Using Sobolev spaces on Galois coverings \cite{Ramachandran}, we deduce from the ellipticity and the $\Gamma$-invariance of the extended operator $\tD^-$ to $\tW$, that the section $\tu$ is smooth and $C^\infty$ uniformly bounded, and is a solution of a Dirac equation over $\tW$, see Proposition 2.2.2 in \cite{Ramachandran}. Now the unique continuation property holds as well in our situation. Indeed, the proof of Theorem 8.2 in \cite{BW} remains valid for  Galois coverings over connected closed manifolds and the ($\Gamma$-invariant) generalized Dirac operator on the cover, the Carleman estimate being local, see also \cite{Antonini}. We deduce from this discussion that $\tu$ is the zero section over $\tW$ and hence the conclusion.

\end{proof}

\begin{lemma}
We denote by $\hV$ and $\thV$ the partial isometries appearing in the polar decomposition of the closed operators $\hQ$ and $\thQ$ respectively. Then the pair $(\thV, \hV)$ belongs to the  $C^*$-algebra
$
D^*_{(2)} (\hM, L^2(\hM, \hS)).
$
\end{lemma}

\begin{proof}\
Consider as in \cite{HigsonRoe2010} the bounded operators
$$
\hF:=\hQ (\hQ^2+I)^{-1/2} \text{ and } \thF:=\thQ (\thQ^2+I)^{-1/2}.
$$
It is a straightforward observation that these operators commute with $C(\hM)$ modulo compact and $\tau$ compact operators respectively. This is a consequence of Proposition 1.1 from \cite{BaumDouglasTaylor} and its easy extension to the semi-finite setting, see the second item of Lemma \ref{Lemma9} in Appendix \ref{Resolvent}. Since the operators $\Sgn (\hQ)$ and $\hF$ are both zero on the kernel of $\hQ$, their difference is a compact operator by the Lemma 9 in \cite{BaumDouglasTaylor}. The same holds on the cover by Lemma \ref{Lemma9}  and we get that $\Sgn (\thQ) - \thF$ is $\tau$-compact. Indeed and more precisely, notice that if $\hP$ and $\thP$ are the projections onto the orthogonal subspaces of the kernels of $\hQ$ and $\thQ$ respectively, then
$$
\hF - \Sgn (\hQ) = \hF (\vert \hQ\vert - (\hQ^2+I)^{1/2}) \left[ \hP (\hQ^2+I)^{-1/2} \hP \right] \text{ and }\thF - \Sgn (\thQ) = \thF (\vert \thQ\vert - (\thQ^2+I)^{1/2}) \left[ \thP (\thQ^2+I)^{-1/2} \thP \right].
$$
Hence using the regular operator $\widehat {\maQ}$ induced by $\hQ$ on the Mishchenko bundle and the fact that for any continuous bounded  function $f$, the operator $f(\widehat\maQ)$ induces through the regular and average representations the operators $f(\thQ)$ and $f(\hQ)$ respectively, we deduce that $\Sign (\widehat \maQ)$ defined using the polar decomposition in Hilbert modules lifts our pair $(\Sign (\thQ), \Sign (\hQ))$. The conclusion follows.

\end{proof}

Given a smooth section $u_+$ of $\hS^+$ over $\hM$, we denote by $b_+u_+$ the trace (restriction) of $u$ at the boundary. In the same way we define $b_-u_-$ as the trace of a section $u_-$ of $\hS^-$ at the boundary. Recall indeed that Clifford multiplication by the inward unit covector $du$ which is  normal to the boundary yields a vector bundle isomorphism $G: \hS^+\vert M \simeq \hS^-\vert M$ which extends to a tubular neighborhood of $M$ in $\hM$. We shall also denote the induced isomorphism $L^2(M,\hS^+\vert_M)\rightarrow L^2(M,\hS^-\vert_M)$ by $G$.
The Dirac operator $\hD^\pm: C^\infty(\hM,\hat S^\pm)\rightarrow C^\infty(X ,\hat S^\mp)$ has the product form in a collar neighbourhood $M\times [0,1] \hookrightarrow \hM$ of the boundary, i.e. it can be written as $\hD=G (\del_u+ D)$, where $D$ is the Dirac operator on the odd closed boundary manifold $M$ and $G$ is the bundle isomorphism defined above. In this expression, we have identified the restriction of $\hS^\pm$ with the pull-back of  the Clifford bundle $S$ over $M$ to the tubular neighborhood.  All these considerations are valid on the cover  $\thM$ with boundary $\tM$. In particular, we also have the isomorphism $\tG:L^2(\tM,\thS^+\vert_{\tM})\rightarrow L^2(\tM,\thS^-\vert_{\tM})$ and we can similarly define the boundary restriction maps on the covering manifold that we denote by ${\tilde b}_\pm$.

It is a classical result that the trace maps $b_\pm$ extend to bounded operators $W^s (\hM, \hS^\pm) \to W^{s-1/2} (M, \hS^\pm\vert_{M})$ if we insist that $s>+1/2$. The proof of this result, given in \cite{BW}, extends to the Galois cover situation and is given in \cite{Ramachandran}. Moreover, by Theorem 13.8 in \cite{BW},  we know that the restriction   $b_\pm (v)$ of $v\in  W^t(\hM, \hS^\pm)$ makes sense in $W^{t-1/2} (M, \hS^\pm\vert_{M})$  for any real number $t$, if we assume in addition that $\hD (v) \in W^{s}(\hM, \hS^\pm\vert_{M})$  for some  $s>-1/2$. We explain in Appendix \ref{BVP}  the extension  of this result to the Galois cover case. As a byproduct of this discussion we deduce:

\begin{proposition}
Assume that $\tv\in L^2 (\thM, \thS^\pm)$ belongs to the $L^2$ kernel of $\thD_{\max}$. Then its trace at the boundary $M$ belongs to $W^{-1/2} (\tM, \thS^\pm\vert_{\tM})$. \end{proposition}

In the sequel we denote by $\Psi_{s,t}: W^s \to W^t$ the usual isomorphism between the Sobolev spaces induced  by powers of $I+$ square of the Dirac operator (on the manifolds with boundary, their boundaries  and also on the corresponding  covers).

In Appendix \ref{BVP} we also construct  the $\Gamma$-invariant Calderon projection $\maC (\thD^+)$ for the Dirac operator $\thD^+$. The  principal symbol of $\maC (\thD^+)$ is simply denoted by $c^+$ or $c(\thD^+)$ when it is necessary to emphasize which Dirac operator is considered ($=$ the principal symbol of the Szego projection).

Let $\thD^+_{APS}$ denote the APS operator  with domain given by
$$
\Dom(\thD^+_{APS}):=\{ u \in \Dom(\thD^+_{\max}), \; \Psi_{-1/2,0}(\tilde b_+u)\in \Ker(\chi_{\geq}(\tD))\}.
$$
As for general Boundary Value Conditions considered in the appendix, we consider the following composite map $\tF_{APS}^+: \Ker \thD^+_{\max} \rightarrow \Range (\chi_{\geq}(\tD))$ given by:
$$
\Ker \thD^+_{\max} \xrightarrow{{\tilde b}^0_+} W^{-1/2}(\tY,\hS^+_{|_{\tY}})\xrightarrow{\Psi_{-1/2,0}}L^2(\tY,\hS^+_{|_{\tY}})\xrightarrow{\chi_{\geq}(\tD)} \Range (\chi_{\geq}(\tD))%\subseteq  L^2(\tY,\hS^+_{|_{\tY}})
$$
where ${\tilde b}^0_+$ is the map $\tb^+$ that we have restricted to $\Ker \thD^+_{\max}$.

Applying Proposition \ref{KerCoker} in Appendix \ref{BVP} to $R= \chi_\geq (\tD)$, we deduce:

\begin{proposition}\label{ker}
We have
$$
\Ker \tF_{APS}^+ = \Ker \thD^+_{APS}\text{ and } \Coker \tF_{APS}^+ = \Coker \thD^+_{APS}.
$$
In particular, the $\tau$ index of the APS operator $\thD^+_{APS}$ coincides with the $\tau$ index of $\tF_{APS}^+$.
\end{proposition}

\begin{proof} (of Proposition \ref{Bordism})
Let $\tilde P^+$ (resp. $P^+$) denote the projection onto the subspace $\Ker(\thQ)$ (resp. $\Ker(\hQ)$).  From Proposition \ref{ker} we deduce that there are pairs of projections $[\tilde e^+,e^+ ], [\tilde e^-, e^-] \in \maS^{(2)}_1(\hM, L^2(\hM, \hS))$ such that
$$
[\tilde P^+, P^+] - [\tilde e^+,e^+ ] = i_{\del,*}([\chi_\geq (\tD),\chi_\geq (D)]) - [\tilde e^-,e^-]\in S^{(2)}_1(\hM, L^2 (\hM, \hS)),
 $$
where $i_{\del}: M\hookrightarrow \hM$ is the boundary inclusion. Moreover,
$$
\tau(\tilde e^+)-\tau(\tilde e^-)= \Ind_{(2)}(\thD^+_{APS}) \text{ and }  \Tr(e^+)-\Tr(e^-)=\Ind(\hD^+_{APS}).
$$
Using the results of Section \ref{analytic} and the definition of the class $[x]\in \maS^{(2)}_1(\Gamma)$, we deduce that
$$
[\tilde P^+, P^+] = [\chi_\geq (\tD),\chi_\geq (D)]+[x].
$$
It thus only remains to show that the class $[\tilde P^+,P^+]$ is zero in $\maS^{(2)}_{1}(\Gamma)$.

From Proposition \ref{denserange} we know that the range of $\thD^+_{\max}$ is dense in $L^2(\thM,\thS^-)$. Therefore the partial isometry part $\tU^+$ of the polar decomposition of $\thD^+_{\max}$ is surjective. But
\begin{equation}
\label{plus}
 \tP^++(\tU^+)^*\tU^+=id_{L^2(\thM,\thS^+)},
\end{equation}
and $\tU^+(\tU^+)^*$ is the projection onto the closure of the range space of $\thD^+_{\max}$. Therefore,
\begin{equation}
\label{minus}
\tU^+(\tU^+)^*=id_{L^2(\thM,\thS^-)}
\end{equation}
Thus from Equations (\ref{plus}) and (\ref{minus}), we get at the level of $K$-theory:
$$
[\tilde P^+, P^+] = [id_{L^2(\thM,\thS^+)}, id_{L^2(\hM,\hS^+)}]-[id_{L^2(\thM,\thS^-)}, id_{L^2(\hM,\hS^-)}].
$$
The right hand side of the above equation is zero in $K$-theory since the operators commute with the representation, by an  ``Eilenberg swindle" argument of Lemma \ref{ES}.   Hence the proof is complete.

\end{proof}

\begin{proposition}
We have an exact sequence
$$
 0 \to \R \longrightarrow \maS_1^{(2), geo} (\Gamma) \longrightarrow K_1^{geo} (B\Gamma) \to 0.
$$
which fits in the following commutative diagram
\vspace{0,2cm}
\hspace{0,5cm}
\begin{equation}\label{geo-an}
\begin{CD}
0 @>>>  \R   @>>>    \maS_1^{(2),geo} (\Gamma)  @>>> K_1^{geo} (B\Gamma) @>>>  0\hspace{1cm}\\
@. @V{=}VV @VVV  @V{\cong}VV @.\\
0 @>>> \R @>>>  \maS_1^{(2)}(\Gamma) @>>>  K_1(B\Gamma)@>>>  0\hspace{1cm}\\
\end{CD}
\end{equation}
\vspace{0,5cm}
\end{proposition}

\begin{proof}
The maps are defined as follows. We assign to any $x\in \R$ the class of $(M,S,f,D,x)$ where $M=\emptyset$ and we assign to any $(M,S,f,D,x)$ its class in $K_1^{geo} (B\Gamma)$, where we consider the version of the Baum-Douglas geometric $K$-homology group \cite{BaumDouglas}with oriented cycles, we refer the reader to \cite{HigsonRoe2010}, Section 3 for the detailed definitions and  \cite{KeswaniControlledPaths}, Section 2 for the isomorphism between this group and the Baum-Douglas geometric group. The forgetful map $\maS_1^{(2), geo} (\Gamma) \rightarrow K_1^{geo} (B\Gamma)$ is then surjective.
%\green{is a consequence of the equivalence of the Baum-Douglas definition with the geometric definition in \cite{KeswaniControlledPaths} using only triple of the type $(M, S, f)$, so generalize Dirac operators over not necessarily $K$-oriented maps.}
It is clear by definition that the composite map
$$
\R \longrightarrow \maS_1^{(2), geo} (\Gamma) \longrightarrow \maS_1^{(2)}(\Gamma)
$$
coincides with the map $x\mapsto [x]$ appearing in the analytic exact sequence
$$
0\to  \R   \longrightarrow \maS_1^{(2)}(\Gamma) \longrightarrow   K_1(B\Gamma)\to  0
$$
Hence the injectivity of the map $\R \longrightarrow \maS_1^{(2), geo} (\Gamma)$ is deduced right away.

Now  given a cycle $(M,S,f,D,x)\in \maS_1^{(2), geo} (\Gamma)$,   and if $(e_1, e_2)\in \maK(\maM_M,\tau)\oplus \maK (L^2(M, S))$ is a projection satisfying $\tau(e_1) - \Tr (e_2) = x$ then
$$
[\chi_\geq (\tD), \chi_\geq (D)] + [x] = [\chi_\geq ({\tilde D}) + e_1, \chi_\geq (D) + e_2] \in K_0(D^*_{(2)} (M, L^2(M,S))).
$$
This latter is sent under the morphism, induced by the quotient map from $D^*_{(2)} (M, L^2(M,S))$  to the quotient algebra $D^*_{(2)} (M, L^2(M,S))/\left[\maK(\maM_M,\tau)\oplus \maK (L^2(M, S))\right]$, to the class in the $K_0$ of this quotient $C^*$-algebra of $\chi_\geq (D)$. But, this latter is obviously sent under the identification
$$
K_0\left( D^*_{(2)} (M)/[\maK(\maM_M,\tau)\oplus \maK (L^2(M, S))]\right)\simeq K_1 (M)
$$
to the class $[D]\in K_1(M)$. Hence we deduce that
the image of $[\chi_\geq (\tD), \chi_\geq (D)] + [x] \in \maS_1^{(2)} (\Gamma)$ in $ K_1 (B\Gamma)$ coincides the analytic class $f_*[D]\in K_1(B\Gamma)$ where $f_*: K_1(M) \to K_1(B\Gamma)$ is the push forward map in $K$-homology.
\end{proof}

We are now in position to prove the following

\begin{theorem}\label{Isomorphism}
The morphism $\maS_1^{(2), geo} (\Gamma) \longrightarrow \maS_1^{(2)}(\Gamma) $ defined above  by
$$
[M, S, f, D, x] \longmapsto f_*[\chi_\geq (\tD), \chi_\geq (D)] + [x]
$$
is an isomorphism.
\end{theorem}

\begin{proof} \
This is a consequence of the commutativity of the previous diagram \eqref{geo-an}. More precisely, the isomorphism between the Baum-Douglas geometric $K$-homology and analytic $K$-homology was first proved in \cite{BaumHigsonSchick} while  the identification between the Baum-Douglas geometric $K$-homology group \cite{BaumDouglas} and its oriented version used in our paper was proved in \cite{KeswaniControlledPaths}. Then, the proof is now complete by applying  the five lemma.
\end{proof}

\section{Cheeger-Gromov invariant as a structure morphism}\label{RhoMorphism}

We  associate with any choice of a Dirac operator $D$ for the geometric  cycle $(M,S,f)$, and using the Galois cover associated with $f$, the Cheeger-Gromov $\ell^2$ eta invariant \cite{CheegerGromov} that we  denote by $\eta_{(2)} (\tD)$ as well as the usual APS eta invariant $\eta (D)$  of the operator $D$. Then we define the $\ell^2$ rho invariant by the formula
$$
\rho_{(2)} (M, S, f, D) :=\frac{1}{2}\left[ (\eta_{(2)} (\tD) - \eta (D)) + ({\tilde h} - h) \right]
$$
where
$$
h= \dim (\Ker (D)) \text{ and } {\tilde h}= \dim_{(2)} (\Ker (\tD)).
$$

\begin{definition}
The $\ell^2$ relative eta invariant of the geometric cycle $(M,S,f,D, x)$ is by definition
$$
\xi (M,S,f,D, x) := \rho_{(2)} (M, S, f, D) + x \quad \in \R.
$$
\end{definition}

\begin{proposition}\label{RelativeEta}
The $\ell^2$ relative eta invariant of the geometric cycle $(M,S,f,D, x)$  only depends on its  class $[M,S,f,D, x] $ in $\maS_1^{geo, (2)} (\Gamma)$. Moreover, the resulting map
$$
\xi: \maS_1^{geo, (2)} (\Gamma)\longrightarrow \R,
$$
 is a group morphism.
\end{proposition}

\begin{proof}
It is clear that the $\ell^2$ relative eta invariant of the disjoint union of two geometric cycles is the sum of the respective $\ell^2$ relative eta invariants.
Assume now that the  geometric cycle $(M, S, f, D, x)$ is such that the Baum-Douglas cycle $(M,S,f)$ bounds the chain $({\hat M}, {\hat S}, {\hat f})$, the Dirac operator $D$ bounds ${\hat D}$ with
$$
\Ind^{APS}_{(2)}(\widetilde{\hat D}) - \Ind^{APS} ({\hat D}) = x.
$$
Then applying the APS theorem to ${\hat D}$ and the $\ell^2$ APS theorem to $\widetilde{\hat D}$, we get
$$
\Ind^{APS}_{(2)}(\widetilde{\hat D}) - \Ind^{APS} ({\hat D}) = -\rho_{(2)} (M, S, f, D).
$$
Therefore, we end up with the equality $\xi (M,S,f,D, x) = \rho_{(2)} (M, S, f, D) + x= 0$. This shows, together with the compatibility with direct sums that the  $\ell^2$ relative eta invariant respects the bordism  relation.

For the bundle modification relation, the proof in the type I case given in \cite{HigsonRoe2010} needs to be modified so as to fit with the new type II situation. We shall use the notations of the bundle modification relation described in the previous section.
Let $(M, S, f, D, x)$ be a geometric cycle and $SO(2k) - P \stackrel{p}{\rightarrow} M$  a principal bundle, and let $\pi: {\hat M} = P\times_{SO(2k)} {\mathbb S}^{2k}\to M$ be the associated sphere bundle. Recall that we have defined the Dirac operator ${\hat D}$ acting on the sections of ${\hat S}$ over  ${\hat M}$.

%Denote by $F_\theta$ the sign of the closed self-adjoint fiberwise operator $D_\theta$ and let $J$ and $\tJ$ be the operators corresponding in the above two identifications to (we use here \cite{HigsonRoe2010})
%$$
%J\simeq I\otimes {\sqrt{-1}} \ep F_\theta \text{ and } \tJ\simeq I\otimes {\sqrt{-1}} \ep F_\theta.
%$$
%Then $J$ and $\tJ$ are self-adjoint operators which {\underline{anticommute}} respectively with ${\hat D}$ and ${\widetilde{\hat D}}$. Moreover, $J^2 = I- Q$ where $Q$ is the orthogonal projection onto  $K= \left[ L^2 (P, p^*S) \otimes \Ker (D_\theta)\right]^{SO(2k)}$ and $\tJ^2 = I- \tQ$ where $\tQ$ is the orthogonal projection onto  $\tK=\left[ L^2 (\tP, \tp^*\tS) \otimes \Ker (D_\theta)\right]^{SO(2k)}$.

%With
From the proof of Proposition (\ref{Modification}), we have the following decompositions
$$
L^2 ({\hat M}, {\hat S}) \simeq K \oplus K^\perp \quad \text{ and } L^2 (\widetilde{\hat M}, \widetilde{\hat S}) \simeq \tK \oplus \tK^\perp,
$$
where $K= \left[ L^2 (P, p^*S) \otimes \Ker (D_\theta)\right]^{SO(2k)}$ and $\tK=\left[ L^2 (\tP, \tp^*\tS) \otimes \Ker (D_\theta)\right]^{SO(2k)}$, so we can write
$$
{\hat D}  = \left( \begin{array}{cc} A_K & 0 \\ 0 & A_K^\perp \end{array}\right) \text{ and } {\widetilde{\hat D}} = \left(\begin{array}{cc} A_{\tK} & 0 \\ 0 & A_{\tK}^\perp \end{array}\right)
$$
Recall also that we have isomorphisms $K\simeq L^2 (M, S)$ and $\tK\simeq L^2(\tM, \tS)$ and that the operators $A_K$ and $A_{\tK}$ correspond through these isomorphisms respectively to the initial operators $D$ and $\tD$.
For any  bounded Borel function $f$ we  have the relation $
f(J A_K^\perp J) = J f(A_K^\perp) J$. Since $J$ anticommutes with ${\hat D}$ the self-adjoint operator $J A_K^\perp J$ coincides with $-A_K^\perp$. Moreover, if $f$ is  odd, then
$$
f(J A_K^\perp J) + f(A_K^\perp) =0.
$$
The similar relations hold on the cover, namely $f(\tJ A_{\tK}^\perp \tJ) = \tJ f(A_{\tK}^\perp)\tJ $ and when $f$ is odd $ f(\tJ A_{\tK}^\perp \tJ) + f(A_{\tK}^\perp) =0$. As a consequence, we deduce that for any $t>0$:
$$
J A_K^\perp  \exp (-t (A_K^\perp  )^2) J + A_K^\perp  \exp (-t (A_K^\perp)^2) = 0 \text{ and } \tJ A_{\tK}^\perp  \exp (-t (A_{\tK}^\perp  )^2) \tJ+ A_{\tK}^\perp  \exp (-t (A_{\tK}^\perp)^2) = 0.
$$
Therefore
$$
\Tr \left(A_K^\perp  \exp (-t (A_K^\perp)^2)\right) =-  \Tr\left(J^2 A_K^\perp  \exp (-t (A_K^\perp)^2) \right) \text{ and } \tau\left(A_{\tK}^\perp  \exp (-t (A_{\tK}^\perp  )^2) \right) = - \tau\left( \tJ^2 A_{\tK}^\perp  \exp (-t (A_{\tK}^\perp  )^2) \right).
$$
Notice that the von Neumann trace $\tau$ used here is the Atiyah trace for $\Gamma$-invariant operators on the $\Gamma$-Hilbert subspace $\tK^\perp$.  Since $
J^2 A_K^\perp = A_K^\perp \text{ and } \tJ^2A_{\tK}^\perp = A_{\tK}^\perp$,
we get:
$$
\Tr \left(A_K^\perp  \exp (-t (A_K^\perp)^2)\right) = 0 \text{ and } \tau\left(A_{\tK}^\perp  \exp (-t (A_{\tK}^\perp  )^2) \right) = = 0.
$$
Finally,
$$
\eta ({\hat D}) = \eta (D) \text{ and } \eta_{(2)} (\widetilde{\hat D}) = \eta_{(2)} (\tD).
$$
The obvious isomorphism between the kernel of ${\hat D}$ and the kernel of $D$ and the same isomorphism between the kernel of ${\widetilde{\hat D}}$ and that of $\tD$ allow to finish the proof, since the real number $x$ is unchanged in this bundle modification relation.

That the map $\xi$ is a group morphism is then straightforward.

\end{proof}

\section{Some corollaries}\label{Applications}

In the rest of the paper, we shall identify the geometric and analytic $\ell^2$ structure  groups through the isomorphism of Theorem \ref{Isomorphism} in the previous section.

\begin{definition}
Using  the group morphism $\xi: \maS_1^{(2), geo} (\Gamma)\to \R$ described above, we end up with a group morphism
$$
\xi_{(2)}:= \xi \circ \alpha_*: \maS_1 (\Gamma) \longrightarrow  \R.
$$
\end{definition}

\subsection{PSC and vanishing of the $\ell^2$ rho invariant}

Assume that $M$ is a closed odd dimensional spin manifold which has a metric of positive scalar curvature and let $f:M\to B\Gamma$ be a classifying map for the $\Gamma$-cover $\tM\to M$. The Dirac operator $D$ for the prescribed spin structure yields a regular elliptic operator $\maD_{max}$ in the maximal Mishchenko calculus whose spectrum is, by the Lichnerowicz formula, contained in an interval of the form $\R\smallsetminus [-\ep, +\ep]$ for some $\ep >0$. Therefore the {\underline{continuous}} functional calculus allows to define the class $[\chi_\geq (\maD)]$ in the Higson-Roe structure group $\maS_1(\Gamma)$. This class  is a pre image of the $K_1$ class of $D$ in the Higson-Roe exact sequence. See \cite{HigsonRoe2010} for more details.

\begin{proposition}
The real number $\xi_{(2)} ([\chi_\geq (\maD)])$ coincides with half the Cheeger-Gromov rho invariant of $D$, i.e.
$$
\xi_{(2)} ([\chi_{\geq} (\maD)]) = \frac{1}{2} \left(\eta_{(2)}(\tD) - \eta (D)\right).
$$
\end{proposition}

\begin{proof}
Denote by $H$ the Hilbert space of $L^2$ spinors. Then, since the spectra of all operators $\maD, \tD$ and $D$ contain a gap $[-\ep, +\ep]$, the class $\alpha_*[\chi_\geq (\maD)]$ in $\maS_1^{(2)}(M, H)$ is represented by $(\chi_\geq (\tD), \chi_\geq (D))$. This is a consequence of the compatibility of the continuous functional calculus with the composition of Hilbert modules, see for instance \cite{Lance}. Now, the class $[\chi_\geq (\tD), \chi_\geq (D)]$ is clearly the image under the isomorphism of Theorem \ref{Isomorphism} of the $\ell^2$  geometric class represented by $(M, S, f, D, 0)$, where $S$ is the spin bundle.
Therefore, we have
$$
\xi_{(2)} [\chi_\geq (\maD)] = \xi (M,S, f, D, 0) = \rho_{(2)} (f, D) +0= \frac{1}{2} \left(\eta_{(2)} (\tD) +\tilde{h} - \eta(D) - h\right).
$$
Again the positive scalar curvature implies that $h=\tilde{h}=0$ and hence the conclusion.
\end{proof}

\begin{theorem}\label{PSC}\cite{KeswaniVN}
Assume that $M$ is  a closed odd dimensional spin manifold which has a metric of positive scalar curvature and let $f:M\to B\Gamma$ be a classifying map for the $\Gamma$-cover $\tM\to M$. Assume that the assembly map $\mu_{\max}: K_*(B\Gamma) \rightarrow K_*(C^*_{\max}\Gamma)$, appearing in the Higson-Roe exact sequence, is an isomorphism, then the Cheeger-Gromov rho invariant of the spin Dirac operator vanishes.
\end{theorem}

\begin{proof}
Applying the previous proposition together with the Higson-Roe exact sequence we deduce that
$$
\maS_1(\Gamma) = \{0\}\text{ and } \eta_{(2)} (\tD) - \eta (D) = 2  \rho_{(2)} (f, D) \in \xi_{(2)} \left(\maS_1(\Gamma)\right).
$$
Hence using that $\xi_{(2)}$ is a group morphism, we get $\eta_{(2)} (\tD) - \eta (D) = 0$.
\end{proof}

\begin{remark}
When $\Gamma$ is torsion free and a-T-menable, the map $\mu_{\max}: K_*(B\Gamma) \rightarrow K_*(C^*_{\max}\Gamma)$ coincides with the reduced Baum-Connes map which is then known to be an isomorphism, and the theorem applies. Compare \cite{KeswaniVN}.
\end{remark}

\subsection{Homotopy invariance of the Cheeger-Gromov rho}\label{SectionHomotopy}

We now deduce the homotopy invariance of the $\ell^2$ rho invariant for the signature operator. Let $M$ and $M'$ be closed, odd-dimensional, oriented smooth riemannian manifolds equipped with an orientation preserving homotopy equivalence $F:M\rightarrow M'$.  Let $f':M'\rightarrow B\Gamma$ be classifying map for the $\Gamma$-covering $\tM '\to M'$ and set $f:= f'\circ F: M\to B\Gamma$ and denote by $\tM\to M$ the associated $\Gamma$-covering.

Let $S$ (resp. $S'$) be in this section the Clifford bundle of even degree differential forms on $M$ (resp. $M'$). We shall associate with this data an analytic class $[F,f,f']\in \maS_1(\Gamma)$. This is exactly the same construction as in \cite{HigsonRoe2010}, we give it here for completeness.

Let us recall first the alternative description of $K_0(A)$ for a unital $C^*$-algebra $A$ with a closed ideal $J$. Denote by $\pi:A\rightarrow A/J$ the obvious projection map. By abuse of notation we will also denote by $\pi$ all the induced maps $M_n(A)\rightarrow M_n(A/J)$ for all $n>0$.

\begin{definition}
\label{checkK}
$\check{K}_0(A)$ is defined to be the Grothendieck group of homotopy classes of triples $(\gamma,\tilde{p},p)$ where:
\begin{enumerate}
\item $p\in M_\infty(A/J)$ is a projection.
\item $\tilde{p}\in M_\infty(A)$ is a lift of $p$, i.e. $\pi(\tilde{p})=p$
\item $\gamma:[0,2]\rightarrow GL_\infty(J)$ is a path such that $\gamma(0)=\exp(2\pi i \tilde p)$ and $\gamma(2)= I$.
\end{enumerate}
\end{definition}
\begin{remark}
$GL_n(J)$ is the algebra of invertible operators in $M_n(A)$ which are identity modulo $M_n(J)$, and $GL_\infty(J)$ is endowed with the usual inductive limit topology, under such a topology each path $\gamma$ lies entirely in some $GL_n(J)$.
\end{remark}

\begin{remark}
By a homotopy of triples $(\gamma,\tilde{p},p)$ we mean the following: two triples $(\gamma_0,\tilde{p}_0,p_0)$ and $(\gamma_1,\tilde{p}_1,p_1)$ are homotopy equivalent if there exist:
\begin{enumerate}
\item a path of projections $p_s, s\in [0,1]$ in $M_\infty(A/J)$ connecting $p_0$ and $p_1$,
\item $\tilde{p}_s\in M_\infty(A)$ is a lift of $p_s$ ,
\item for each $s\in [0,1]$, $\gamma_s:[0,2]\rightarrow GL_\infty(J)$ is a path such that $\gamma_s(0)=\exp(2\pi i \tilde p_s)$ and $\gamma_s(2)=I$.
\end{enumerate}

\end{remark}
We can now define a map $\psi: K_0(A)\rightarrow \check{K}_0(A)$ by setting for any $q\in \Proj_n(A)$,
$$
\psi([q]):= [(\gamma,q,\dot{q})]
$$
where
\begin{enumerate}
\item $\dot{q}$ is the image of of $q$ in $\Proj_n(A/J)$,
\item $\gamma(t)=I$ for all $t\in [0,1]$.
\end{enumerate}

Notice that for a projection $q\in M_n(A)$ one has $\exp(2\pi iq)=I$.

\begin{lemma}%[Lemma 7.4, \cite{HigsonRoe2010}]
The map $\psi$ is an isomorphism of the abelian groups $K_0(A)$ and $\check{K}_0(A)$.
\end{lemma}

\begin{proof}
We refer to Lemma 7.4, \cite{HigsonRoe2010} for a detailed proof.
\end{proof}

  Let $X$ be a finite $CW$-complex and $H$ be a separable Hilbert space with an ample representation of $C(X)$. Using the above alternative description of $K_0$, we can describe the structure groups $\maS_{1,\Gamma}(X,H)$ by taking $A=D^*_\Gamma(X)$ and $J=\maK_{C^*\Gamma}(\maE_{S,\Gamma})$ in definition (\ref{checkK})(see sections 2 and 3 for notations). Then any class in $\maS_{1,\Gamma}(X,H)$ is given by:

\begin{enumerate}
\item a projection $P$ in $Q^*_H(X)\simeq Q^*_\Gamma(X)=D^*_\Gamma(X)/\maK_{C^*\Gamma}(\maE_{S,\Gamma})$,
\item a lift $\maP$ of $P$ to $D^*_\Gamma(X)$,
\item a path $\gamma$ connecting $\exp(2\pi i \maP)$ to the identity via invertible operators in $\maL(\maE_{S,\Gamma})$ that are compact perturbations of the identity on $\maE_{S,\Gamma}$.
\end{enumerate}

We will now describe briefly the construction of the class $[F,f,f']\in \maS_1(\Gamma)$. Let us consider the geometric cycle $(M,S,f)\sqcup -(M',S',f')$. We now choose $X$ to be a finite subcomplex of $B\Gamma$ such that $f(M)\sqcup f'(M')\subseteq X$; we shall initially define the class $[F,f,f']$ as an element in $\maS_{1,\Gamma}(M\sqcup M', L^2(M,S)\oplus L^2(M',S'))$, and then using the covariant functor $\maS_1$, along with an ample representation $H$ of $C(X)$, map it to $\maS_{1,\Gamma}(X,H)$. This mapping will be assumed implicitly in the sequel.

Let $d$ (resp. $d'$) and $*$ (resp. $*'$) be the de Rham differential operator and Hodge $*$-operator respectively on $M$ (resp. $M'$). Since $S$ is taken to be the exterior algebra bundles one can define the odd signature operator:

\begin{definition}
 The signature operator on even degree forms is given by
\begin{equation*}
D^{sign}_{even}=i^{m+1}(-1)^{k+1}(*d-d*) \text{ for forms of degree
$2k$}
\end{equation*}
where $\dim M=2m+1$.
\end{definition}
 We shall denote the signature operator on $M$ and $M'$ simply by $D$ and $D'$, respectively. Let $\tD$ (resp. $\tD'$) be the $\Gamma$-equivariant lifts of $D$ (resp. $D'$) to $\tM$ (resp. $\tM'$). The lifted operator $\tD$ induces an unbounded regular operator $\maD$ on the Hilbert $C^*\Gamma$-module $\maE_{S,\Gamma}$. The de Rham operator $d$ induces a regular operator $d_\Gamma$ and the Hodge $*$-operator induces an adjointable operator $\star$ on $\maE_{S,\Gamma}$.
Define operators $\maJ = i^{p(p-1)+m} \star$ on forms of degree $p$, and $\maB=d_\Gamma+d^*_\Gamma$ on $\maE_{S,\Gamma}$. $\maJ$ is a self-adjoint involution which anticommutes with $\maB$. Then we have the following relations:
$$\maD=i\maB\maJ,\quad (\maD-iI)(\maD+iI)^{-1}=(\maB-\maJ)(\maB+\maJ)^{-1}$$

We call $(\maD-iI)(\maD+iI)^{-1}$ the Cayley transform of $\maD$.

We now consider the operator $P=\frac{1}{2}(sign(D)+I)$ viewed as a projection in $Q^*_{L^2(M,S)}(M)$, which determines the $K$-homology class of $D$. We define
$$\maP=\frac{1}{2}(\frac{2}{\pi}\arctan(\maD)+I)$$
The operator $\maP$ is then a lift of $P$ to $D^*_\Gamma(M)$. Thus to define a structure class associated with the cycle $(M,S,f)$ one needs a path $\gamma$ connecting the identity to $\exp(2\pi i\maP)=(\maD-iI)(\maD+iI)^{-1}$.

 Let $P',\maP',\maB',\maJ',\maD'$ be the analogously defined operators on $(M',S',f')$. We also define:
\begin{enumerate}
\item the operator $\hat{P}=P\oplus -P'$,
\item a lift $\hat{\maP}=\maP\oplus -\maP'$ of $\hat{P}$
\end{enumerate}
Finally, to define a structure class $[F,f,f']$ associated with the geometric cycle $(M,S,f)\sqcup -(M',S',f')$ and homotopy equivalence $F:M\rightarrow M'$, we need to connect the Cayley transform of $\hat{\maD}:=\maD\oplus -\maD'$ to the identity on $\maE_{S,\Gamma}\oplus \maE'_{S',\Gamma}$. We define an adjointable map on Hilbert modules
$$\maA= \exp(-\maD^2)\maF\exp(-\maD'^2): \maE'_{S',\Gamma}\rightarrow \maE_{S,\Gamma}$$

where $\maF$ is the operator induced by $F$ on the level of Hilbert modules. The above operator $\maA$ gives a homotopy equivalence of Hilbert-Poincar\'{e} complexes in the sense of \cite{HigsonRoe} (see also \cite{BR}) and therefore induces an isomorphism on the (unreduced) de Rham cohomology groups. Let $\hat{\maB}:= \maB\oplus -\maB'$.

\begin{definition}
Define the path $\Sigma_t$ of duality operators (again in the sense of \cite{HigsonRoe}) on the Hilbert-Poincar\'{e} complex associated with $(M,S,D,f)\sqcup-(M',S',D',f')$ as follows:
\begin{equation}
\Sigma_t=\left\{ \begin{array}{ccc} \left(\begin{array}{ccc} \maJ & 0 \\ 0 & -2t\maA\maJ\maA^*+(2t-1)\maJ'\end{array}\right) \text{ for } t\in [0,\frac{1}{2}]\\
\left(\begin{array}{ccc} \sin(\pi t)\maJ&\cos(\pi t)\maJ\maA^*\\ \cos(\pi t)\maA\maJ &-\sin(\pi t)\maA\maJ\maA^*\end{array}\right) \text{ for } t\in [\frac{1}{2},1]\\
\left(\begin{array}{ccc} 0 & \exp(\pi it)\maJ\maA^*\\ \exp(-\pi it)\maA\maJ &0\end{array}\right) \text{ for } t\in [1,2]\\
\end{array}\right.
\end{equation}
\end{definition}

We then define the path $\hat{\sigma}_t$ of unitaries as follows:

\begin{definition}
Define the path $\hat{\sigma}_t$ as follows:
\begin{equation}
\hat{\sigma}_t=\left\{ \begin{array}{ccc} (\hat{\maB}-\Sigma_t)(\hat{\maB}+\Sigma_t)^{-1} \text{ for } t\in [0,1]\\
(\hat{\maB}-\Sigma_1)(\hat{\maB}+\Sigma_t)^{-1} \text{ for } t\in [1,2]\\
\end{array}\right.
\end{equation}
\end{definition}

Therefore the class $[F,f,f']\in \maS_1(\Gamma)$ is given by the projections $\hat{P}$, $\hat\maP= \frac{1}{2}(\frac{2}{\pi}\arctan(\hat{\maD})+1)$ and the path $\hat{\sigma}_t$ connecting the Cayley transform $(\hat\maD-i)(\hat\maD+i)^{-1}$ to the identity. Note that the path $\hat{\sigma}_t$ consists of invertible operators which are compact perturbations of the identity, cf. \cite{HigsonRoe}.

The image under $\alpha_*$ of the analytic class $[F,f,f']$ in the $\ell^2$ structure group is given by the class, denoted by $H$, of the couple $(\tilde{P}\oplus -\tP',P\oplus -P')$ together with the pair of paths $(\tilde{\sigma}_t,\sigma_t)$ connecting the pair of Cayley transforms $((\tD-i)(\tD+i)^{-1}\oplus (\tD'+i)(\tD'-i)^{-1},(D-i)(D+i)^{-1}\oplus (D'+i)(D'-i)^{-1})$ to the pair $(Id_{L^2(\tilde{M},\tilde{S})}\oplus Id_{L^2(\tilde{M}',\tilde{S}')}, Id_{L^2(M,S)}\oplus Id_{L^2(M',S')})$.

Our goal will be to show the following lemma:
\begin{lemma}
\label{analyticgeometricclass}
Under the isomorphism given in Theorem (\ref{analyticGeometric}), the class $H \in \maS_1^{(2)}(\Gamma)$ is the analytic structure class of the geometric cycle
$$ (M, S, f, D, h-\tilde h)\sqcup -(M',S',f',D',h'-\tilde h')=(M, S, f, D, h-\tilde h)\sqcup (M',-S',f',-D',0)$$
\end{lemma}

\begin{remark}
A relevant feature that differs from the corresponding statement in \cite{HigsonRoe2010} is our use of the the ``shifted" cycle $(M, S, f, D, h-\tilde h)\sqcup -(M',S',f',D',h'-\tilde h')$ to identify the image of the class $[F,f,f']$ in $\maS_1^{(2)}(\Gamma)$. This is due to the appearance of the spectral projections $\chi_\geq(D)$, instead of $\chi_>(D)$ as in \cite{HigsonRoe2010}, in our definition of the analytic structure class of a geometric cycle.
\end{remark}

Recall the map $\xi_{(2)}: \maS_1(\Gamma)\rightarrow \maS_1^{(2)}(\Gamma)\simeq \maS_1^{(2),geo}(\Gamma)\xrightarrow{\xi}\R$. Recall that the $\ell^2$ rho invariant of $D$ is the real number $\rho_{(2)} (D):= \eta_{(2)} (\tD) - \eta (D)$. As a corollary of Lemma (\ref{analyticgeometricclass}) we obtain the following important theorem on the homotopy invariance of Cheeger-Gromov rho-invariants:

\begin{theorem}\label{Homotopy}\cite{KeswaniVN}
Assume the data $(M,M',S,S',F,f,f',D,D')$ given above. Assume also that the assembly map $\mu_\Gamma$ described in Section \ref{analytic}  is an isomorphism. Then we have
$$
\rho_{(2)}(D)=\rho_{(2)}(D').
$$
\end{theorem}

\begin{proof}
We have $\xi_{(2)}([F,f,f'])= \xi((M, S, f, D, h-\tilde h)\sqcup -(M',S',f',D',h'-\tilde h'))$. However, since $\mu_\Gamma:  K_*(B\Gamma) \rightarrow K_*(C^*\Gamma)$ is an isomorphism by hypothesis, we have $\maS_1(\Gamma)=\{0\}$. Thus $\xi_{(2)}([F,f,f'])=0$.
Finally we note that
\begin{eqnarray*}
\xi((M,S,f,D,h-\tilde h)\sqcup(M',-S',f',-D',0))&=&\rho_{(2)}(M, S,f,D)-(\tilde h-h)+\rho_{(2)}(M', -S', f',-D')\\
&=& \frac{1}{2}(\eta_{(2)}(\tilde{D})-\eta(D)+\tilde h-h)- (\tilde h-h)\\&+& \frac{1}{2}(-\eta_{(2)}(\tilde{D}')+\eta(D')+\tilde h'-h')\\
&=& \frac{1}{2}((\eta_{(2)}(\tilde{D})-\eta(D))-(\eta_{(2)}(\tilde{D}')\\&-&\eta(D')))-\frac{1}{2}((\tilde h-\tilde h')-(h-h'))\\
&=& \frac{1}{2}(\rho_{(2)}(D)-\rho_{(2)}(D'))
\end{eqnarray*}
where in the last step we have used the homotopy invariance of the ordinary and $\ell^2$-Betti numbers.
Hence $\rho_{(2)}(D)-\rho_{(2)}(D')=0$.
\end{proof}

We shall now carry out the proof of Lemma (\ref{analyticgeometricclass}). We remark that the arguments in this proof are almost exactly the same as the elegant arguments given in \cite{HigsonRoe2010}, the only new ingredient is the compatibility of functional calculi of the lifted Dirac operator $\maD$ on the Hilbert module $\maE_{S,\Gamma}$ and that of the operator on the covering $\tD$ affiliated to the von Neumann algebra $\maM_M$.  We refer to \cite{BenameurPiazza} for a detailed proof.

\begin{proof}(of Lemma (\ref{analyticgeometricclass}))
Recall that the homotopy equivalence $F: M\rightarrow M'$ induces a map on Hilbert modules
$$\maA= \exp(-\maD^2)\maF\exp(-\maD'^2): \maE'_{S',\Gamma}\rightarrow \maE_{S,\Gamma}$$

Using the isomorphisms $\Psi_{reg}$ and $\Psi_{av}$ we get operators $\tilde{A}: L^2(\tM',\tS')\rightarrow L^2(\tM,\tS)$
and $A: L^2(M',S')\rightarrow L^2(M,S)$.
Since we have passed from Hilbert $C^*$-modules to Hilbert spaces we can use the Borel functional calculi on the von Neumann algebras $\maM_M$ and $B(L^2(M,S))$. In particular the projection $\tilde{\Pi}$ (resp. $\Pi$) onto the space of $\ell^2$-harmonic forms $\tilde{\maH}$(resp. onto the space of harmonic forms $\maH$ on $L^2(M,S)$)  belongs to the von Neumann algebra $\maM_M$ (resp. $B(L^2(M,S))$) and is $\tau$-compact (resp. compact). Let $\tilde{\Pi}'$ and $\Pi'$ be defined analogously on $M'$.
These projections induce the identity on $\ell^2$ and ordinary de Rham cohomologies, respectively.

By considering linear homotopies we can replace $\tilde{A}$ and $A$ in the paths $(\tilde{\sigma}_t,\sigma_t)$ with the operators
$$
\tilde{A}_1= \tilde{\Pi}\tilde{F}^*\tilde{\Pi}' \text{ and } A_1= \Pi F^*\Pi',
$$
where $\widetilde{F}:\tM\rightarrow \tM'$ is the lift of $F: M\rightarrow M'$.

We now consider the decomposition of $(L^2(\tM,\tS)\oplus L^2(\tM',\tS'))\oplus (L^2(M,S)\oplus L^2(M',S'))$ as a direct sum of the spaces of harmonic forms $(\tilde\maH\oplus \tilde\maH')\oplus (\maH\oplus \maH')$ and its orthogonal. On the orthogonal subspace of the harmonic forms we have the following description of the paths $(\tilde{\sigma}_t,\sigma_t)$:
$$
\tilde{\sigma}_t= (\hat{\tilde{\bB}}-\hat{\tilde{\bJ}}_t)(\hat{\tilde{\bB}}+\hat{\tilde{\bJ}}_t)^{-1}, \text{ and }
\sigma_t=(\hat{B}_t-\hat{J}_t)(\hat{B}_t+\hat{J}_t)^{-1}
$$
where $\hat{\tilde{B}}= \left( \begin{array}{ccc} \tilde B& 0\\ 0 & -\tilde B'\end{array} \right)$, and
$\hat{\tilde\bJ}_t$ is given by:

\begin{equation}
\hat{\tilde\bJ}_t=\left\{ \begin{array}{ccc} \left( \begin{array}{ccc} \tilde\bJ & 0\\ 0 & (2t-1)\tilde\bJ'\end{array} \right) \text{ for } 0\leq t\leq \frac{1}{2} \\
\left( \begin{array}{ccc} \sin(\pi t)\tilde\bJ & 0\\ 0 & 0\end{array} \right) \text{ for } \frac{1}{2}\leq t\leq 1\\
\left( \begin{array}{ccc} 0 & 0\\ 0 & 0\end{array} \right) \text{ for } 1\leq t\leq 2\\
\end{array}
\right.
\end{equation}

The same equations hold with the operators without the tilde's.

The path $\tilde\sigma_t$ is homotopic (with fixed end-points) to the path $(\hat{\tilde\bB}-s\hat{\tilde\bJ})(\hat{\tilde\bB}+s\hat{\tilde\bJ})^{-1}, s\in [0,1]$ through the following straight-line homotopy defined for $0\leq s\leq 1, 0\leq u\leq 1$:
$$
\tilde\sigma_u(s,t)=\begin{cases}
\left(\begin{array}{ccc} (us+(1-u))\tilde\bJ & 0 \\ 0& (us+(1-u)(2t-1))\tilde\bJ'\end{array}\right) &\mbox{ for } 0\leq t\leq \frac{1}{2},\\
\left(\begin{array}{ccc} (us+(1-u)sin(\pi t))\tilde\bJ & 0 \\ 0& (us)\tilde\bJ'\end{array}\right) &\mbox{ for } \frac{1}{2}\leq t\leq 1,\\
\left(\begin{array}{ccc} (us)\tilde\bJ & 0 \\ 0& (us)\tilde\bJ'\end{array}\right) &\mbox{ for } 1\leq t\leq 2,
\end{cases}
$$
Let $\tilde I=id_{L^2(\tM,\tS)}, \tilde I' = id_{L^2(\tM',\tS')}$, $I= id_{L^2(M,S)}, I'=id_{L^2(M',S')}$.
 Now, since for any  $\lambda\neq 0$,  $\frac{2}{\pi}\arctan(s^{-1}\lambda)$ is uniformly bounded in $s$ and converges pointwise to $sign(\lambda)$ as  $s\to 0$,
we find that as $s\rightarrow 0$, $(\tilde\bB-s\tilde\bJ)(\tilde\bB+s\tilde\bJ)^{-1}\rightarrow \exp(2\pi i \chi_>(\tD))=\tilde{I}$ strongly on $\tilde\maH^\perp$. Since  for each $s\in [0,1]$,
$$
\frac{1}{2}\left(\frac{2}{\pi}\arctan(s^{-1}(\tilde{D}\oplus -\tilde{D}')+\tilde I), sign(D\oplus -D')+I \right)\in D^*_{(2)}(X,H)
$$
the image under $\alpha_*$ of the structure class $[F,f,f']$ in $\maS^{(2)}_1(\Gamma)$ is represented by
$$
(\chi_>(\tilde{D})\oplus \chi_>(-\tD'),\chi_>(D)\oplus \chi_>(-D')),
$$
together with the constant path $(\tilde I\oplus \tilde I')\oplus (I\oplus I')$ restricted to the orthogonal complement of the harmonic forms.
Notice that this structure class is also represented by the operator pair
$$
(\chi_\geq(\tilde{D})\oplus \chi_\geq(-\tD'),\chi_\geq(D)\oplus \chi_\geq(-D'))- (\tilde{\Pi}\oplus \tilde{\Pi}',\Pi\oplus \Pi')$$
again with the constant path.

On the space of harmonic forms $\tilde \maH$, we get $\tilde{P}=\frac{1}{2}\tilde I, P= \frac{1}{2} I$ and the path $\tilde \sigma_t$ is given by

\begin{equation}
\tilde\sigma_t=\left\{ \begin{array}{ccc}-\tilde I\oplus -\tilde I' \text{ for } 0\leq t\leq 1 \\
-\hat{\tilde{\bJ}}_1\hat{\tilde{\bJ}}_t^{-1} \text{ for } 1\leq t\leq 2
\end{array}\right.
\end{equation}

Noting that $\tilde A_1, \tilde{A}^*_1$ and $\tilde{J}'$ induce isomorphisms on cohomology and that $\tilde{J}'$ is an involution, we find for $t\in [1,2]$,
$$
\tilde\sigma_t=-\hat{\tilde\bJ}_1\hat{\tilde\bJ}_t^{-1}= \left(\begin{array}{ccc} \exp(-\pi i t)\tilde I & 0 \\0 & \exp (\pi i t)\tilde I'\end{array}\right)
$$

We now deform the matrix $\left( \begin{array}{ccc} \frac{1}{2}\tilde I &0 \\ 0 & \frac{1}{2}\tilde I'\end{array}\right)$ and the path $\tilde{\sigma}_t$ for $1\leq t\leq 2$.
For $0\leq s\leq 1$ define the path of operators $\tilde P_s$ as follows:
$$
\tilde P_s= \left( \begin{array}{ccc} \frac{1-s}{2}\tilde I &0 \\ 0 & \frac{1+s}{2}\tilde I'\end{array}\right)
$$

and consider the path of invertible operators

$$
\tilde\sigma_{s}(t) = \left(\begin{array}{ccc} \exp(-2\pi i (s+\frac{(1-s)}{2}t))\tilde I & 0 \\0 & \exp (2\pi i (s+\frac{(1-s)}{2}t))\tilde I'\end{array}\right) \quad \text{  for $1\leq t\leq 2, 0\leq s\leq 1$. }
$$

The path $\tilde P_s$ is a norm continuous path of operators that connects $\tilde{P}\oplus \tilde{P}'$ to the operator $\left(\begin{array}{ccc}0 & 0 \\0 & \tilde I'\end{array}\right)$, with each $(\tilde P_s,P\oplus P')\in D^*_{(2)}(X,H)$, since $\tilde{I}$ is $\Gamma$-compact on the space of harmonic forms. Therefore
$$(\tilde\sigma_{s}, \tilde P_s\oplus (\frac{1}{2}I\oplus \frac{1}{2}I')),\quad s\in [0,1]$$ gives a homotopy of triples leaving the structure class determined by $(\tilde{P}\oplus \tilde{P}',P\oplus P')$ and the path $\tilde\sigma_t$ invariant. Therefore by setting $s=1$, we get that on $\tilde\maH\oplus \tilde\maH'$ the structure class is given by the operator $\left(\begin{array}{ccc}0 & 0 \\0 & \tilde I'\end{array}\right)$ and the constant path.

Combining the two components on $(\tilde\maH\oplus \tilde\maH')\oplus (\maH\oplus \maH')$ and $((\tilde\maH\oplus\tilde\maH')\oplus (\maH\oplus \maH'))^\perp$ we find therefore that the structure class in $\maS_1^{(2)}(\Gamma)$ is given by constant path together with the operator
$$
\left[ \left(\begin{array}{ccc}\chi_\geq(\tD) & 0 \\0 & \chi_\geq(-\tD')\end{array}\right),\left(\begin{array}{ccc}\chi_\geq(D) & 0 \\0 & \chi_\geq(-D')\end{array}\right)\right]+\left[\left(\begin{array}{ccc}-\tilde\Pi & 0 \\0 & 0\end{array}\right),\left(\begin{array}{ccc}-\Pi & 0 \\0 & 0\end{array}\right)\right]
$$

Noting that $[-\tilde\Pi,-\Pi]=[-(\tilde{h}-h)]$, where as before $\tilde{h}=\dim_{(2)}\Ker \tD, h=\dim \Ker D$, we see that this is the structure class of the geometric cycle
$$
(M,S,f,D,h-\tilde h)\sqcup(M',-S',f',-D',0)= (M,S,f,D,h-\tilde h)\sqcup -(M',S',f',D',h'-\tilde h')
$$

This finishes the proof of the lemma.

\end{proof}

\bigskip

\appendix

\section{$K$-theory of $\Gamma$-compact operators}\label{Compact}
We give for the reader's convenience the proof of the following folklore lemma (we thank T. Fack for helpful discussions):

\begin{lemma}
\label{K}
We have the following identification of $K$-theory groups for the $\Gamma$-compact operators in the von Neumann algebra $(\maM_X,\tau)$ described in Section 3:
\begin{enumerate}
\item $K_0(\maK(\maM_X,\tau))\cong \R$
\item $K_1(\maK(\maM_X,\tau))=0$
\end{enumerate}
\end{lemma}

\begin{proof}
First note that the ideal of $\tau$-compact operators $\maK(\maM_X,\tau)$ can be identified with $\maK(L^2(X,S))\otimes \maN\Gamma$, where $\maK(L^2(X,S))$ is the ideal of compact operators on the Hilbert space of $L^2$-sections of $S$ over  $X$, and $\maN\Gamma$ is the group von Neumann algebra associated with $\Gamma$.  Let us denote this isomorphism by $\maK(\maM_X,\tau)\xrightarrow{\cong} \maK(L^2(X,S))\otimes \maN\Gamma$.

By the stability of $K$-theory for $C^*$-algebras, we then have $K_i(\maK(\maM_X,\tau))\cong K_i(\maN\Gamma), i=0,1$.

There is a compact metric space $Z$ (the center) with a standard Borel structure $(\mathfrak{B},\nu)$ such that $\nu$ is a positive Radon measure with support $Z$, and there is a direct integral decomposition of $\maN\Gamma$ (cf. \cite{Dixmier}, Chapitre II, page 210):
$$
 \maN\Gamma=\int^\oplus_Z \maN(\zeta) d\nu(\zeta),
 $$
where $\maN(\zeta)$ is a type II$_1$-factor for $\zeta\in Z,\text{ }\nu-a.e.$ Moreover, the trace $\tau$ corresponding to  evaluation at the neutral element $e$ on $\maN\Gamma$ decomposes as well:
$$
\tau= \int^\oplus_Z \tau_\zeta d\nu(\zeta).
$$

Consider the group homomorphism  $\tau_{*}: K_0(\maN\Gamma)\rightarrow \R$ induced  on $K$-theory by the trace $\tau$. This is our required isomorphism.

Let $p,q\in \Proj(M_n(\maN\Gamma))$ be such that $\tau_{ *}([p])=\tau_{ *}([q])$ and set $p=\int^\oplus_Z p(\zeta)d\nu(\zeta)$ and $q= \int^\oplus_Z q(\zeta)d\nu(\zeta)$. We have $p(\zeta),q(\zeta)\in \Proj(M_n(\maN(\zeta)))$ for $\zeta \in Z, \text{} \nu-a.e.$. By a theorem of Murray and von Neumann \cite{MvN}, for any two projections $p(\zeta)$ and $q(\zeta)$ in the factor $M_n(\C)\otimes \maN(\zeta)$, there exists a partial isometry $v(\zeta)\in M_n(\C)\otimes \maN(\zeta)$ such that $p(\zeta)=v(\zeta)v(\zeta)^*$ and $v(\zeta)^*v(\zeta)q(\zeta)=v(\zeta)^*v(\zeta)$ (the $v(\zeta)$'s can be chosen measurably). In particular, since replacing $p(\zeta)$ with its equivalent $v(\zeta)^*v(\zeta)$ leaves the trace invariant, one can assume $p(\zeta)\leq q(\zeta)$ for almost every $\zeta\in Z$. Therefore, as we have
$$
\int_Z \tau_\zeta(p(\zeta)) d\nu(\zeta) = \int_Z \tau_\zeta(q(\zeta)) d\nu(\zeta)
$$
the positivity of $\nu$  implies that, up to replacing almost everywhere $p(\zeta)$ by the equivalent projection $v(\zeta)^*v(\zeta)$, we may assume that  $\tau_\zeta(p(\zeta))=\tau_\zeta(q(\zeta)), \text{ for } \zeta \in Z, \text{ }\nu-a.e.$ Since two projections in a II$_1$-factor are Murray-von Neumann equivalent if and only if their traces are equal, we have that $p(\zeta)\sim_{MvN} q(\zeta)$ for $\nu-a.e.\text{ } \zeta$.  One can choose a measurable field of partial isometries $u(\zeta)\in \maN(\zeta)$ such that $p(\zeta)=u(\zeta)u(\zeta)^*$ and $q(\zeta)=u(\zeta)^*u(\zeta)$ (this requires a measurable selection theorem on analytic subsets of Polish spaces, see for instance \cite{TakesakiI}, Appendix A). The operator $$u=\int^\oplus_Z u(\zeta) d\nu(\zeta)$$ is therefore well-defined (as it is essentially-bounded) and belongs to $M_n(\C)\otimes \maN\Gamma$, and we have
$$ p= uu^*, q= u^*u$$
This gives the required Murray-von Neumann equivalence of projections, implying that $[p]=[q] \in K_0(\maN\Gamma)$. Thus $\tau_{ *}$ is injective.

Since $Z$ is a compact metric space, it has a Lebesgue number $\delta>0$. To show that $\tau_{*}$ is surjective,  it suffices  to show that its range contains the subset $[0,\delta/2]$.  Let $x$ be any number in $[0,\delta/2]$. As $\nu$ is standard, we can consider a measurable subset $Y$ with $\nu(Y)=x$ such that the Hilbert space  $H_\zeta$ is non zero  for $\zeta\in Y$.  Let $p_0$ be a projection defined by
 $$
 p_0 = \int^\oplus_Z p_0(\zeta) d\nu(\zeta)
 $$
where $p_0(\zeta)$ is the measurable family given by $p_0(\zeta)=id_{\maN(\zeta)}$ for $\zeta\in Y$ and zero on $Z\smallsetminus Y$. Then we have
$$
\tau(p_0) =\int_Z \tau_\zeta( p_0(\zeta)) d\nu(\zeta)= \int_Z\chi_Y(\zeta) d\nu(\zeta) = \nu(Y)=x,
$$
where we have used the fact that the traces $\tau_\zeta$ can be chosen to be normalized so that $\tau_\zeta(I)=1$.

The second item follows immediately from the isomorphism $K_1(\maK(\maM_X,\tau))\cong K_1(\maN\Gamma)$. No decomposition theorems are needed here, and we just use the well-known connectedness to the group of invertibles.
\end{proof}

\begin{remark}
The above proof shows of course the same statement for the von Neumann algebra $\maM_{X, H}$ associated with any ample representation in a separable Hilbert space $H$.
\end{remark}

\section{$\tau$-compactness of resolvent}\label{Resolvent}

We first prove an important lemma which extends to the semi-finite setting Lemma 1.2 and Proposition 1.1 in \cite{BaumDouglasTaylor} and which plays an important role in some boundary value constructions. Fix  two separable Hibert spaces $H_0$ and $H_1$ and a $\Z_2$-graded  semi-finite von Neumann algebra $\maM$ which is faithfully represented in $H=H_0\oplus H_1$. We denote by $\maM_0$ and $\maM_1$ the von Neumann subalgebras of $B(H_0)$ and $B(H_1)$ respectively corresponding to the left upper corner and the right lowed corner.  We assume that (and hence each $\maM_i$)  is endowed with  the faithful normal  semi-finite positive trace $\tau$. So,  $\maM$ is semi-finite von Neumann algebra which is faithfully represented in a  $\Z_2$-graded separable Hilbert space $H$ and which inherits a $\Z_2$ grading from the usual one on $B(H)$. Denote as usual by $\maK (\maM, \tau)$  the $C^*$-algebra of $\tau$-compact operators in $\maM$, see for instance  \cite{Benameur03, BenameurFack}.

Let $A$ be a densely defined closed operator from $H_0$ to $H_1$ such that  $A^*A$ (resp. $AA^*$) is affiliated with the von Neumann algebra $\maM_0$ (resp. $\maM_1$).

Denote by $P$ the orthogonal projection onto $\Ker (A)^\perp\subset H_0$, which is then an element of $\maM_0$. Set $B= \left(\begin{array}{cc}   0 & A^* \\ A & 0 \end{array} \right)$. Then the following is an extension to the semi-finite setting of the classical Lemma 1.2 and of part of Proposition 1.1 in \cite{BaumDouglasTaylor}. For the proof of the first item of this lemma, the authors benefited from an encouraging helpful discussion with Georges Skandalis.

\begin{lemma}\label{Lemma9}\
Assume that $AA^*$ has $\tau$-compact resolvent in $\maM_1$. Then
\begin{enumerate}
\item The  operator  $P (A^*A + I)^{-1/2} P$ is $\tau$-compact in $\maM_0$.
\item  Assume in addition that  $\pi: \maU \to \maM$ is a $*$-representation of a $C^*-$algebra $\maU$  through even-graded operators such that for any $c$ is some dense $*$-subalgebra  $\maU_0$ of $\maU$, $\pi (c)$ preserves the domain of $B$ and $[B, \pi (c)]$ extends to a bounded operator in $\maM$. Then for any $c\in \maU$, the operator $[\pi (c), B (I+B^2)^{-1/2}]$ is well defined and is a $\tau$-compact operator in $\maM$.
\end{enumerate}
\end{lemma}

\begin{proof}
For a subset $Y\subseteq \R$ we denote by $\chi_Y$ the characteristic function of $Y$. As before, we simply denote by $\chi_\geq$ and $\chi_>$ the characteristic functions of $[0, +\infty[$ and $]0, +\infty[$ respectively.
First note that $\chi_{]0,M]}=\chi_{>}\chi_{[0,M]}$.

We thus get using that $\chi_{[0,M]}(A^*A)$ preserves $\Ker A^\perp$,
$$
\chi_{]0,M]}(A^*A)=\chi_{[0,M]}(A^*A)|_{\Ker A^\perp}.
$$
Hence
$$
\chi_{[0,M]}(A^*A|_{\Ker A^\perp}) = \chi_{]0,M]}(A^*A).
$$
Set now $A=U|A|$ for the polar decomposition of $A$. Then  the restriction $V$ of $U$ to $\Im(\chi_{]0,M]}(A^*A))$ furnishes a unitary isomorphism $
V: \Im(\chi_{]0,M]}(A^*A))\longrightarrow \Im(\chi_{]0,M]}(AA^*)).$ Since $U^*U$ and $UU^*$ belong to $\maM_0$ and $\maM_1$ respectively,
we deduce that
$$
 \tau \left(\chi_{[0,M]}(A^*A|_{\Ker A^\perp})\right)= \tau \left(\chi_{(0,M]}(AA^*)\right) < +\infty,
 $$
since $AA^*$ has $\tau$-compact resolvent in $\maM_1$ \cite{FackKosaki}. Using again classical arguments, we deduce that $A^*A|_{\Ker A^\perp}$ has $\tau$-compact resolvent in $P\maM_0 P$  and hence in $\maM_0$. This finishes the proof of the first item.

The proof of the second item is inspired from the similar proof in the type I case given in \cite{BaumDouglasTaylor}. We may assume by density of $\maU_0$ and since $\maK (\maM, \tau)$ is closed in $\maM$, that $c\in \maU_0$ so that $\pi(c)$ preserves  the domain of $B$ and yields a bounded commutator. It will then be obvious that for $c\in \maU$, the commutator is well defined as  a choices-independent  limit  which is automatically $\tau$-compact.

We may now write,  on the domain of $B$, the odd operator $[\pi(c), B(I+B^2)^{-1/2}]$ as:
$$
[\pi(c), B(I+B^2)^{-1/2}] = [\pi(c), (I+B^2)^{-1/2} ] B + (I+B^2)^{-1/2} [\pi(c), B].
$$
So, denoting by  $P_0$ the orthogonal projection onto the kernel of $A$ in $H_0$, we have
$$
 [\pi(c), B(I+B^2)^{-1/2}] P_0=  - (I+AA^*)^{-1/2} A \pi (c) P_0.
$$
But  $-A\pi(c) P_0=[\pi (c), B]P_0$ is a bounded operator from $H_0$ to $H_1$ which belongs to $\maM$. Thus, using that $ (I+AA^*)^{-1/2} $ is  $\tau$-compact, we conclude that $[\pi(c), B(I+B^2)^{-1/2}] P_0$ is a $\tau$-compact operator. It remains to treat the operator $[\pi(c), B(I+B^2)^{-1/2}] (I-P_0)$ and we now  use the expression
$$
[\pi(c), B(I+B^2)^{-1/2}](I-P_0)= [\pi (c) , B] (I+B^2)^{-1/2}(I-P_0) + B [\pi(c), (I+B^2)^{-1/2} ] (I-P_0).
$$
Recall that $(I+B^2)^{-1/2} (I-P_0)$ is an even $\tau$-compact operator in $\maM$ by the first item. Therefore the first term in the RHS is $\tau$-compact and we concentrate on the second term. A classical argument using the integral expression of $(I+B^2)^{-1/2}$ allows to reduce to the commutators with the resolvents of $B^2$:
$$
(I+B^2)^{-1/2}  = \frac{2}{\pi} \int _0^{+\infty} (I+B^2+\lambda^2)^{-1} d\lambda.
$$
Notice  that here and since $B^2+I\geq I$, this integral is convergent in the uniform operator norm and that the integrand belongs to $\maK(\maM,\tau)$ for any $\lambda \geq 0$. The rest of the proof is then deduced using again the ideas of the proof of Lemma 1.2 in \cite{BaumDouglasTaylor}  since all involved bounded operators do belong to our von Neumann algebra $\maM$.

\end{proof}

\section{Some results on BVP on coverings}\label{BVP}

We review in this appendix some classical results on boundary value problems and explain how they extend to our semi-finite case. Most of the results were expanded in the seminal book \cite{BW} and some results were first obtained by Calderon and Seeley. See for instance \cite{Calderon, Seeley}. The second author would like to thank Paolo Antonini for many helpful discussions about this section.

Recall the double construction with the notations used in Section \ref{geometric}, so $(\thM, \thS, \thD)$ is an even geometric $\Gamma$-equivariant triple constructed using a Galois cover of the chain $(\hM, \hS, \hD)$ with boundary $(M, S, D)$. The boundary of  $(\thM, \thS, \thD)$ is the $\Gamma$-equivariant triple $(\tM, \tS,  \tD)$ which covers the triple $(M,S,D)$. Notice that we have assumed (which is always possible, see \cite{BW} pp. 52-53) that all structures are of product type near the boundary. Let $N$ be the double manifold $\hM\amalg_M (-\hM)$ and let $\tN$ be the double $\Gamma$-manifold  $\thM\amalg_{\tM} (-\thM)$ obtained similarly. The Clifford bundles $\hS^\pm$ yield a Clifford bundle $S_N^\pm$ over the even dimension closed manifold $N$ and a $\Gamma$-equivariant Clifford bundle $S_{\tN}^\pm$ over $\tN$. More precisely, we can glue the two manifolds over a collar neighborhood of $M$ in $N$ and we have
$$
S_N^+ = \hS^+ \amalg_G \hS^- \text{ and }  S_N^- = \hS^- \amalg_{G^{-1}} \hS^+,
$$
and similarly for $S_{\tN}^\pm$ over $\tN$.

The  Dirac operators $\hD^\pm$ and $\thD^\pm$ extend into the double operator $D_N$ over $N$ and a $\Gamma$-invariant double operator $D_{\tN}$ over $\tN$.   A section $s_N^+$ of $S_N^+$ is a couple $s_N^+=(s_+, s_-)$ with $s_+$ a section of $\hS^+$ over $\hM$ and $s_-$ a section of $\hS^-$ over $-\hM$ and such that $s_-= G s_+$ in a collar neighborhood. It is then classical  to check that the operator $D_N^+$ defined by $D_N^+ s_N^+:=(\hD^+ s_+, \hD^-s_-)$ is  well defined and is an elliptic first order operator on $N$. The similar construction gives the operator $D_N^-$.  We get similarly the $\Gamma$-invariant operators $D_{\tN}^\pm$ over $\tN$. The operators $D_N$ and $D_{\tN}$ are easily shown to be generalized Dirac operators for the natural induced connections on the Clifford bundles $S_N$ and $S_{\tN}$, see \cite{BW}. The following statement for the operator $D_{\tN}$ is a generalization of the classical statement for the operator $D_N$ about invertibility of the double, see \cite{BW} and \cite{XieYu}.

\begin{proposition}\label{L2invertible}
The operator $D_{\tN}^2$ is $L^2$-bounded below. More precisely, there exists a constant $\alpha>0$ such that in $L^2$-norms:
$$
 \| D_{\tN} {\tilde s} \| \geq \alpha \| {\tilde s}\| \quad \text{ for any }  {\tilde s}\in C_c^\infty (\tN, S_{\tN}),
$$
The extended operator $D_{\tN}$ is $L^2$ invertible (with {\underline{bounded}} inverse) and the operator $D_{\tN}^{-1}$ is then a $\Gamma$-invariant  pseudodifferential operator of order $-1$ on $\tN$ \footnote{In this paper, the pseudo differential calculus is always the uniformly supported calculus, see for instance \cite{NWX}}.
\end{proposition}

\begin{remark}
The first statement in Proposition \ref{L2invertible} can be deduced from Theorem 5.1 in \cite{XieYu}. Indeed, we may apply the composition construction for the two representations (regular and average) and  conclude. See also Theorem 4.1 in \cite{Antonini}.
\end{remark}

We shall though give  a direct proof below because some of its steps are used in Section \ref{geometric}, especially the $\Gamma$-equivariant Carleman estimate \ref{Carleman}. Our proof adapts the classical one  in \cite{BW} and relies on new techniques from \cite{XieYu}, it has the advantage of being immediately extendable to other geometric situations, especially  foliations.

\begin{proof}
 If we assume that ${\tilde s} = ({\tilde s}_1, {\tilde s}_2)$ is a smooth compactly supported section of $S_{\tN}^+$ (so ${\tilde s}_2=\tG{\tilde s}_1$ in the collar neighborhood) then applying the Green formula we obtain
$$
<\thD^+{\tilde s}_1, {\tilde s}_2> - <{\tilde s}_1, \thD^-{\tilde s}_2> = \|{\tilde s}_2\vert_{\tM} \|^2.
$$
But due to the relation between $\thD^-$ and $\thD^+$ in the collar neighborhood and the fact that $\tG$ is skew adjoint, and applying this Green formula to $\thD^-$, we deduce the existence of a constant $C_1>0$ such that
\begin{equation}\label{(1)}
\|{\tilde s} \vert_{\tM} \|^2 \leq C_1 \|D_{\tN}^+ {\tilde s}\| \; \|{\tilde s}\|, \text{ for any } {\tilde s}\in C_c^\infty ( \tN, S_{\tN}).
\end{equation}
On the other hand  the unique continuation property can be stated more precisely as Lemma \ref{Carleman1} below. We deduce that in a ($\Gamma$-stable) collar neighborhood $\tU$ of $\tM$ in the double $\tN$, there exists $C_2>0$ such that
\begin{equation}\label{(2)}
\|{\tilde s} \vert_{\tU} \|^2 \leq C_2 \left( \|(D_{\tN} {\tilde s})\vert_{\tU}\|^2 +  \|{\tilde s} \vert_{\tM} \|^2\right).
\end{equation}
Notice that such $\tU$ is simply constructed as the inverse image of a collar neighborhood $U$ of $M$ in the double $N$. Combining the estimates \ref{(1)} and \ref{(2)}, we deduce the existence of $C_3>0$ such that
\begin{equation}\label{(4)}
\|{\tilde s} \vert_{\tU} \|^2 \leq C_3 \left( \|D_{\tN} {\tilde s}\|^2 + \|D_{\tN} {\tilde s}\|\cdot \|{\tilde s} \|\right).
\end{equation}
Now, exactly as in  Appendix A of \cite{XieYu}, see also \cite{Ramachandran}, one proves the existence, for any such $\Gamma$-stable open set $\tU$, of constants $C_4, C_5>0$ such that
\begin{equation}\label{(3)}
\|{\tilde s} \| \leq C_4 \|{\tilde s} \vert_{\tU} \| + C_5  \|D_{\tN} {\tilde s}\|.
\end{equation}
Combining the estimates \ref{(3)} and  \ref{(4)} we get the allowed estimate.

That $D^{-1}_{\tN}$ is a $\Gamma$-invariant  pseudo differential operator of order $-1$ is then classical.
\end{proof}

Adapting the proof of Lemma 8.6 in \cite{BW}, we have   (keeping the above notations and recalling that the Dirac operator $D_{\tN}$ is $\Gamma$-invariant):

\begin{lemma}\label{Carleman1}\
There exist constants $C_4, C_5 >0$ such that:
$$
\|{\tilde s} \| \leq C_4 \|{\tilde s} \vert_{\tU} \| + C_5  \|D_{\tN} {\tilde s}\|.
$$
\end{lemma}

\begin{proof}
We first use the Carleman estimate. In our covering situation, this estimate can be stated as follows. Let ${\tilde x}_0$ be an element of the boundary $\pa\tU$ of $\tU$ and denote by $x_0$ its projection in $U\subset N$. Choose  $r_0>0$ such that  the ball $B=B(x_0, r_0)$ is contained in $U$ and its $\Gamma$-cover $\tB\subset \tU$ can be identified with $B\times \Gamma$. Then for any $\ep>0$, there exists $C_\ep>0$ such that for any $R>0$ we have:
\begin{multline}\label{Carleman}
R \sum_{\gamma\in \Gamma} \int_{u=0}^\ep \int_{y\in \S_u} e^{R(\ep-u)^2} \left< {\tilde\sigma } (u, \gamma; y) , {\tilde\sigma } (u, \gamma; y) \right> du dy \\  \leq C_\ep \sum_{\gamma\in \Gamma} \int_{u=0}^\ep \int_{y\in \S_u} e^{R(\ep-u)^2} \left< D_{\tN} {\tilde\sigma } (u, \gamma; y) , D_{\tN} {\tilde\sigma } (u, \gamma; y) \right> du dy,
\end{multline}
where $\S_u$ is the sphere of radius $r_0+u$. The proof given in \cite{BW} applies with minor changes since we integrate over the pull-back in the cover. We omit it for simplicity and leave it as an exercise. See also \cite{XieYu}.

Now, fix $\lambda >0$ and, as in \cite{XieYu}, denote by $\tU^\lambda$ the metric $\lambda$-neighborhood of $\tU$, then since the metric is $\Gamma$-invariant upstairs, the open set $\tU^\lambda$ is also $\Gamma$-stable. Hence we deduce that there exists constants $C^\lambda_4, C^\lambda_5 >0$ such that:
$$
\|{\tilde s}\vert_{\tU^\lambda}  \| \leq C^\lambda_4 \|{\tilde s} \vert_{\tU} \| + C^\lambda_5  \|D_{\tN} {\tilde s}\|.
$$
Since $N$ is compact, the proof is complete since we may repeat the construction of $\Gamma$-stable neighborhoods $\tU\to \tU^\lambda$ a finite number of times and cover the whole $\Gamma$-manifold $\tN$.
\end{proof}

Recall that ${\tilde b}^0_\pm$ is the trace map restricted to $\Ker(\thD^\pm_{\max})$ and which is valued in $W^{-1/2}(\tM,\thS^\pm\vert_{\tM})$, see Lemma \ref{W-1/2} below. Recall also that the range of ${\tilde b}^0_\pm$ is the so-called  Cauchy data space  $H^{-1/2}(\tD^\pm)$.

\begin{corollary}
The space $H^{-1/2}(\tD^+)\cap H^{-1/2}(\tD^-)$ is trivial.
\end{corollary}

\begin{proof}
The proof in \cite{BW} pp 77-78 extends immediately to our situation of $\Gamma$-covering. We only notice that  the elliptic regularity argument works on the covering manifold and that the double operator is injective by \ref{L2invertible}.
\end{proof}

As a corollary, we can state:

\begin{proposition}\label{Calderon}\
There exist  $0$-th order $\Gamma$-invariant pseudodifferential idempotents  $\tC^\pm=C(\thD^\pm)$ such that if we denote by  $\tC^\pm_s$ their bounded extensions to the Sobolev spaces $W^s(\tM,\thS^\pm\vert_{\tM})$ then
\begin{enumerate}
\item The range of  $\tC_{-1/2}^\pm: W^{-1/2}(\tM,\thS^\pm\vert_{\tM})\rightarrow W^{-1/2}(\tM,\thS^\pm\vert_{\tM})$  is  precisely $H^{-1/2}(\tD^\pm)$.
\item
$
\tC_{-1/2}^++\tG^{-1}\tC_{-1/2}^-\tG=id_{W^{-1/2}(\tM, \thS^+\vert_{\tM})}.
$
\end{enumerate}
\end{proposition}

\begin{proof}
The proof given in \cite{BW} of Theorem 12.4 adapts to the  $\Gamma$-invariant pseudodifferential calculus. Here, for simplicity,  we just apply Theorem 4.5 of \cite{Antonini} again together with the composition construction associated with the regular representation of the group $\Gamma$.
\end{proof}

\begin{remark}
The statement is more precise and one defines explicitly the operators $\tC^\pm$ exactly as in Theorem 12.4 of \cite{BW}. We  implicitly use these constructed operators in this paper.   The idempotents $\tC^\pm= C (\thD^\pm)$ are  called  Calderon projectors associated with our Dirac operators $\thD^\pm$.
\end{remark}

\begin{definition} Given a $\Gamma$-equivariant vector bundle $\tV$ over $\tM$, a $\Gamma$-invariant $0$-th order pseudo differential (idempotent) $R:C^\infty(\tM, \tS^+) \to C^\infty (\tM, \tV)$ will be called  a $\Gamma$-invariant Boundary Value Condition for the $\Gamma$-invariant Dirac operator $\thD^+$ if the principal symbol $r$ of $R$ satisfies the pointwise condition
$$
\Im (r)= \Im (r c^+) \simeq \Im (c^+) \text{ so that } r: \Im (c^+) \rightarrow \Im (r) \text{ is an isomorphism.}
$$
\end{definition}

We fix such $\Gamma$-invariant Boundary Value Condition $R$ for the operator $\thD^+$ and consider the realization $\thD^+_R$ of $\thD^+$ defined as the restriction of $\thD^+_{\max}$ to the domain  associated with $R$:
$$
\Dom (\thD^+_R) := \{ u \in \Dom (\thD^+_{\max}), \tb^+ (u) \in \Ker (R)\}.
$$
Consider on the other hand the following composition of maps that we denote by $\tF_R^+: \Ker \thD^+_{\max} \rightarrow \Range (R)$:
$$
\Ker \thD^+_{\max} \xrightarrow{{\tilde b}^0_+} W^{-1/2}(\tM,\hS^+_{|_{\tM}})\xrightarrow{\Psi_{-1/2,0}}L^2(\tM,\hS^+_{|_{\tM}})\xrightarrow{R} \Range (R)
$$

\begin{proposition}\label{KerCoker}
We have
$$
\Ker \tF_R^+ = \Ker \thD^+_{R}\text{ and } \Coker \tF_R^+ = \Coker \thD^+_{R}.
$$
In particular, the $\tau$ index of the realization operator $\thD^+_{R}$ coincides with the $\tau$-index of $\tF_R^+$.
\end{proposition}

\begin{proof}
We have
$$
\Ker \tF_R^+ =\{ u \in \Ker (\thD^+_{\max}) \text{ such that } F_R^+u=0 \in L^2(\tM, \thS^+_{\tM})\},
$$
while
$$
\Ker \thD^+_{R} =\{ u' \in \Ker (\thD^+_{\max}) \text{ such that } R \Psi^+_{-1/2,0} \tb^0_+u'=0\in  L^2(\tM, \thS^+\vert_{\tM})\}
$$
from which the first claim follows.

 Let $I_{R^*}$  denote the orthogonal projection onto the range of $R^*$.  Then  it is easy to check that the cokernel of the operator $\thD^+_{R}$ is isomorphic to the kernel of the operator $\tilde{\hat{D}}^-_{\tilde{G}(I-I_{R^*})\tilde{G}^{-1}}$, see \cite{BW}.  So, we need to prove that
$\Ker (\tF_R^+)^*\cong  \Ker \tilde{\hat{D}}^-_{\tilde{G}(I-I_{R^*})\tilde{G}^{-1}}$.  We shall use the Calderon projections $\tC^\pm$ of  Proposition \ref{Calderon} above.

From the first item applied to the BVP $\tilde{G}(I-I_{R^*})\tilde{G}^{-1}$ and the operator $\thD^-$ now,  we can write
$$
\Ker \thD^-_{\tG (I-I_{R^*}) \tG^{-1}}  \simeq  \Ker (\tG (I-I_{R^*}) \tG^{-1} \Psi^-_{-1/2,0} \tC^-\Psi^-_{0,-1/2}).
$$
Consider the operator $\tilde{G}(I-I_{R^*})\tilde{G}^{-1}: L^2(\tM,\thS^-\vert_{\tM})\rightarrow L^2(\tM,\thS^-\vert_{\tM})$ then  we have
$$
\Dom(\tilde{\hat{D}}^-_{\tilde{G}(I-I_{R^*})\tilde{G}^{-1}}):= \{ u \in \Dom (\thD^-_{\max}) \text{ such that } \Psi_{-1/2,0}({\tilde b}^0_-u)\in \Ker(\tilde{G}(I-I_{R^*})\tilde{G}^{-1})\}.
$$
So, using again the properties stated in Proposition \ref{Calderon} we can write a list of $\Gamma$-equivariant isomorphisms, distinguishing the isomorphisms $\Psi^+_{s,t}$ and $\Psi^-_{s,t}$ corresponding to the Sobolev spaces associated with the bundles $\thS^+$ and $\thS^-$, respectively, which are conjugates of each other through $\tG$.

\begin{eqnarray*}
\Ker \thD^-_{\tG (I-I_{R^*}) \tG^{-1}} & \simeq & \Ker (\tG (I-I_{R^*}) \tG^{-1} \Psi^-_{-1/2,0} \tC^-\Psi^-_{0,-1/2})\\
&\simeq & \Ker(\tG (I-I_{R^*}) \tG^{-1} \Psi^-_{-1/2,0}\tG(I-\tC^+)\tG^{-1}\Psi^-_{0,-1/2} )\\
&\simeq & \Ker(\tG (I-I_{R^*}) \Psi^+_{-1/2,0}(I-\tC^+)\Psi^+_{0,-1/2} \tG^{-1})\\
&\simeq & \Ker((I-I_{R^*})\Psi^+_{-1/2,0}(I-\tC^+)\Psi^+_{0,-1/2}).
\end{eqnarray*}

But notice that for $u\in \Range \Psi_{-1/2, 0}^+(I-C^+)\Psi_{0, -1/2}^+$,
$$
u\in \Ker\left((I-I_{R^*})\Psi_{-1/2,0}(I-\tC^+)\Psi_{0,-1/2}\right) \Longleftrightarrow I_{R^*} u=u \text{ and } \Psi_{-1/2,0}^+\tC^+\Psi_{0,-1/2}^+ u=0.
$$

Set $\tilde{E}^+:= \Psi^+_{-1/2,0}\tilde{b}^0_+$. Note that $\Range{\tilde{E}^+}\simeq \Range{\Psi_{-1/2,0}^+\tC^+\Psi_{0,-1/2}^+}$.  Hence
\begin{eqnarray*}
\Ker \thD^-_{\tG (I-I_{R^*}) \tG^{-1}} & \simeq & \Ker \{\Psi_{-1/2,0}^+ \tC^+\Psi_{0,-1/2}^+ R^*: \Range(R)\rightarrow \Range(\Psi_{-1/2,0}\tC^+\Psi_{0,-1/2})\}\\
%&\simeq & \Ker \{\Psi_{-1/2,0}(\tC^+)^*\Psi^+_{0,-1/2} R: \Range(R)\rightarrow \Range(\Psi_{-1/2,0}(\tC^+)^*\Psi^+_{0,-1/2})\}\\
&\simeq& \Ker \left( (\tE^+)^*R^*: \Range(R)\rightarrow \Range((\tE^+)^*)\right)\\
&\simeq& \Ker (\tF_R^+)^*.
\end{eqnarray*}
Notice that the isomorphism
\begin{eqnarray*}
&\Ker \left( \{\Psi_{-1/2,0}\tC^+\Psi_{0,-1/2}R^*: \Range(R)\rightarrow \Range(\Psi_{-1/2,0}\tC^+\Psi_{0,-1/2})\}\right) & \simeq  \\
&\Ker\left( \{\Psi_{-1/2,0}(\tC^+)^*\Psi_{0,-1/2}R^*: \Range(R)\rightarrow \Range(\Psi_{-1/2,0}(\tC^+)^*\Psi_{0,-1/2})\}\right) &
\end{eqnarray*}
is due to the fact that  $(\tC^+)^*$ furnishes a $\Gamma$-equivariant isomorphism between $\Range(\Psi_{-1/2,0}\tC^+\Psi_{0,-1/2})$ and $\Range(\Psi_{-1/2,0}(\tC^+)^*\Psi_{0,-1/2})$.
\end{proof}

We end this appendix by pointing out an easy generalization of Theorem 13.8 in \cite{BW} to the $\Gamma$-equivariant setting.

\begin{lemma}\label{W-1/2}
Fix $t\in \R$ and assume that ${\tilde \sigma}\in W^t (\thM, \thS^+)$ satisfies that $\thD^+{\tilde \sigma}\in W^s(\thM, \thS^-)$ for some $s> -1/2$. Then the trace of ${\tilde \sigma}$ at the boundary $\tM$ belongs to $W^{t-1/2} (\tM, \thS^+\vert_{\tM})$.
\end{lemma}

\begin{proof}
We again simply repeat the argument in the proof of Theorem 13.8 in \cite{BW} by adapting it to the $\Gamma$-equivariant calculus. Notice that  this argument relies on  the more precise construction of the Calderon projector of Proposition \ref{Calderon}.
\end{proof}

\end{document}